\documentclass[11pt]{article}
\usepackage{graphicx}
\usepackage[ruled,vlined]{algorithm2e}
\usepackage{amsmath,amssymb}
\usepackage{amsthm}
\usepackage[a4paper,margin=2cm]{geometry}
\usepackage{newtxtext}
\usepackage{setspace}
\usepackage{hyperref}
\usepackage{graphicx}
\usepackage{multirow,booktabs,makecell}
\usepackage{longtable}
\usepackage{tabularx}
\usepackage{pdflscape}
\usepackage[natbibapa]{apacite}
\makeatletter\def\NAT@sort{\z@}\makeatother
\usepackage{xcolor}
\hypersetup{colorlinks=true,linkcolor=blue!60!black,citecolor=blue!60!black,urlcolor=blue!60!black}
\allowdisplaybreaks
\makeatletter\NAT@longnamesfalse\makeatother

\usepackage[]{authblk}
\renewcommand{\Affilfont}{\normalsize}

\usepackage{mathtools}
\usepackage{array}
\usepackage{amsthm}
\newcolumntype{P}[1]{>{\centering\arraybackslash}p{#1}}
\newtheorem{proposition}{Proposition}
\newtheorem{lemma}[proposition]{Lemma}

\theoremstyle{remark}

\title{%
Designing battery swapping and charging networks \\ for electric freight transport under demand uncertainty
}
\date{}
\begin{document}
\setlength{\affilsep}{2em}
\author[*1,2]{Yujia Xu}
\author[3,1]{Walid Klibi}
\author[1,2]{Benoit Montreuil}

\affil[1]{Physical Internet Center, Supply Chain \& Logistics Institute, \protect\\ Georgia Institute of Technology, Atlanta GA, USA}
\affil[2]{\Affilfont School of Industrial \& Systems Engineering, Georgia Institute of Technology, Atlanta GA, USA}
\affil[3]{\Affilfont Centre of Excellence for Supply Chain Innovation \& Transportation (CESIT) \protect\\
Kedge Business School, Bordeaux, France}
\affil[*]{\textit{Corresponding author: yujiaxu.a@gmail.com}}

\maketitle

\begin{abstract}
Long-haul truck electrification requires rapid energy replenishment near freight demand, although dispersed hubs may lack sufficient grid capacity. We study a ``Swap-Locally, Charge-Centrally'' architecture embedded in a relay freight network: swapping remains close to freight flows, charging can be shifted to grid-suitable hubs, and co-location is retained where economical. A multi-period, two-stage stochastic location--inventory--fleet model coordinates station deployment, replenishment assignments and grid-bounded charging capacity with scenario-dependent battery stocks, batch sizes and incremental fleet requirements. Corridor electrification competes with temporary diesel operation. Building on submodularity-based constraint generation for location--inventory systems, we give an alternative scalar partial-minimization proof that the battery value function is submodular, show that the fleet recourse is affine in the binary swap-station decisions, and develop a Lov\'asz-cut Benders decomposition with hub--period block cuts and a convex battery oracle with analytic gradients and memoization. On a seven-state southeastern U.S. freight network, the method reaches gaps below 0.01\% on all five instances, the largest in about four minutes, whereas a strengthened extensive formulation reaches this tolerance only on the two smallest instances within three hours. Stochastic planning avoids the capacity shortfalls of a mean-value plan, and joint planning avoids the battery and fleet costs that sequential siting shifts downstream. The optimized network uses remote charging selectively, keeps most replenishment local, and expands geographic coverage before intensifying capacity within its committed deployment plan.

\end{abstract}

\textbf{Keywords}: Logistics; Stochastic programming; Electric trucks; Benders decomposition; Submodularity

\section{Introduction}

\subsection{Motivation}
Decarbonizing long-haul road freight is an important part of the transition toward low-emission transportation. Heavy-duty trucks are particularly challenging to electrify because they travel long distances, consume large amounts of energy, and operate on demanding schedules with limited time for charging. Nevertheless, the potential environmental benefits are substantial. Under its Net Zero Emissions by 2050 Scenario, the International Energy Agency estimates that electric heavy-duty vehicles, including trucks and buses, could avoid more than 2 billion metric tonnes of cumulative greenhouse-gas emissions between 2025 and 2035 compared with continued reliance on conventional vehicles \citep{IEA2026EVOutlook}. Realizing these benefits will require reliable and time-efficient energy-replenishment infrastructure that can support long-distance freight operations without significantly reducing vehicle availability.

Conventional plug-in charging remains a major barrier to long-haul truck electrification. Heavy-duty trucks carry large batteries and operate under tight schedules, so restoring sufficient range within limited dwell times requires very high charging power. When multiple trucks charge simultaneously, logistics hubs can generate substantial peak electricity demand and may require costly grid connections, transformers, and charging equipment \citep{icct2025charging}. Consequently, widespread deployment of high-power charging may be economically prohibitive, particularly because locations that are convenient for freight operations often lack sufficient grid capacity.

Battery swapping offers an alternative by allowing depleted batteries to be replaced within minutes and recharged separately from truck operations, thereby reducing vehicle downtime and enabling more flexible charging schedules. CATL's QIJI Energy system demonstrates the technical feasibility of this approach through standardized 171-kWh battery blocks that can be combined according to route, load, and operating requirements and exchanged using an automated chassis-level system \citep{CATL2023Qiji}. Swapping alone does not remove the grid problem, however: if every swapping station also charges its own packs, high-power equipment and electrical capacity must still be spread across dispersed sites. Swapping is best located close to freight demand, whereas large-scale charging is better concentrated where grid capacity is available.

Research on electric freight transportation has examined fleet electrification, charging-infrastructure deployment, and the integrated routing and charging of electric trucks \citep{schiffer2021integrated,alp2022transitioning,bragin2024toward}, but it remains centered on plug-in charging and does not represent modular batteries as circulating assets that are exchanged at one hub and recharged at another. Studies of heavy-duty battery swapping address vehicle economics, battery management, and station-level operations \citep{zhu2023does,deng2023battery,wang2024comprehensive} rather than the integrated design of a long-haul swapping, charging, battery-inventory, and fleet system.

\subsection{Problem Statement}
\label{sec:gap} 
\label{sec:statement}
To address this spatial mismatch, we consider a two-tier energy-replenishment architecture for heavy-duty electric trucks based on the ``Swap-Locally, Charge-Centrally'' (SLCC) concept \citep{qi2023scaling}. Battery Swapping Stations (BSSs) at logistics hubs along freight corridors provide rapid battery replacement close to freight demand, while Centralized Charging Stations (CCSs) at grid-suitable hubs perform large-scale charging. Each active BSS is assigned to one CCS, which recharges its depleted batteries and returns them as charged inventory, by intra-facility handling when the two are co-located and by batched inter-hub transport otherwise. At hubs without swapping, trucks recharge their onboard batteries on site, which lengthens dwell and may require additional trucks.

We embed SLCC in a hyperconnected logistics network, whose hub-based structure supports shared freight and energy services and the pooling of resources and modular capacity across hubs \citep{montreuil2011toward,crainic2023hyperconnected,liu2025dynamiccontainerized,xu2026dynamic}. Freight follows predetermined paths through logistics hubs under relay-based truck operations, so trucks already stop at intermediate hubs for freight transfer \citep{liu2025dynamic}. We take this relay structure, which is decided in hub-network design, as an input and study how energy infrastructure should be added to it; because loads, trailers, or drivers are exchanged at these stops whatever energy technology is used, the model does not optimize route deviations toward equipped hubs. The coverage and capacity findings are therefore conditional on the given relay structure.

We interpret the system as an Electrification-as-a-Service (EaaS) platform, in which an operator bundles vehicles, batteries, and charging infrastructure so that carriers adopt electric operations without owning the underlying assets, a model already deployed commercially in fleet electrification \citep{mckinsey2024taas}. We take the perspective of an integrated EaaS operator or stakeholder consortium that minimizes the system cost of electric operation, measured relative to a common freight-service baseline; the allocation of costs among individual stakeholders lies outside the scope of the study. The operator's decisions are coupled through two mechanisms. First, the facility configuration determines whether trucks swap batteries or recharge on site, and hence hub dwell times and fleet requirements. Second, BSS-to-CCS assignments determine replenishment lead times, which shape battery turnaround, stock requirements, batch sizes, and transshipment costs; concentrating charging at grid-suitable hubs reduces infrastructure duplication but lengthens these replenishment cycles.

Uncertainty in long-haul freight demand amplifies this coordination problem. Facility activation, BSS-to-CCS assignments, charging capacity, and corridor electrification levels must be committed before demand is observed, whereas battery stocks, replenishment batch sizes, and fleet capacity depend on the realized demand. Planning for expected demand may leave battery or fleet capacity short under high demand, while planning for peak demand over-invests. We therefore formulate a two-stage stochastic location, inventory, and fleet-planning model that minimizes the expected incremental system cost subject to approximate per-location battery cycle-service targets, offered-load fleet margins, and charging-energy limits in every demand scenario. Corridor electrification is treated as an economic decision rather than a mandate: demand on corridors not yet equipped for electric service is met by diesel operation at a substitution premium, so the model weighs early infrastructure commitment against temporary diesel continuation under each adoption trajectory.

The closest location--inventory foundation is the joint location and repairable-inventory model of \citet{qi2023scaling} for a static urban SLCC network, which takes the swapping stations as given and represents centralized charging through charging-time parameters rather than the explicit hub-level power budgets considered here. In a long-haul freight network, freight-path demand drives both battery requirements at the hubs and truck workload, and the swap-or-charge choice at each hub ties infrastructure to truck dwell and hence to fleet size, which turns the location--inventory problem into a location--inventory--fleet problem. 
\citet{xu2024network} study station deployment, freight consolidation and routing, battery inventory, and transshipment within hyperconnected hub networks.

 The gap addressed here concerns planning BSS and CCS activation, BSS-to-CCS assignment, and grid-bounded charging capacity jointly with circulating battery stocks, replenishment batch sizes, and path-level freight-fleet capacity when demand is unknown at the time of commitment and corridor electrification competes economically with diesel operation,  and evaluating the combined effects of those decisions on long-haul freight infrastructure.
On the methodological side, square-root safety-stock pooling is a classical submodular structure, with extended-polymatroid descriptions of the associated binary epigraphs \citep{atamturk2008polymatroids,atamturk2020submodularity}. More directly, \citet{qi2023scaling} already establish submodularity with optimized replenishment quantities (Section~\ref{sec:submodular-literature}); our approach builds on this foundation. Section~\ref{sec:submodular-recourse} provides a scalar partial-minimization proof for the battery value function, while the computational development combines this structure with the present freight, fleet and grid-capacity model.

\subsection{Contributions}
This paper makes  three main contributions:
\begin{itemize}
    \item We introduce an integrated planning problem for long-haul freight electrification under uncertainty, which coordinates BSS and CCS deployment, BSS-to-CCS assignments, and grid-bounded charging capacity with scenario-dependent battery inventory, replenishment batches, and freight-fleet capacity. We formulate it as a multi-period, two-stage stochastic program in which assignments determine battery turnaround through a repairable-inventory model, the swap-or-charge choice determines incremental fleet requirements through dwell time and Little's law, and corridor electrification is phased against diesel substitution under node-level grid limits.

    \item We develop a Lov\'asz-cut Benders decomposition for the joint energy-infrastructure, battery-inventory and fleet planning problem on a given relay freight network, building on the submodularity-based location--inventory framework of \citet{qi2023scaling}. A scalar partial-minimization argument provides an alternative derivation of the submodularity of the battery value function and supports hub--period block cuts that dominate cuts aggregated over hubs and periods. The fleet recourse has an exact affine representation. A convex battery oracle with certified error bounds and memoization accelerates recourse evaluation. On five nested instances drawn from the seven-state network, the method reaches the $0.01\%$ target gap on all of them, whereas a strengthened extensive formulation reaches it only on the two smallest within the three-hour limit.

    \item  We provide computational evidence and four planning insights. First, mean-value planning provides a useful siting baseline but underestimates capacity needs and misses assignments required at grid-constrained hubs. Second, sequential siting reduces infrastructure costs at the expense of higher battery and fleet costs. Third, SLCC decouples rapid energy service from grid access, keeping swapping close to freight demand while shifting charging selectively to grid-suitable hubs. Requiring every swapping hub to charge locally increases total cost by 14.5\% in the case study. Fourth, the cost-minimizing horizon-committed plan builds geographic coverage first and expands capacity at existing hubs as adoption matures.
\end{itemize}

\section{Literature Review}
\subsection{Electric Freight Charging and Battery Swapping}

Battery-electric truck adoption is constrained by limited range, long charging times, battery durability, high acquisition costs, and uncertain performance across duty cycles \citep{cheng2024electric}, barriers that weigh most on long-haul operations, where charging dwell reduces vehicle utilization. Because operators may invest in private infrastructure when public charging is insufficient \citep{alp2022transitioning}, infrastructure development and adoption must be coordinated over time; on a southeastern U.S.\ network derived from the Freight Analysis Framework, \citet{kim2026structure} couple Bass diffusion with capital-constrained sequencing of hydrogen refueling stations and show that deployment pace and network position jointly shape adoption. Misaligned facilities and delayed corridor-wide deployment restrict the operability of electric freight networks \citep{otero2026battery}, and large-scale electrification may require substantial power-system upgrades at high-demand freight locations \citep{shoman2023battery}. Optimization studies coordinate electric-truck investment and charging deployment with power-system capacity \citep{ye2024joint} and design capacity-constrained public charging networks for long-distance truck traffic \citep{speth2025optimal}, but remain centered on plug-in charging, in which batteries stay attached to trucks.

Battery-swapping research optimizes charging and swapping schedules \citep{sarker2014optimal}, coordinates battery inventory with vehicle routing and real-time operations \citep{ni2020inventory}, and jointly determines battery procurement and charging \citep{schneider2018optimization}; for heavy-duty trucks, studies address degradation-aware battery management and station configuration \citep{deng2023battery} and state-of-charge-dependent charging \citep{wang2024comprehensive}, and show that viability depends strongly on station utilization \citep{zhu2023does}. At the network level, electric commercial-vehicle studies integrate siting with routing, charging, inventory, and capacity decisions \citep{sun2019robust,schiffer2021integrated,bragin2024toward}, and swapping-specific studies jointly locate swapping and fast-charging facilities \citep{zhang2021joint}. These studies plan against a demand trajectory taken as known, and none treats modular batteries as circulating assets whose assignment, inventory, and replenishment arise when swapping is separated from charging; the closest antecedents, \citet{qi2023scaling} and \citet{xu2024network}, are discussed in Section~\ref{sec:gap}.

A separate body of work designs the hub networks in which such services would be embedded. Planning models for hyperconnected logistics networks \citep{montreuil2011toward,crainic2023hyperconnected} size hub capacity in relay transportation \citep{liu2025dynamic}, allocate containerized modular capacity across supply chain ecosystems \citep{liu2025dynamiccontainerized}, and deploy pooled operating resources \citep{xu2026dynamic}. These models dimension hubs for handling and storage and treat the energy that moves freight as exogenous, whereas energy-facility models locate stations without representing the freight-transfer function that brings trucks to them. The problem studied here sits at their junction: it adds energy replenishment to hubs whose freight role, and the truck stops it already induces, are given.

\subsection{Stochastic Location--Inventory Planning and Fleet Sizing}

Stochastic programming distinguishes here-and-now decisions from recourse actions taken after uncertainty is revealed \citep{birge1997introduction,shapiro2021lectures} and has been widely used to coordinate facility location and inventory. \citet{atamturk2012conic} show that joint location--inventory optimization with stochastic demand and service-level constraints, solved by conic integer programming, improves on sequential planning, and \citet{farahani2015location} review how uncertainty in demand, lead times, and disruptions couples facility configuration, inventory, and service. Applications include closed-loop systems with uncertain forward and return flows \citep{hu2025stochastic}. Relative to this literature, our model links infrastructure and assignment decisions with scenario-dependent battery-inventory and replenishment recourse driven by long-haul freight demand. Fleet sizing adds a service-time dimension: Little's law relates the average number of entities in a stable system to their arrival rate and time in system \citep{little1961proof}, \citet{mahavir2025electric} jointly determine the minimum fleet size and charging capacity of electric fleets for a target service level, and \citet{zackrisson2026quantifying} show that intrashift charging and joint reoptimization of fleet, routing, and charging can serve a given freight workload with fewer trucks. Charging time thus determines fleet capacity, and we use Little's law to link the swap-or-charge choice at each hub to the dwell time and fleet requirement of each path.

\subsection{Submodular Structure and Decomposition in Discrete Design Problems}
\label{sec:submodular-literature}
A concave nondecreasing function of a nonnegative modular sum is submodular, and square-root mean-risk terms of this kind are well understood in discrete optimization: the convex hull of their binary epigraphs is described by extended-polymatroid inequalities, a consequence of \citet{lovasz1983submodular}, with conic quadratic extensions to mixed 0--1 sets \citep{atamturk2008polymatroids,atamturk2020submodularity}, and related formulations support the maximization of concave utility functions of modular sums \citep{ahmed2011maximizing}. Such inequalities also serve as cuts for submodular value functions within decomposition methods. In choice-based facility location, \citet{ljubic2018outer} combine outer-approximation and submodular cuts, and \citet{mendezvogel2023follower} develop cut-generation methods for nested-logit choice; \citet{rahmaniani2017benders} survey the broader Benders literature. The closest work is \citet[Section~5.2]{qi2023scaling}, whose evaluation function already optimizes batch quantities inside a pooled square-root safety stock: their Lemma~8 establishes submodularity through a feasible-solution exchange, Lemma~9 gives the Lov\'asz representation used in a constraint-generation algorithm, and a parametric search over a scalar multiplier evaluates the subproblems. We use the same representation. Our setting differs in three respects: the value functions are indexed by scenario and period inside a two-stage stochastic program with grid-bounded charging capacity and phased electrification, and their per-charging-station constraint family becomes one family per scenario, period and charging hub, which dominates cuts aggregated over hubs and periods (Proposition~\ref{prop:cut-valid}); the battery recourse is combined with a fleet recourse that does not depend on the assignments and is represented exactly by flip-difference cuts; and the subproblems are evaluated as convex programs under a stated variance condition. Our scalar partial-minimization argument offers an alternative derivation of submodularity that connects it to minimization over a lattice \citep{topkis1978minimizing}.

\section{Problem Setting and Modeling Framework}
\label{sec:problem}

This section describes the SLCC network operations, the demand model, and the battery, fleet, and charging-capacity relationships that underlie the optimization model of Section~\ref{sec:two-stage}. Notation is summarized in Table~\ref{Tab:description-two-stage}.
Throughout, durations are converted to days and mean demands to counts per day before optimization. Variances are specified for a one-day reference interval and scaled linearly with interval length (the variance-rate approximation of Section~\ref{sec:capacity}). Parameter tables may report durations in hours or minutes for readability.

\subsection{Network, Demand, and Corridor Activation}
\label{sec:network}

We consider an open-access long-haul freight network operated by the EaaS operator of Section~\ref{sec:statement}. Candidate logistics hubs are connected by predetermined freight paths; trucks operate in relay and stop at intermediate hubs for freight transfer, and routing is not optimized. Each hub may host a BSS, a CCS, or both. At a hub with an active BSS, trucks exchange depleted packs for charged ones; otherwise they recharge on board at the hub's CCS, which takes longer and therefore raises dwell times and fleet requirements. Each active BSS is assigned to one active CCS, which recharges its depleted packs and returns them as charged inventory; co-located pairs exchange packs by intra-facility handling at no shipping cost, while remote pairs use batched shuttle transport. Replenishment is modeled through lead times, batch sizes, and transportation costs, separately from path-level freight operations, and both swap replenishment and on-site truck charging draw on the hub's grid-limited charging capacity.

We represent uncertainty in the evolution of long-haul freight demand through a finite scenario set $\Omega$ with probabilities $p_\omega$; each scenario specifies a demand trajectory across freight paths and planning periods. Conditional on scenario $\omega$, the aggregated truck-arrival rate on path $p\in\mathcal P$ in period $t\in\mathcal T$ is approximated by
\begin{equation}
\Theta_{pt}(\omega)
\sim
\mathcal N\!\left(\alpha_{pt}(\omega),\,\beta_{pt}^{2}(\omega)\right),
\label{eq:theta-normal}
\end{equation}
where $\alpha_{pt}(\omega)$ and $\beta_{pt}(\omega)$ denote the scenario-dependent mean and standard deviation. The scenarios capture alternative long-term demand trajectories, whereas the conditional distribution represents day-to-day variability around each trajectory; unconditionally, $\Theta_{pt}$ follows the finite normal mixture $\sum_{\omega}p_\omega\mathcal N(\alpha_{pt}(\omega),\beta^2_{pt}(\omega))$, which accommodates heterogeneous and time-varying demand while retaining the tractability of scenario-based stochastic programming \citep{birge1997introduction,shapiro2021lectures}. Equation~\eqref{eq:theta-normal} is a moment-based approximation for safety sizing; the normal left tail is not interpreted as negative demand.

A path is \emph{active} in period $t$ under scenario $\omega$ if its electric-truck demand reaches a minimum threshold,
\begin{equation}
\mathcal P_t(\omega) \;:=\; \bigl\{p\in\mathcal P : \alpha_{pt}(\omega) \ge d_{\min}\bigr\},
\qquad
\mathcal P_t := \bigcup_{\omega\in\Omega}\mathcal P_t(\omega),
\label{eq:activation-scenario}
\end{equation}
so the timing of corridor activation is scenario-dependent. Serving an active corridor is an economic decision rather than a mandate: the electrification level $w_{pt}\in[0,1]$ of path $p$ is limited by station availability along the path, and demand not served electrically is met by diesel trucks at a substitution premium $c^{\mathrm{D}}$ per truck-day; rigid full coverage is recovered as the limiting case $c^{\mathrm{D}}\to\infty$.

Truck arrivals generate energy-replenishment requirements at the hubs visited by each path. Let $R_{ph}(\omega)$ denote the number of battery packs required per truck arriving on path $p$ at hub $h$ under scenario $\omega$, and let $P(h)$ denote the set of paths visiting hub $h$. In our experiments, $R_{ph}$ follows a just-in-time policy (at each hub a truck takes the whole packs needed for the next leg, capped by the onboard pack limit) and is identical across scenarios; the scenario argument is kept for generality. Assuming that path-level arrivals are conditionally independent given the scenario, the mean and variance of potential battery-pack demand at hub $h$ in period $t$ are
\begin{equation}
\mu_{ht}(\omega)
=
\sum_{p\in P(h)\cap\mathcal P_t(\omega)}
\alpha_{pt}(\omega)R_{ph}(\omega),
\qquad
\sigma_{ht}^{2}(\omega)
=
\sum_{p\in P(h)\cap\mathcal P_t(\omega)}
\beta_{pt}^{2}(\omega)R_{ph}^{2}(\omega),
\label{eq:demand-moments-h}
\end{equation}
and we write $\sigma_{ht}(\omega):=\sqrt{\sigma_{ht}^2(\omega)}$. When a BSS is active at hub $h$, this demand enters the circulating-battery system as swap requests; otherwise it is served by on-site charging and generates no demand for swappable batteries. Battery and fleet resources are sized on potential electric demand: $w_{pt}$ prices the diesel-served share in the master but does not scale the sizing coefficients, so a corridor with $w_{pt}=0$ can still contribute battery demand at an equipped hub and fleet cost at the hubs it visits. This fixed-demand approximation preserves the set-function structure used in Section~\ref{sec:submodular-recourse}, but it does not represent exactly the resources used by the electrically served flow; Section~\ref{sec:planning-insights} reports its magnitude in the computed plans.

Conditional path independence supports the hub-level moments~\eqref{eq:demand-moments-h}. Pooling at a charging hub requires an additional approximation: swap-demand streams at the assigned hubs are treated as independent over their respective replenishment windows. This does not follow from path independence, because a path that visits two hubs assigned to the same charging hub makes their demands dependent. The omitted covariances are nonnegative, so the approximation understates the pooled variance whenever two hubs of a pool share a path and is not conservative by construction. Online Supplement~S7 derives these covariances under the variance-rate approximation and bounds them for any temporal dependence along a path, and Appendix~\ref{app:pool-sim} quantifies their effect for the SE-5 plan.

\subsection{Capacity Planning}
\label{sec:capacity}

We now use the demand model to size three coupled capacities: battery stocks at charging hubs, the incremental truck fleet, and grid-limited charging capacity at the hubs that serve both loads. All three depend on the network configuration and on the charging lead times and truck dwell times it induces.

\subsubsection{Battery stock sizing under an \texorpdfstring{$(r,Q)$}{(r,Q)} replenishment policy}

\paragraph{Service-level requirement.}
We approximate a charging cycle-service target $1-\varepsilon^C$, with $0<\varepsilon^C<1/2$, by sizing the circulating battery pool through normal-quantile safety stocks. The target concerns a replenishment cycle at an assigned swapping hub, conditional on the arrival, batching and pooling approximations specified below. This is the \emph{per-cycle, per-location} non-stockout measure conventional in the location--inventory literature \citep{shen2003joint,atamturk2012conic,qi2023scaling}, not a joint guarantee over the network and horizon, which would require a Bonferroni-type tightening of $\varepsilon^C$ or a joint chance constraint. The reference service parameter is $\varepsilon^C=0.05$, corresponding to $\Phi^{-1}(1-\varepsilon^C)\simeq1.645$. The pool $n^b_{t,h'}(\omega)$ counts all packs circulating through charging hub $h'$, not only those stacked at $h'$, and must satisfy
\begin{equation}
n_{t,h'}^b(\omega) \;\ge\; \Psi_{t,h'}\bigl(z,Q(\omega);\mu(\omega),\sigma(\omega)\bigr),
\qquad \forall t \in \mathcal T,\ \forall h'\in\mathcal H,
\label{eq:Battery-service}
\end{equation}
where $z_{hh't}=1$ only when hub $h$ operates as a swapping hub in period $t$ and is replenished by $h'$, and $Q(\omega)$ collects the batch sizes. The function $\Psi_{t,h'}$ combines the expected number of packs in charging and transit with a safety stock against arrival and batching variability.

\paragraph{Repairable-inventory approximation.}
To construct $\Psi_{t,h'}$, we adapt the repairable-inventory approximation of \citet{qi2023scaling} to the multi-period, scenario-based network setting. Each swapping hub $h$ operates an $(r,Q)$ policy for charged batteries: whenever its inventory position drops to $r$, it ships $Q_{hh't}(\omega)$ depleted packs to its assigned charging hub $h'$ and orders the same number of charged packs. Only the batch size is a model decision; for a given pool size, the smallest $r$ attaining the target non-stockout probability is pinned by the service requirement, so $r$ remains implicit. Shuttle trips take the deterministic one-way time $T^t_{hh'}$, each pack requires the deterministic charging time $T^c$ before it can be returned, and partially charged packs are not swapped to customers. The \emph{effective charging time} for a pack moved from $h$ to $h'$ is
\begin{equation}
\Delta_{hh'} \;=\; T^{c} + T^{t}_{hh'},
\label{eq:Delta}
\end{equation}
the delay from a reorder at the swapping hub until the corresponding recharged pack is ready for dispatch at the charging hub.
The full pipeline window $T^c+2T^t_{hh'}$ covers charging and both transport legs. The pooled term uses $\Delta_{hh'}=T^c+T^t_{hh'}$, while the third line of~\eqref{eq:Battery2} protects the return leg locally. For a single independent stream, splitting a variance buffer is conservative relative to a single round-trip buffer because $\sqrt{(T^c+T^t)\sigma^2}+\sqrt{T^t}\sigma\ge\sqrt{(T^c+2T^t)\sigma^2}$. This inequality alone does not establish the cycle-service probability of the composed, multi-hub inventory system; the simulation in Appendix~\ref{app:pool-sim} shows that the split meets the target under day-to-day rate variability, but not at remote hubs when variability also accrues within the day, because both buffers must then suffice.

Heavy-duty packs weigh several hundred kilograms, so a shuttle carries few packs per trip, whereas the expected number of swaps $\mu_{ht}(\omega)\Delta_{hh'}$ over one effective charging interval is typically large. We exploit this scale separation by capping the batch size at a fraction $\gamma>0$ of that expected demand, together with the shuttle capacity $\overline Q$:
\begin{equation}
z_{hh't}\le Q_{hh't}(\omega)\le
\overline Q_{hh't}(\omega)z_{hh't},
\quad
\overline Q_{hh't}(\omega)
:=\min\!\left\{\overline Q,\max\!\left\{1,
\gamma\mu_{ht}(\omega)\Delta_{hh'}\right\}\right\},
\label{eq:small-batch-bound}
\end{equation}
for all $h'\in\mathcal H$, $h\in I_{h'}$, $t\in\mathcal T$, and $\omega\in\Omega$, so that $Q=0$ on an inactive link and $Q\ge1$ on an active one.

The batching variance follows from a counting argument adapted from \citet{qi2023scaling} (Online Supplement~S2). Let $X_{hh't}(\omega)$ be the number of full batches shipped from $h$ to $h'$ during the most recent effective charging interval, so that the \emph{batched deficit} $Q_{hh't}(\omega)X_{hh't}(\omega)$ counts the packs removed from the stock at $h$ and tied up in full batches in transit or charging. For integer batch sizes, it differs from the demand over that interval, whose mean is $\mu_{ht}(\omega)\Delta_{hh'}$, only through the incomplete-batch remainders at the two endpoints, each approximately uniform on $\{0,\ldots,Q_{hh't}(\omega)-1\}$. Treating the remainders as independent of each other and of the interval demand, scaling the one-day demand variance linearly with interval length (a variance-rate approximation), and allowing $Q$ to be continuous gives the smooth planning approximation
\begin{equation}
  \mathrm{Var}\bigl(Q_{hh't}(\omega)X_{hh't}(\omega)\bigr)
    \;\approx\; \Delta_{hh'}\,\sigma_{ht}^2(\omega) + \frac{Q_{hh't}^2(\omega)-1}{6},
  \label{eq:VarQX-scenario}
\end{equation}
whose first term reflects arrival variability over the interval and whose second reflects batching. If $z_{hh't}=1$ and $\gamma\mu_{ht}(\omega)\Delta_{hh'}<1$, constraint~\eqref{eq:small-batch-bound} fixes $Q_{hh't}(\omega)=1$; the remainder then vanishes, and only the variance-rate approximation remains.
The small-batch bound~\eqref{eq:small-batch-bound} is a modeling-consistency restriction: it confines the model to the region where~\eqref{eq:VarQX-scenario} is reliable. We calibrate $\gamma$ by simulation, following the protocol of \citet{qi2023scaling}, so that within the admissible range the approximation errs toward additional safety stock (Online Supplement~S2). The demand-variance input $\Delta_{hh'}\sigma^2_{ht}(\omega)$, which carries day-to-day rate variability, adds a further cushion wherever it exceeds the Poisson count variance $\sum_{p}\alpha_{pt}(\omega)R_{ph}^2(\omega)$ (Section~\ref{sec:scenarios}).

Combining these elements, the battery requirement is
\begin{equation} \label{eq:Battery2}
\begin{aligned}
\Psi_{t,h'}(z,Q(\omega);\mu(\omega),\sigma(\omega))
:= &\;
    \sum_{h \in I_{h'}}
        z_{hh't}\bigl((T^c + 2T^{t}_{hh'})\mu_{ht}(\omega)
                      - 1 + Q_{hh't}(\omega)\bigr) \\
    &\;+\;
    \Phi^{-1}(1-\varepsilon^C)
      \sqrt{\sum_{h \in I_{h'}}
          z_{hh't}\Bigl((T^c + T^{t}_{hh'})\sigma_{ht}^2(\omega)
                         + \frac{Q_{hh't}^2(\omega)-1}{6}\Bigr)} \\
    &\;+\;
    \Phi^{-1}(1-\varepsilon^C)
      \sum_{h \in I_{h'}} \sqrt{T^{t}_{hh'}}\,z_{hh't}\,\sigma_{ht}(\omega),
      \qquad \forall t \in \mathcal T,\ \forall h'\in\mathcal H.
\end{aligned}
\end{equation}
The first line of~\eqref{eq:Battery2} is the expected pipeline: packs in charging at $h'$ and in transit in either direction, plus the batching allowance $Q_{hh't}(\omega)-1$. The second line buffers arrival and batching variability~\eqref{eq:VarQX-scenario} and is \emph{pooled} across the assigned swap streams, because packs recharge in a common queue at $h'$ and are fungible across the swapping hubs it serves. The third line is an unpooled per-hub buffer over the one-way transport time, because a truck at hub $h$ cannot be served from stock at another swapping hub; accordingly, the model permits no lateral transshipment between swapping hubs, and the guarantee concerns replenishment cycles rather than the instantaneous queue a truck may encounter.

\paragraph{Sign of the battery-stock requirement.}
\label{sec:stock-sign-check}
For an active link, \eqref{eq:small-batch-bound} gives $Q_{hh't}(\omega)\ge1$, so the pipeline term $(T^c+2T^t_{hh'})\mu_{ht}(\omega)-1+Q_{hh't}(\omega)$ and the batching term $(Q^2_{hh't}(\omega)-1)/6$ are nonnegative. Since $\varepsilon^C<1/2$, $\Phi^{-1}(1-\varepsilon^C)>0$, so both buffer terms in~\eqref{eq:Battery2} are nonnegative as well. Hence $\Psi_{t,h'}\ge0$ on the whole feasible domain, the inventory constraint binds at optimality with $n^b_{t,h'}(\omega)=\Psi_{t,h'}$, and the reduced value function~\eqref{eq:f-reduced} needs no positive-part operator.

\subsubsection{Fleet sizing via dwell time and Little's law}
We model the fleet as a steady-state service system and apply Little's law. Each freight job occupies a truck for line-haul travel along its path and dwell at the hubs it visits. For path $p\in\mathcal P$, let $\mathcal H_p\subseteq\mathcal H$ be the hubs it visits and $T^t_p$ its driving time; at hub $h\in\mathcal H_p$, $T^r_{hp}$ is the dwell floor, namely the loading and unloading window at origin and destination hubs and a short operational stop at intermediate hubs. Energy replenishment takes $T^s$ per swap or $T^{c,\text{fast}}$ per pack fast-charged on board a waiting truck; $T^{c,\text{fast}}$ is distinct from the replenishment pace $T^c$ of swap batteries (Section~\ref{sec:grid-capacity}). With $x^s_{ht}$ indicating an open swapping station at hub $h$ in period $t$, the effective dwell time of a truck on path $p$ at hub $h$, and the incremental dwell that energy replenishment adds beyond the floor, are given for all $p\in\mathcal P$, $h\in\mathcal H_p$, and $t\in\mathcal T$ by
\begin{alignat}{2}
r_{pht}(\omega)
   &\;=\; \max\!\bigl\{T_{hp}^r,\,
      T^s x_{ht}^s + T^{c,\text{fast}} \bigl(1-x_{ht}^s\bigr) R_{ph}(\omega)\bigr\},
\label{eq:dwell-time}\\
\tilde r_{pht}(\omega)
   &\;=\; r_{pht}(\omega) - T^r_{hp}
   \;=\; \max\!\bigl\{0,\;
      T^s x_{ht}^s + T^{c,\text{fast}} \bigl(1-x_{ht}^s\bigr) R_{ph}(\omega)
      - T^r_{hp}\bigr\}.
\label{eq:dwell-excess}
\end{alignat}
The floor $T^r_{hp}$ and the line-haul time $T^t_p$ are incurred under every energy-service configuration and under diesel operation alike, so the fleet they occupy is a constant that we omit and the fleet cost is charged on the incremental dwell $\tilde r_{pht}(\omega)$. A charging session that fits within the loading window at an origin or destination hub thus ties up no additional truck-time, whereas an energy stop at an intermediate hub counts almost in full.

By Little's law, the average number of busy trucks on path $p$ for a fixed arrival rate $\Theta_{pt}(\omega)$ is
\[
\Theta_{pt}(\omega)\Bigl(\sum_{h\in\mathcal H_p}r_{pht}(\omega)+T^t_p\Bigr).
\]
Because this is a relation between means and the rate is itself uncertain within the period, we size the fleet by offered load with a safety factor, provisioning for the $(1-\varepsilon^C)$-quantile of the arrival rate, $\alpha_{pt}(\omega)+\Phi^{-1}(1-\varepsilon^C)\beta_{pt}(\omega)$; this is a capacity margin on the offered load, not a guarantee on the instantaneous queue. The quantile is shared with the battery target for parsimony only: here it scales provisioned truck capacity rather than setting a stockout probability, and the two parameters could be set independently. The relation presumes stationarity within a period and ignores queueing at the swapping or charging interface. With $O_c$ the fraction of the day a truck is operated, the incremental fleet provisioned for a scenario-active path $p\in\mathcal P_t(\omega)$ is
\begin{equation}
n^v_{t,p}(\omega)=\frac{\alpha_{pt}(\omega)+\Phi^{-1}(1-\varepsilon^C)\beta_{pt}(\omega)}{O_c}\sum_{h\in\mathcal H_p}\tilde r_{pht}(\omega):
\label{eq:fleet-closed}
\end{equation}
the safety-adjusted offered load (trucks per day) times the incremental dwell (days) gives busy trucks, which are converted into provisioned trucks and priced at the annual cost $M^v_t$ of a driver--truck unit. Online Supplement~S7 states the equivalent linear constraints. Because $x^s_{ht}$ is binary, this block is exactly affine in $x^s$ (Proposition~\ref{prop:fleet-affine}), which the decomposition of Section~\ref{sec:Methodology} exploits.

\subsubsection{Charging capacity and grid limits}
\label{sec:grid-capacity}

Both requirements draw on charging capacity, which is scarce where grid access is limited, so the charging power $C_{ht}\ge 0$ (kW) installed at each hub is a first-stage decision. The same high-power equipment serves two loads in different modes: a through truck occupies a charger at full power (one pack in $T^{c,\text{fast}}$), whereas returned swap batteries are inventory whose charging is staggered across the pool, so each recharges at the slower effective pace $T^c$ while no one waits for it.

The on-site load depends on how much of a corridor's demand is served electrically. With $w_{pt}$ and $\mathcal P_t:=\bigcup_\omega\mathcal P_t(\omega)$ as in Section~\ref{sec:network}, let $v_{pht}\in[0,1]$ be the electrified share of path $p$ in period $t$ that charges on site at hub $h\in\mathcal H_p$ because $h$ offers no swapping:
\begin{equation}
v_{pht} \;\ge\; w_{pt} - x^s_{ht},
\qquad v_{pht}\;\ge\;0,
\qquad
\forall p\in\mathcal P_t,\ \forall h\in\mathcal H_p,\ \forall t\in\mathcal T;
\label{eq:onsite-indicator}
\end{equation}
since decreasing $v$ cannot worsen the objective or feasibility, an optimal solution can always be chosen with $v$ at its lower bound. Each kW of capacity delivers $\rho\zeta$ kWh per day, where $\zeta:=24O_c$ is the daily operating window in hours and $\rho\in(0,1)$ is the maximum average utilization, and each pack carries $E^{\mathrm{pack}}$ kWh. Requiring the daily charging energy to fit within capacity in every scenario,
\begin{equation}
\begin{aligned}
&\underbrace{\sum_{h\in I_{h'}}
 z_{hh't}\mu_{ht}(\omega)E^{\mathrm{pack}}}_{\text{swap-replenishment energy}}
+\underbrace{\sum_{p\in P(h')\cap\mathcal P_t(\omega)}
 \alpha_{pt}(\omega)R_{ph'}(\omega)E^{\mathrm{pack}}v_{ph't}}_{\text{on-site charging energy}}
\\[-1mm]
&\hspace{35mm}\le \rho\zeta C_{h't},
\qquad
\forall h'\in\mathcal H,\ \forall t\in\mathcal T,\ \forall\omega\in\Omega,
\end{aligned}
\label{eq:charger-throughput}
\end{equation}
ensures that installed capacity covers the average daily charging energy assigned to each hub in every training scenario, with utilization margin $1-\rho$.  The on-site term scales with the electrically served share, whereas the swap term retains full potential demand at assigned BSSs. This asymmetry is consistent with the fixed-demand approximation of Section~\ref{sec:network}, but persists even for binary $w$ whenever a deployed BSS is visited by a corridor with $w=0$. Equality with a served-flow energy balance requires all contributing corridors to be fully electrified. The constraint supplies an average energy-feasibility condition. It does not, by itself, validate the deterministic charging times in~\eqref{eq:Delta} and~\eqref{eq:dwell-time}; intraday peaks and charger queueing remain outside the approximation.
The grid bounds the power drawn at a node:
\begin{equation}
C_{ht} \;\le\; \bar P_{ht}\, x^c_{ht},
\qquad
\forall h\in\mathcal H,\ \forall t\in\mathcal T,
\label{eq:grid-cap}
\end{equation}
where $x^c_{ht}\in\{0,1\}$ indicates an open charging station and $\bar P_{ht}$ is the power available for truck charging at node $h$ in period $t$, set by the node's transmission voltage level and growing with transmission-and-distribution expansion. Both loads compete for the same grid budget: fast-charging a through truck consumes grid capacity at the point of need, whereas a swap shifts the same energy to schedulable replenishment, possibly at a grid-rich hub via the assignment $z$, so the network concentrates charging \emph{where the grid allows}. Capacity is priced at $c^{\mathrm{kW}}$ per kW-year. Because~\eqref{eq:charger-throughput} is enforced in every scenario while the objective minimizes expected cost, first-stage capacity is \emph{scenario-feasible}, a finite-scenario design requirement related to sample-robust planning \citep{bertsimas2022twostage}; Section~\ref{sec:saa-stability-main} evaluates capacity exceedances empirically. Hard enforcement is also structural: the deterministic charging times in~\eqref{eq:Delta} and~\eqref{eq:dwell-time} presuppose average throughput within capacity, so a priced violation would invalidate the service-time model on which the recourse rests. For general scenario data the requirement is $\max_{\omega\in\Omega}A_{h't}^{\omega}(z,v)\le\rho\zeta C_{h't}$, where $A_{h't}^{\omega}$ denotes the left-hand side of~\eqref{eq:charger-throughput}, and the maximizing scenario can depend on the design. In the experimental generator, $R_{ph}$ is scenario-independent and all path rates in period $t$ share the multiplier $(1+g_\omega)^{y_t-2022}\varphi_t(\omega)$, so one scenario dominates every nonnegative load coefficient of a period, including the threshold-activated loads; this reduction need not hold under spatially heterogeneous adoption.

\section{Two-Stage Stochastic Formulation}
\label{sec:two-stage}

\begin{table}[!t]
\centering
\caption{Notation. A symbol in parentheses is the BSS counterpart of the preceding CCS symbol; $(\omega)$ marks scenario-dependent parameters.}
\label{Tab:description-two-stage}
\footnotesize
\renewcommand{\arraystretch}{1.00}
\setlength{\tabcolsep}{3pt}
\begin{tabularx}{\textwidth}{@{}l>{\raggedright\arraybackslash}X@{\hspace{1.2em}}l>{\raggedright\arraybackslash}X@{}}
\toprule
\multicolumn{4}{@{}l}{\textit{Sets}} \\[1pt]
$\mathcal H$ & Candidate hubs for charging and swapping & $\mathcal T$ & Planning periods \\
$\mathcal P$ & Freight paths & $\Omega$ & Demand scenarios \\
$\mathcal H_p$ & Hubs visited by path $p$ & $P(h)$ & Paths visiting hub $h$ \\
$I_{h'}$ & Swapping hubs that CCS $h'$ can replenish & $\mathcal P_t(\omega)$ & Active paths in $t$; $\mathcal P_t=\bigcup_\omega\mathcal P_t(\omega)$ \\
\midrule
\multicolumn{4}{@{}l}{\textit{Cost parameters}} \\[1pt]
$\delta_t$ & Discount factor of period $t$ &  &  \\
$A^c$ ($A^s$) & Annualized fixed cost of a CCS (BSS) & $M^c_t$ ($M^s_t$) & Maintenance cost of an open CCS (BSS) \\
$U^c_t$ ($U^s_t$) & Initialization cost of a new CCS (BSS) & $c^{\mathrm{kW}}$ & Annualized capacity cost per kW \\
$M^v_t$ & Annual cost per truck (vehicle and driver) & $M^b_t$ & Annual cost per battery pack \\
$\eta_t$ & Shuttle cost per driving-day & $c^{\mathrm{D}}$ & Diesel-substitution premium per truck-day \\
\midrule
\multicolumn{4}{@{}l}{\textit{Time and operating parameters}} \\[1pt]
$T^s$ & Swap time per stop & $T^c$ & Charging time per returned pack \\
$T^{c,\text{fast}}$ & On-site fast-charging time per pack & $T^r_{hp}$ & Dwell floor of path $p$ at hub $h$ \\
$T^t_{hh'}$ & Driving time between hubs $h$ and $h'$ & $T^t_p$ & Driving time along path $p$ \\
$Y$ & Working days per year & $O_c$ & Operated fraction of the day; $\zeta=24O_c$ \\
\midrule
\multicolumn{4}{@{}l}{\textit{Demand and service parameters}} \\[1pt]
$p_\omega$ & Probability of scenario $\omega$ & $d_{\min}$ & Activation threshold (trucks per day) \\
$\alpha_{pt}(\omega)$ & Mean daily truck arrivals on path $p$ & $\beta_{pt}(\omega)$ & Standard deviation of arrivals on path $p$ \\
$\mu_{ht}(\omega)$ & Mean daily pack demand at hub $h$ & $\sigma_{ht}(\omega)$ & Standard deviation of pack demand at $h$ \\
$R_{ph}(\omega)$ & Packs replenished per truck of $p$ at $h$ & $\varepsilon^C$ & Service level $1-\varepsilon^C$; $\Phi^{-1}$: normal quantile \\
$\Psi_{t,h'}(\cdot)$ & Battery-stock requirement at CCS $h'$, \eqref{eq:Battery2} &  &  \\
\midrule
\multicolumn{4}{@{}l}{\textit{Capacity and grid parameters}} \\[1pt]
$\overline Q$ & Maximum packs per shuttle load & $\rho$ & Maximum average charger utilization \\
$E^{\mathrm{pack}}$ & Energy per pack (kWh) & $\bar C_0$ & Capacity included in a base CCS \\
$\bar P_{ht}$ & Grid power for charging at $h$ in $t$ &  &  \\
\midrule
\multicolumn{4}{@{}l}{\textit{First-stage decision variables}} \\[1pt]
$x^c_{ht}$ ($x^s_{ht}$) & 1 if a CCS (BSS) operates at $h$ in $t$ & $u^c_{ht}$ ($u^s_{ht}$) & 1 if a CCS (BSS) opens at $h$ in $t$ \\
$z_{hh't}$ & 1 if BSS $h$ is assigned to CCS $h'$ in $t$ & $C_{ht}$ & Charging capacity (kW) at $h$ in $t$ \\
$w_{pt}$ & Electrified share of path $p$ in $t$ & $v_{pht}$ & On-site charging share of $p$ at $h$ \\
\midrule
\multicolumn{4}{@{}l}{\textit{Second-stage decision variables (all scenario-dependent)}} \\[1pt]
$n^b_{t,h'}$ & Circulating packs assigned to CCS $h'$ & $n^v_{t,p}$ & Incremental fleet of path $p$ \\
$Q_{hh't}$ & Replenishment batch size between BSS $h$ and CCS $h'$ & $g_{hh't}$ & Perspective variable ($=1/Q_{hh't}$ if active) \\
$r_{pht}$ & Effective dwell, determined by $x^s$ through~\eqref{eq:dwell-time} & $\tilde r_{pht}$ & Dwell in excess of $T^r_{hp}$, \eqref{eq:dwell-excess} \\
$o_{pht}$ & Incremental busy trucks of $p$ at $h$ &  &  \\
\bottomrule
\end{tabularx}
\end{table}

We formulate the planning problem as a two-stage stochastic program over the scenario set $\Omega$. The first stage commits station activation, BSS-to-CCS assignments, charging capacity, and corridor electrification before demand is observed; the second stage sets battery stocks, batch sizes, and fleet capacity for each realized scenario. Battery stocks and fleet capacity are contracted one year at a time. Because the recourse separates by period and the period-$t$ recourse depends only on the period-$t$ demand moments of the scenario, these decisions coincide with those taken at the start of each year once that year's demand is known; they do not use knowledge of later years. The first stage, in contrast, is committed for the whole horizon and does not adapt to observed adoption. Table~\ref{Tab:description-two-stage} summarizes the notation.

The first-stage variables are $x^c_{ht},x^s_{ht},u^c_{ht},u^s_{ht},z_{hh't}\in\{0,1\}$, $w_{pt},v_{pht}\in[0,1]$, and $C_{ht}\ge0$; $z$ replaces a stand-alone assignment variable, so inactive swapping hubs are not assigned. We assume every candidate swapping hub has at least one feasible replenishment hub, $\{h'\in\mathcal H: h\in I_{h'}\}\neq\emptyset$ for all $h\in\mathcal H$. For scenario $\omega$, let $G_{t,h'}^{B}(z;\omega)$ denote the optimal battery and replenishment cost at charging hub $h'$ in period $t$ and $G_{t,p}^{\mathrm V}(x^s;\omega)$ the fleet cost of path $p$; the second-stage cost of a first-stage solution is
\begin{equation}
\mathcal{G}(x^s,z;\omega)
=
\sum_{t\in\mathcal T}\sum_{h'\in\mathcal H}
G_{t,h'}^{B}(z;\omega)
+\sum_{t\in\mathcal T}\sum_{p\in\mathcal P_t(\omega)} G_{t,p}^{\mathrm V}(x^s;\omega),
\label{eq:recourse-decomp}
\end{equation}
which separates by period, by charging hub, and by path; the fleet block is evaluated for scenario-active corridors, including their diesel-served shares under the fixed-demand approximation. The recourse is feasible for every first-stage solution, so only optimality cuts are needed in Section~\ref{sec:Methodology}. The charging-capacity and grid constraints of Section~\ref{sec:grid-capacity} enter the first-stage problem below unchanged.

\paragraph{First-stage problem.}
\begin{equation}
\min_{x^c,x^s,u^c,u^s,z,w,v,C}\quad
\begin{aligned}
&\sum_{h\in\mathcal H}\sum_{t\in\mathcal T}\delta_t\Bigl((A^c+M_t^c)x^c_{ht}+U_t^c u^c_{ht}
+c^{\mathrm{kW}}(C_{ht}-\bar C_0x^c_{ht})\Bigr)\\
&+\sum_{h\in\mathcal H}\sum_{t\in\mathcal T}\delta_t\Bigl((A^s+M_t^s)x^s_{ht}+U_t^s u^s_{ht}\Bigr)\\
&+\sum_{\omega\in\Omega}p_\omega\sum_{t\in\mathcal T}\sum_{p\in\mathcal P_t(\omega)}
\delta_tc^{\mathrm{D}} Y\alpha_{pt}(\omega)(1-w_{pt})\\
&+\sum_{\omega\in\Omega}p_\omega\Biggl(
\sum_{t\in\mathcal T}\sum_{h'\in\mathcal H}G^B_{t,h'}(z;\omega)
+\sum_{t\in\mathcal T}\sum_{p\in\mathcal P_t(\omega)}G^{\mathrm V}_{t,p}(x^s;\omega)
\Biggr)
\end{aligned}
\label{eq:first-stage-obj-decomp}
\end{equation}
\begin{alignat}{2}
\text{s.t.}\quad
&u^c_{ht}\ge x^c_{ht}-x^c_{h,t-1},
&&\quad \forall (h,t)\in\mathcal H\times(\mathcal T\setminus\{0\}),
\label{eq:design-1-split}\\
&u^c_{h0}=x^c_{h0},
&&\quad \forall h\in\mathcal H,
\label{eq:design-2-split}\\
&u^s_{ht}\ge x^s_{ht}-x^s_{h,t-1},
&&\quad \forall (h,t)\in\mathcal H\times(\mathcal T\setminus\{0\}),
\label{eq:design-3-split}\\
&u^s_{h0}=x^s_{h0},
&&\quad \forall h\in\mathcal H,
\label{eq:design-4-split}\\
&w_{pt}\le x^c_{ht}+x^s_{ht},
&&\quad \forall t\in\mathcal T,\ \forall p\in\mathcal P_t,\ \forall h\in\mathcal H_p,
\label{eq:design-5-split}\\
&\sum_{h':\,h\in I_{h'}}z_{hh't}=x^s_{ht},
&&\quad \forall (h,t)\in\mathcal H\times\mathcal T,
\label{eq:assign-1-split}\\
&z_{hh't}\le x^c_{h't},
&&\quad \forall h'\in\mathcal H,\ \forall h\in I_{h'},\ \forall t\in\mathcal T,
\label{eq:assign-2-split}\\
&z_{hh't}=0,
&&\quad \forall h'\in\mathcal H,\ \forall h\notin I_{h'},\ \forall t\in\mathcal T,
\label{eq:assign-feasible-split}\\
&\eqref{eq:onsite-indicator},\quad
\eqref{eq:charger-throughput},\quad
\eqref{eq:grid-cap},
\notag\\[-1mm]
&C_{ht}\ge \bar C_0x^c_{ht},
&&\quad \forall (h,t)\in\mathcal H\times\mathcal T,
\label{eq:base-capacity}\\
&x^c_{ht},x^s_{ht},u^c_{ht},u^s_{ht}\in\{0,1\},
&&\quad \forall h\in\mathcal H,\ \forall t\in\mathcal T,
\notag\\
&z_{hh't}\in\{0,1\},
&&\quad \forall h,h'\in\mathcal H,\ \forall t\in\mathcal T,
\notag\\
&w_{pt}\in[0,1],
&&\quad \forall t\in\mathcal T,\ \forall p\in\mathcal P_t,
\notag\\
&v_{pht}\in[0,1],
&&\quad \forall t\in\mathcal T,\ \forall p\in\mathcal P_t,\ \forall h\in\mathcal H_p,
\notag\\
&C_{ht}\ge0,
&&\quad \forall h\in\mathcal H,\ \forall t\in\mathcal T.
\label{eq:binary-split}
\end{alignat}
The third objective term prices unmet electrification: the non-electrified share $(1-w_{pt})$ of an active corridor's demand incurs the premium $c^{\mathrm{D}}$ over $Y$ operating days, discounted by $\delta_t$. Constraints~\eqref{eq:design-1-split}--\eqref{eq:design-4-split} track station openings and trigger initialization costs only when a facility is commissioned. Because infrastructure costs are annualized, as for leased or contracted equipment, the formulation does not forbid closing a station or reducing capacity between periods; irreversible investment would add the monotonicity constraints $x_{ht}\ge x_{h,t-1}$ and $C_{ht}\ge C_{h,t-1}$, which involve only first-stage variables and leave the decomposition of Section~\ref{sec:Methodology} unchanged. Section~\ref{sec:planning-insights} reports their effect on the SE-5 plan.
Constraint~\eqref{eq:design-5-split} is the corridor bottleneck: the electrification level of path $p$ cannot exceed station availability at any hub it visits, including its destination, where an arriving truck must be able to recharge before its next dispatch, which the path-based model does not represent explicitly. For a fixed station configuration, $w_{pt}$ would attain its integral station-availability bound unless raising it incurs capacity cost or affects feasibility; with endogenous charging capacity, fractional values can arise at throughput limits or economic breakpoints, so no integrality property is claimed for $w$, and a binding throughput row need not imply that the physical grid cap is saturated. Constraints~\eqref{eq:assign-1-split}--\eqref{eq:assign-feasible-split} assign each active swapping hub to exactly one open, feasible charging hub, and~\eqref{eq:base-capacity} credits the capacity $\bar C_0$ included in the reference site, so that a base site costs exactly $A^c$. The capacity constraints~\eqref{eq:onsite-indicator}, \eqref{eq:charger-throughput}, \eqref{eq:grid-cap} and~\eqref{eq:base-capacity} couple only first-stage variables with fixed scenario data, so they restrict the master without affecting the recourse structure or the submodularity results of Section~\ref{sec:Methodology}.

\paragraph{Battery subproblem for fixed $(t,h',\omega)$.}
The local battery recourse is
\begin{subequations}
\label{eq:Gthprime}
\begin{alignat}{2}
G_{t,h'}^{B}(z;\omega)
:=\;&
\min_{\substack{n_{t,h'}^b(\omega),\,Q_{\cdot h't}(\omega),\\
                  g_{\cdot h't}(\omega)}}
\ \delta_t M_t^b n_{t,h'}^b(\omega)
&\notag\\[-1mm]
&+2\delta_t\eta_t Y
\sum_{h\in I_{h'}\setminus\{h'\}}
T^{t}_{hh'}\mu_{ht}(\omega)g_{hh't}(\omega)
&\label{eq:Gthprime-obj}\\[1mm]
\text{s.t.}\quad
&\eqref{eq:Battery-service},\quad\eqref{eq:small-batch-bound}\ \text{for all } h\in I_{h'},
&\notag\\
&g_{hh't}(\omega)Q_{hh't}(\omega)\ge z_{hh't}^2,
&\quad \forall h\in I_{h'}
\label{eq:Gthprime-persp}\\
&n_{t,h'}^b(\omega)\ge0,\quad g_{hh't}(\omega)\ge0,
&\quad \forall h\in I_{h'}.
\label{eq:Gthprime-nonneg}
\end{alignat}
\end{subequations}
For a remote active link with positive shipping cost, optimality gives $g_{hh't}(\omega)=1/Q_{hh't}(\omega)$; on a co-located link this value may be selected without loss of optimality, and for an inactive link both $Q_{hh't}(\omega)$ and its shipping cost vanish, which avoids the undefined ratio $z_{hh't}/Q_{hh't}(\omega)$. Because $\overline Q_{hh't}(\omega)$ in~\eqref{eq:small-batch-bound} is data, the active-link batch domains are fixed intervals. Co-located pairs ($h=h'$) incur no shipping cost, so their batch size falls to one and the batching terms vanish; the model thus distinguishes piece-flow replenishment at co-located hubs from batched shuttle transport between remote pairs, at $\mu_{ht}(\omega)Y/Q_{hh't}(\omega)$ round trips per year.

\paragraph{Truck fleet subproblem for fixed $(t,p,\omega)$.}
The path-specific fleet recourse is $G^{\mathrm V}_{t,p}(x^s;\omega)=\delta_tM^v_t\,n^v_{t,p}(\omega)$ with $n^v_{t,p}(\omega)$ given by~\eqref{eq:fleet-closed}, i.e., the optimal value of the linear program over the fleet-sizing constraints of Online Supplement~S7.

\section{Lov\'asz-Cut Benders Decomposition} \label{sec:Methodology}
\subsection{Submodular Recourse and Lov\'asz Cuts}
\label{sec:submodular-recourse}

By~\eqref{eq:recourse-decomp}, the recourse $\mathcal G(x^s,z;\omega)$ of the two-stage program in Section~\ref{sec:two-stage} separates into hub-wise battery subproblems~\eqref{eq:Gthprime} and path-wise fleet subproblems $G^{\mathrm V}_{t,p}$. The battery subproblem depends on the first stage only through the active-assignment vector $z$: when $z_{hh't}=0$, hub $h$ sends no swapping demand to charging hub $h'$ in period $t$, whatever value $x^s$ takes. For each charging hub $h'\in\mathcal H$, period $t\in\mathcal T$, and scenario $\omega\in\Omega$, let
\[
S_{h't}(z) := \{\,h\in I_{h'} : z_{hh't}=1\,\}
\]
be the set of active swapping hubs assigned to $h'$, and let $f_{h't}^\omega(S)$ denote the optimal value of~\eqref{eq:Gthprime} with $z_{hh't}=1$ if and only if $h\in S$. Because the hub-level coefficients $(\mu_{ht}(\omega),\sigma_{ht}(\omega))$ are fixed scenario data, $f_{h't}^\omega$ is a set function on $2^{I_{h'}}$ that does not change between master iterations. The fleet terms $G_{t,p}^{\mathrm V}(x^s;\omega)$ depend on $x^s$ rather than $z$ and are represented by exact cuts in Section~\ref{sec:lbb-algorithm}.

\subsubsection{Submodularity of the active-assignment value function}

We first record the reduced form of the value function. Fix $(h',t,\omega)$ and suppress these indices where unambiguous. Because the inventory requirement is nonnegative on the feasible batch domain (Section~\ref{sec:stock-sign-check}), inventory minimization gives $n^b=\Psi$, and $g_{hh't}(\omega)=1/Q_{hh't}(\omega)$ can be selected for every active hub without loss of optimality (Section~\ref{sec:two-stage}). Substituting \eqref{eq:Battery2} into the objective \eqref{eq:Gthprime-obj} and collecting terms by hub yields, for every $S\subseteq I_{h'}$,
\begin{equation}
f_{h't}^\omega(S)
\;=\;
\sum_{h\in S} a_h
\;+\;
\min_{Q\in\mathcal Q(S)}
\;\Biggl\{
\sum_{h\in S}c_h(Q_h)
\;+\;
B\,\sqrt{\sum_{h\in S} s_h(Q_h)}
\Biggr\},
\label{eq:f-reduced}
\end{equation}
where $\mathcal Q(S):=\prod_{h\in S}\,[1,\overline Q_{hh't}(\omega)]$ is the box of feasible batch sizes and
\[
a_h := \delta_t M^b_t\Bigl[(T^c+2T^{t}_{hh'})\mu_{ht}(\omega)-1
+\Phi^{-1}(1-\varepsilon^C)\sqrt{T^{t}_{hh'}}\,\sigma_{ht}(\omega)\Bigr],
\qquad
B := \delta_t M^b_t\,\Phi^{-1}(1-\varepsilon^C)\;\ge 0,
\]
\[
c_h(Q) := \delta_t M^b_t\,Q + \frac{2\delta_t\eta_t Y T^{t}_{hh'}\,\mu_{ht}(\omega)\,\mathbf 1_{\{h\ne h'\}}}{Q},
\qquad
s_h(Q) := (T^c+T^{t}_{hh'})\,\sigma^2_{ht}(\omega) + \frac{Q^2-1}{6}\;\ge 0.
\]
Since the demand moments entering $h'$ add over the active set, \eqref{eq:f-reduced} writes $f_{h't}^\omega$ as a modular term plus the partial minimization, over hub-separable batch sizes, of a modular cost and a concave square root of a nonnegative modular variance sum. For fixed batch sizes, the pooled term is a classical submodular concave function of a nonnegative modular sum \citep{atamturk2008polymatroids,atamturk2020submodularity}. Optimizing those batch sizes requires a value-function argument, as in \citet{qi2023scaling}, rather than relying solely on the fixed-coefficient square-root expression. Partial minimization does not preserve submodularity in general. It does here because the only cross-hub coupling is the pooled square root, which a \emph{single scalar multiplier} linearizes.

For $b\ge 0$ and $V\ge 0$ we use the conjugate (arithmetic--geometric mean) identity
\begin{equation}
b\sqrt{V}\;=\;\inf_{\tau>0}\;\Bigl\{\tau V+\frac{b^{2}}{4\tau}\Bigr\},
\label{eq:sqrt-conjugate}
\end{equation}
with attainment at $\tau=b/(2\sqrt V)$ when $b>0$ and $V>0$; zero cases are interpreted by limits. Applying \eqref{eq:sqrt-conjugate} to \eqref{eq:f-reduced} and exchanging the two minimizations gives
\begin{equation}
f_{h't}^\omega(S)-\sum_{h\in S}a_h
\;=\;
\inf_{\tau>0}\;\Bigl\{\sum_{h\in S}\chi_h(\tau)+\frac{B^{2}}{4\tau}\Bigr\},
\qquad
\chi_h(\tau):=\min_{Q\in[1,\overline Q_{hh't}(\omega)]}\bigl\{c_h(Q)+\tau\,s_h(Q)\bigr\}.
\label{eq:f-tau-form}
\end{equation}
For fixed $\tau$ the bracket in~\eqref{eq:f-tau-form} is \emph{modular} in $S$, so hubs interact only through the common multiplier $\tau$. Each $\chi_h$ is a pointwise minimum of the affine functions $\tau\mapsto c_h(Q)+\tau\,s_h(Q)$, whose slopes $s_h(Q)$ are nonnegative, and is therefore nondecreasing and concave on $(0,\infty)$.

\begin{lemma}[Scalar partial minimization preserves submodularity]
\label{lem:tau-min}
Let $I$ be a finite set, $\psi:(0,\infty)\to\mathbb R\cup\{+\infty\}$ arbitrary, and for each $h\in I$ let $\chi_h:(0,\infty)\to\mathbb R$ be \emph{nondecreasing}. Then
\[
F(S)\;:=\;\inf_{\tau>0}\;\Bigl\{\sum_{h\in S}\chi_h(\tau)+\psi(\tau)\Bigr\}
\]
is submodular on $2^{I}$ whenever the infima are finite.
\end{lemma}

The proof is given in Appendix~\ref{app:proof-lemma}. The lemma is an instance of the preservation of submodularity under minimization over a lattice \citep{topkis1978minimizing}: with the order on $\tau$ reversed, the bracket is jointly submodular in $(S,\tau)$, because its increment from adding $h$ is $\chi_h(\tau)$, which is monotone in $\tau$. The substantive step is the reduction~\eqref{eq:f-reduced}--\eqref{eq:f-tau-form}, which brings a value function defined by continuous batch-size optimization into this scalar form. Unlike the convex oracle of Section~\ref{sec:analytic-subproblem}, the argument does not require convexity of the batch-size objective. The argument uses only the monotonicity of $\chi_h$ in $\tau$, so it extends to any nondecreasing concave pooling function of a nonnegative separable sum through its conjugate representation.

\begin{proposition}[Submodularity]
\label{prop:submodular}
Fix a charging hub $h'\in\mathcal H$, period $t\in\mathcal T$, and scenario $\omega\in\Omega$. The set function $f_{h't}^\omega:2^{I_{h'}}\to\mathbb R$ defined by \eqref{eq:Gthprime} is submodular, i.e.,
\[
f_{h't}^\omega(S\cup\{h\}) - f_{h't}^\omega(S)
\;\ge\;
f_{h't}^\omega(T\cup\{h\}) - f_{h't}^\omega(T),
\qquad
\forall S\subseteq T\subseteq I_{h'},\ \forall h\in I_{h'}\setminus T.
\]
\end{proposition}

\begin{proof}
By the reduction \eqref{eq:f-reduced}--\eqref{eq:f-tau-form},
$f_{h't}^\omega(S)=\sum_{h\in S}a_h+F(S)$ with
$F(S)=\inf_{\tau>0}\{\sum_{h\in S}\chi_h(\tau)+B^2/(4\tau)\}$.
The exchange of minimization order behind \eqref{eq:f-tau-form} is justified because
\(
\min_{Q}\inf_{\tau}\{\cdot\}=\inf_{\tau}\min_{Q}\{\cdot\}
\)
holds for joint minimization over independent domains, and identity \eqref{eq:sqrt-conjugate} holds for every $V\ge0$ (for $V=0$ as an infimum). Each $\chi_h$ in \eqref{eq:f-tau-form} is nondecreasing on $(0,\infty)$ because $s_h(Q)\ge0$ on the feasible box $Q\in[1,\overline Q_{hh't}(\omega)]$ (recall constraint \eqref{eq:small-batch-bound} forces $Q_{hh't}(\omega)\ge z_{hh't}$, so $Q\ge1$ for active hubs, hence $(Q^2-1)/6\ge0$). The infimum is finite: the bracket is bounded below by $\sum_{h\in S}\min_Q c_h(Q)>-\infty$, while choosing any feasible finite batch vector gives a finite value in~\eqref{eq:f-reduced}. This argument also covers zero-variance sets without requiring attainment or coercivity in $\tau$. Lemma~\ref{lem:tau-min} (with $\psi(\tau)=B^2/(4\tau)$) therefore yields submodularity of $F$, and adding the modular term $\sum_{h\in S}a_h$ preserves submodularity.
\end{proof}

The function is also normalized ($f_{h't}^\omega(\emptyset)=0$) and nondecreasing, so every exact chain marginal is nonnegative; the verification is given in Appendix~\ref{app:proof-monotone}.

\subsubsection{Lov\'asz extension and linear under-estimators}
Let $f_{h't}^\omega$ be the active-assignment value function in Proposition~\ref{prop:submodular}. For any permutation $\pi=(\pi_1,\ldots,\pi_{|I_{h'}|})$ of the ground set $I_{h'}$, define the nested sequence
\[
S_0^\pi=\emptyset,
\qquad
S_j^\pi=\{\pi_1,\ldots,\pi_j\},\quad j=1,\ldots,|I_{h'}|.
\]
For any hub $h\in I_{h'}$, let $[h]$ denote its position in the permutation, i.e., $\pi_{[h]}=h$. The marginal increment of $h$ along this chain is
\[
\lambda_{hh't}^\omega(\pi):=f_{h't}^\omega\bigl(S_{[h]}^\pi\bigr)-f_{h't}^\omega\bigl(S_{[h]-1}^\pi\bigr).
\]
The corresponding Lov\'asz linear under-estimator for charging hub $h'$ and period $t$ is
\begin{equation}
L_{h't,\pi}^\omega(z_{h't})
:=
\sum_{h\in I_{h'}}\lambda_{hh't}^\omega(\pi)\,z_{hh't},
\qquad z_{h't}\in[0,1]^{I_{h'}},
\label{eq:Lovasz-h}
\end{equation}
where $z_{h't}=(z_{hh't})_{h\in I_{h'}}$ collects the active assignment variables for charging hub $h'$ in period $t$. Because $f^\omega_{h't}$ is submodular, \eqref{eq:Lovasz-h} is the extended-polymatroid inequality generated by the greedy order $\pi$, and the family over all permutations describes the convex closure of $f^\omega_{h't}$ on $[0,1]^{I_{h'}}$ \citep{lovasz1983submodular,atamturk2008polymatroids}; separation over the family is therefore a sort.

\begin{proposition}[Lov\'asz-type under-estimator]
\label{prop:lovasz}
Let $f_{h't}^\omega$ be the normalized submodular function in Proposition~\ref{prop:submodular}. Then, for any set $S\subseteq I_{h'}$, its incidence vector $z_{h't}=\mathbf 1_S\in\{0,1\}^{I_{h'}}$, and any permutation $\pi$ of $I_{h'}$,
\[
f_{h't}^\omega(S)
\;\ge\;
L_{h't,\pi}^\omega(\mathbf 1_S).
\]
Moreover, if the elements of $S$ appear consecutively at the beginning of $\pi$, then equality holds.
\end{proposition}

This is the greedy property of extended polymatroids \citep{lovasz1983submodular,atamturk2008polymatroids}: by submodularity, the chain marginal of each $h\in S$ is at most its marginal with respect to the elements of $S$ that precede it in $\pi$, and these bounds telescope to $f^\omega_{h't}(S)$, with equality when $S$ is a prefix of $\pi$. By~\eqref{eq:recourse-decomp}, the battery recourse of scenario $\omega$ is the sum of the block values $f_{h't}^\omega\bigl(S_{h't}(z)\bigr)$ over the blocks $(h',t)$ with $I_{h'}\neq\emptyset$. The master therefore carries, besides the scenario epigraph $\theta^B_\omega$, one epigraph variable $\theta^B_{\omega h't}\ge0$ per scenario and block, and each permutation $\pi$ of $I_{h'}$ yields the \emph{block cut}
\begin{equation}
\theta^B_{\omega h't}\;\ge\;L^\omega_{h't,\pi}(z_{h't}).
\label{eq:battery-cut}
\end{equation}
Summing block cuts over all blocks, with permutation $\pi_{h't}$ for block $(h',t)$, gives the hub--period-\emph{aggregated} cut
\begin{equation}
\theta^B_\omega\;\ge\;L_k^\omega(z)
:=
\sum_{t\in\mathcal T}\sum_{h'\in\mathcal H}\sum_{h\in I_{h'}}
\lambda_{hh't}^{k,\omega}\,z_{hh't},
\qquad
\lambda_{hh't}^{k,\omega}:=\lambda_{hh't}^\omega(\pi_{h't}),
\label{eq:Lovasz-global}
\end{equation}
where $k$ indexes the choice of permutations $(\pi_{h't})_{h',t}$. The aggregated layout needs only $\theta^B_\omega$ and one row per scenario and permutation choice, whereas the block layout adds up to one row per block.

\begin{proposition}[Validity of Lov\'asz Benders cuts]
\label{prop:cut-valid}
Fix a scenario $\omega\in\Omega$ and a master-feasible binary $(x^s,z)$. The values $\theta^B_{\omega h't}=f_{h't}^\omega\bigl(S_{h't}(z)\bigr)$ for every block and $\theta^B_\omega=\sum_{t\in\mathcal T}\sum_{h'\in\mathcal H}\theta^B_{\omega h't}$, the battery recourse of $\omega$, satisfy every block cut~\eqref{eq:battery-cut} and every aggregated cut~\eqref{eq:Lovasz-global}, for all permutations. Moreover, for the same permutations, the block cuts and the linking inequality $\theta^B_\omega\ge\sum_{t\in\mathcal T}\sum_{h'\in\mathcal H}\theta^B_{\omega h't}$ imply the aggregated cut.
\end{proposition}

The proof, given in Appendix~\ref{app:proof-cutvalid}, combines Proposition~\ref{prop:lovasz} with the separability~\eqref{eq:recourse-decomp}. Either layout thus keeps the master a relaxation of the two-stage program, and for the same permutations the block layout gives a relaxation at least as tight, at the cost of up to $|\mathcal H||\mathcal T|$ rows per scenario and permutation instead of one. The algorithm below uses the block layout; Section~\ref{sec:algo-comparison} compares the two layouts with the same oracle and solver settings. Because the cut coefficients depend only on scenario data, a cut generated at one master iteration remains valid at all later ones. These cuts could also be separated within a single branch-and-cut tree; the reported implementation uses an iterative loop.

\subsection{The Lov\'asz-Cut Benders Algorithm}
\label{sec:lbb-algorithm}

We embed these cuts in a Benders-type value-function decomposition, called the \emph{Lov\'asz-cut Benders decomposition} (LcBD). Let $\mathbf x:=(x^c,x^s,u^c,u^s,z,w,v,C)$ collect the first-stage variables and $J^{\mathrm{FS}}(\mathbf x)$ denote the infrastructure, capacity and diesel-substitution terms of~\eqref{eq:first-stage-obj-decomp}. The displayed master retains scenario-wise epigraph variables for battery and fleet costs:
\begin{equation}
\begin{aligned}
\min_{\mathbf x,\theta^B,\theta^V}\quad
&J^{\mathrm{FS}}(\mathbf x)+\sum_{\omega\in\Omega}p_\omega(\theta_\omega^B+\theta_\omega^V)\\
\text{s.t.}\quad &\eqref{eq:design-1-split}\text{--}\eqref{eq:binary-split},\\
&\theta_\omega^B\ge\sum_{t\in\mathcal T}\sum_{h'\in\mathcal H}\theta^B_{\omega h't},\quad\omega\in\Omega,\\
&\theta^B_{\omega h't}\ge L^\omega_{h't,\pi}(z_{h't}),\quad \pi\in\Pi^B_{\omega h't},\ \omega\in\Omega,\ h'\in\mathcal H,\ t\in\mathcal T,\\
&\theta_\omega^V\ge F_k^\omega(x^s),\quad k\in\mathcal K_\omega^V,\ \omega\in\Omega,\\
&\theta^B_{\omega h't},\theta_\omega^B,\theta_\omega^V\ge0.
\end{aligned}
\label{eq:benders-master-obj}
\end{equation}
Blocks with $I_{h'}=\emptyset$ carry no variable. The sets $\Pi^B_{\omega h't}$ and $\mathcal K_\omega^V$ index the battery cuts of each block, by their permutations, and the fleet cuts generated so far. The master contains no scenario-wise batch-size, battery-stock or fleet-sizing variables. The cuts are scenario-wise (a multi-cut scheme), and no cut aggregates across scenarios. At a master incumbent $(\bar x^s,\bar z)$, scenario-wise evaluations give the recourse costs. Battery cuts are generated for a scenario when $\bar\theta^B_\omega$ underestimates its battery recourse by more than a tolerance $\varepsilon^{\mathrm{cut}}\ge0$; each permutation used then adds at most one row per block, omitting rows whose coefficients are all zero and rows already in the master. Fleet cuts follow the same rule and are generated once per scenario (Proposition~\ref{prop:fleet-affine}).

\paragraph{Battery cut.} The battery cuts are the block cuts~\eqref{eq:battery-cut}, with chain marginals computed along a permutation that places the hubs assigned in $\bar z$ before the unassigned ones in each $I_{h'}$; by Proposition~\ref{prop:lovasz}, each exact block cut is tight at $\bar z$, so together they are tight for the battery recourse of the scenario. Exact marginals are nonnegative by monotonicity. The implementation must preserve the under-estimator property when values are computed numerically; clipping a negative difference to zero or rounding separate value estimates downward does not, by itself, ensure cut validity. Positive coefficients below a small threshold, which arise from round-off between nearly equal chain values, are rounded down to zero; because $z\ge0$, this only weakens the cut.

\paragraph{Fleet cut.} The fleet recourse $G^V(x^s;\omega)=\sum_{t\in\mathcal T}\sum_{p\in\mathcal P_t(\omega)} G^V_{t,p}(x^s;\omega)$ is separable across hub--period pairs and each $x^s_{ht}$ is binary, so the fleet cost admits an \emph{exact} linear representation on the binary lattice:
\begin{equation}
\theta_\omega^V \;\ge\; F_k^\omega(x^s):=G^V(\bar x^s;\omega) + \sum_{h\in\mathcal H}\sum_{t\in\mathcal T} d_{ht}^\omega\,(x^s_{ht} - \bar x^s_{ht}),
\label{eq:fleet-cut}
\end{equation}
where $d_{ht}^\omega$ is not a local subgradient of the dwell-time maximum but the exact \emph{flip difference} of the fleet cost between $x^s_{ht}=1$ and $x^s_{ht}=0$, which evaluates both branches of~\eqref{eq:dwell-time}:
\begin{equation}
\begin{aligned}
d_{ht}^\omega
={}&\frac{\delta_tM_t^v}{O_c}
\sum_{p\in P(h)\cap\mathcal P_t(\omega)}
\Bigl(\alpha_{pt}(\omega)+\Phi^{-1}(1-\varepsilon^C)\beta_{pt}(\omega)\Bigr)
\\[-1mm]
&\quad\times
\Bigl[\max\{T^r_{hp},T^s\}
-\max\{T^r_{hp},T^{c,\mathrm{fast}}R_{ph}(\omega)\}\Bigr].
\end{aligned}
\label{eq:fleet-gradient}
\end{equation}
The floor $T^r_{hp}$ cancels between the two branches, so the same coefficient holds for the incremental dwell~\eqref{eq:dwell-excess}. When the floor binds in both modes ($T^r_{hp}\ge T^{c,\text{fast}} R_{ph}(\omega)\ge T^s$), the bracket is zero and correctly reports that switching the service mode at $(h,t)$ has no fleet-cost effect, a case that a one-sided linearization of the maximum would misstate.

\paragraph{Fleet representation and optional master simplification.}
\label{sec:fleet-simplification}
\begin{proposition}[Exact affine fleet recourse]
\label{prop:fleet-affine}
For every scenario $\omega\in\Omega$ and every binary $x^s$,
\[
G^V(x^s;\omega)=G^V(\mathbf 0;\omega)+\sum_{h\in\mathcal H}\sum_{t\in\mathcal T}d^\omega_{ht}\,x^s_{ht},
\]
with $d^\omega_{ht}$ given by~\eqref{eq:fleet-gradient}. Hence the fleet cut~\eqref{eq:fleet-cut} coincides with $G^V(\cdot;\omega)$ at every binary $x^s$.
\end{proposition}
The proof is given in Appendix~\ref{app:fleet-verification}. Because the expected fleet cost $\sum_{\omega}p_\omega G^V(x^s;\omega)$ is affine in $x^s$ on the binary lattice, it can equivalently be placed in the master objective with coefficients $\sum_\omega p_\omega d^\omega_{ht}$ and the constant $\sum_\omega p_\omega G^V(\mathbf 0;\omega)$, which removes the epigraph variables $\theta^V_\omega$ and the fleet cuts. An affine function that is nonnegative at the vertices of the unit cube is nonnegative on the cube, so, once the fleet cut of each scenario has been generated, the two representations define the same master problem and the same continuous relaxation. The reported runs generate one fleet cut per scenario in the first iteration; all later iterations add battery cuts only.

Algorithm~\ref{alg:LBDA} summarizes the procedure. With exact master solves and exact recourse values, the finite number of binary vectors $(x^s,z)$ and cuts that are tight at the evaluated design give finite termination. If an iteration adds no cut, every epigraph is within $\varepsilon^{\mathrm{cut}}$ of the evaluated recourse, so the evaluated cost of the master incumbent exceeds the master bound by at most $2\varepsilon^{\mathrm{cut}}$ plus the master's optimality tolerance; the target gap is therefore reached at the first iteration that adds no cut, provided that $2\varepsilon^{\mathrm{cut}}$ plus the absolute master optimality gap is below $\varepsilon\max\{1,|\mathrm{UB}|\}$, and $\varepsilon^{\mathrm{cut}}$ is chosen accordingly.

\begin{algorithm}[!htbp]
\DontPrintSemicolon
\caption{Lov\'asz-cut Benders algorithm with block battery epigraphs and fleet epigraphs}
\label{alg:LBDA}
\KwIn{Gap tolerance $\varepsilon>0$; cut tolerance $\varepsilon^{\mathrm{cut}}\ge0$; scenarios $\Omega$ and probabilities $p_\omega$}
\KwOut{Best first-stage design found and upper bound $\mathrm{UB}$}
$\Pi^B_{\omega h't},\mathcal K_\omega^V\gets\emptyset$ for all $\omega,h',t$; $\mathrm{LB}\gets-\infty$; $\mathrm{UB}\gets+\infty$\;
\While{no incumbent is saved, or $\mathrm{UB}-\mathrm{LB}>\varepsilon\max\{1,|\mathrm{UB}|\}$}{
  Solve~\eqref{eq:benders-master-obj}; obtain an incumbent $(\bar{\mathbf x},\bar\theta^B,\bar\theta^V)$ and a valid master bound $\ell^{\mathrm{MP}}$\;
  $\mathrm{LB}\gets\max\{\mathrm{LB},\ell^{\mathrm{MP}}\}$; $\widehat G\gets0$\;
  \ForEach{$\omega\in\Omega$}{
    Evaluate the battery recourse $B_\omega$ from~\eqref{eq:Gthprime} and the fleet recourse $V_\omega$ from~\eqref{eq:fleet-closed}\;
    \If{$B_\omega>\bar\theta_\omega^B+\varepsilon^{\mathrm{cut}}$}{
      \ForEach{block $(h',t)$ with $I_{h'}\neq\emptyset$}{
        Form an active-first permutation $\pi$ of $I_{h'}$ and its chain marginals\;
        Add $\pi$ to $\Pi^B_{\omega h't}$, i.e., add the block cut~\eqref{eq:battery-cut}, unless its coefficients are all zero or the same row is already present\;
      }
    }
    \If{$V_\omega>\bar\theta_\omega^V+\varepsilon^{\mathrm{cut}}$}{
      Compute the flip differences~\eqref{eq:fleet-gradient} and add cut~\eqref{eq:fleet-cut}\;
    }
    $\widehat G\gets\widehat G+p_\omega(B_\omega+V_\omega)$\;
  }
  $J^{\mathrm{FEAS}}\gets J^{\mathrm{FS}}(\bar{\mathbf x})+\widehat G$\;
  Update $\mathrm{UB}$ and the saved incumbent if $J^{\mathrm{FEAS}}$ improves it\;
  If the target gap remains open and no cut was added, stop (LcBD) or tighten the master tolerance and continue (LcBD+)\;
}
\KwRet{the saved design and $\mathrm{UB}$}\;
\end{algorithm}

\subsection{Acceleration Techniques}
\label{sec:acceleration}

Algorithm~\ref{alg:LBDA} is valid as stated, but two bottlenecks govern its running time: the many evaluations of $f_{h't}^\omega(S)$ required for cut generation, and the time spent in master solves. We address the first with two computational accelerations, an analytic-gradient oracle and the memoization of evaluated sets, and the second with an adaptive master tolerance. Memoization and the adaptive tolerance leave the bounds unchanged; the possible overstatement of the cuts by the numerical oracle is bounded in Section~\ref{sec:analytic-subproblem}. We refer to the resulting variant as LcBD+. Like LcBD, it uses a single permutation per block, so a violated scenario adds at most one row per block in each iteration.

\subsubsection{Analytic-gradient evaluation of the battery subproblem}
\label{sec:analytic-subproblem}

Cut generation evaluates $f_{h't}^\omega(S)$ along permutation chains, so each Benders iteration requires many battery subproblem solves with $S$ fixed. For fixed $S$, optimality gives $n^b=\Psi$, and one may choose $g_h=1/Q_h$ for every active link without loss. A hub with $\mu_{ht}(\omega)=0$ has zero demand moments, its batch size is fixed at $Q_h=1$ by the small-batch bound, and it contributes neither inventory nor transportation cost. Hence $f_{h't}^\omega(S)=f_{h't}^\omega(S^+)$ with $S^+:=\{h\in S:\mu_{ht}(\omega)>0\}$, and the recourse is zero when $S^+=\emptyset$. For nonempty $S^+$ the pooled term of~\eqref{eq:f-reduced} is
\begin{equation}
\sqrt{C_{S^+}+\sum_{h\in S^+}\frac{Q_h^2}{6}},
\qquad
C_{S^+}:=\sum_{h\in S^+}\left((T^c+T^t_{hh'})\sigma_{ht}^2(\omega)-\frac16\right).
\label{eq:variance-soc-clean}
\end{equation}
When $C_{S^+}\ge0$, this is the Euclidean norm of $(\sqrt{C_{S^+}},(Q_h/\sqrt6)_{h\in S^+})$ and is convex in $Q$. A support-independent sufficient condition is $(T^c+T^t_{hh'})\sigma^2_{ht}(\omega)\ge1/6$ for every positive-demand feasible link. We verify it on the instance data before using the convex oracle (Section~\ref{sec:algo-comparison}). If an evaluated support had $C_{S^+}<0$, its evaluation would instead require a global nonconvex solve with valid value bounds. With the convex reciprocal terms $1/Q_h$ and the box $1\le Q_h\le\overline Q_{hh't}(\omega)$, \eqref{eq:f-reduced} is then a smooth convex program in $|S^+|$ variables. We solve it by a projected quasi-Newton method with analytic gradients, avoiding the construction of a conic model for every evaluation. Each evaluation also records a certified bound on its suboptimality: if $\phi_S(Q)$ denotes the objective of~\eqref{eq:f-reduced} for fixed $S$ and $\hat Q$ the returned batch vector, convexity gives $\phi_S(\hat Q)-\min_{Q\in\mathcal Q(S^+)}\phi_S(Q)\le\max_{Q\in\mathcal Q(S^+)}\nabla\phi_S(\hat Q)^\top(\hat Q-Q)$, a closed-form computation on the box. Each computed value is the cost of a feasible batch vector and thus lies above the exact value by at most this bound $\bar\epsilon$, so each computed chain marginal can exceed its exact counterpart by at most $\bar\epsilon$. Because each swapping hub is assigned to one charging hub per period, the block cuts of a scenario, summed over its blocks, can therefore overstate the exact battery recourse by at most $|\mathcal H||\mathcal T|\bar\epsilon$, while upper bounds remain valid because they use feasible recourse costs; Section~\ref{sec:algo-comparison} reports these quantities. Because the demand moments and batch bounds are fixed data for a given run, $f_{h't}^\omega(S)$ is a pure function of $(h',t,\omega,S)$, and all evaluations are memoized.

\subsubsection{Master solves}
\label{sec:impl-refinements}

The master is solved to an \emph{adaptive} optimality tolerance, coarse while the Benders gap is large and tightened as it closes, and tightened at once if a coarse solve yields no violated cut. The lower bound is always taken from the master's best \emph{bound} rather than its incumbent objective, so validity is unaffected. Each master is solved from scratch, without a starting solution (Section~\ref{sec:algo-comparison} explains this choice). As a further safeguard, if a reported master bound exceeds the best evaluated upper bound by more than a small numerical tolerance, the master is solved again from scratch, and the bound is discarded if the inconsistency persists.

Proposition~\ref{prop:cut-valid} holds for \emph{every} permutation of $I_{h'}$, and every permutation that places the assigned hubs first gives a cut that is tight at $\bar z$, so several such cuts per block could be added in one iteration. Both variants use one permutation, ordering the assigned and the unassigned hubs each by decreasing singleton value $f_{h't}^\omega(\{h\})$. Section~\ref{sec:algo-comparison} reports the effect of each technique by removing one at a time, and the effect of three permutations per block.

\section{Numerical Study}
\label{sec:numerical}

\subsection{Instances and Scenario Generation}
\label{sec:scenarios}

\subsubsection{Network data}
We construct test instances from the Freight Analysis Framework (FAF5) database maintained by the U.S.\ Bureau of Transportation Statistics \citep{bts2022faf5}. The network covers seven southeastern U.S.\ states (Alabama, Florida, Georgia, Mississippi, North Carolina, South Carolina, and Tennessee), comprising 111 logistics hubs connected by 2{,}996 road links with calibrated travel times. Origin--destination freight flows are drawn from FAF5 commodity flow estimates, yielding up to 249 origin--destination pairs with annual truck volumes ranging from 237 to 353{,}800 trucks per year. Shortest paths on the weighted road network define the set of transportation paths~$\mathcal P$, and inter-hub travel times $T^t_{hh'}$ are computed as all-pairs shortest path distances.

The hubs are consolidation and transfer facilities rather than roadside service sites: an FAF node aggregates the freight activity of its region, and in a hyperconnected relay network long-haul flow between two such nodes is executed as a sequence of hub-to-hub legs with load, trailer, or driver exchanged at the intermediate hubs \citep{montreuil2011toward,crainic2023hyperconnected,liu2025dynamic}. The intermediate stops, and hence the dwell floor $T^r_{hp}$ of Section~\ref{sec:capacity}, therefore belong to the operating model rather than being introduced for energy replenishment.

\subsubsection{Planning horizon and demand scenarios}
The planning horizon covers 2026--2040 in yearly periods, $\mathcal T=\{0,\ldots,14\}$, with period $t$ corresponding to calendar year $y_t:=2026+t$; it spans the expected transition of heavy-duty electric trucks from early deployment to mainstream adoption. The mean arrival rate $\alpha_{pt}(\omega)$ in~\eqref{eq:theta-normal} is driven by two independent sources of long-term uncertainty, the growth of total freight volume and the pace of electric-truck adoption:
\begin{equation}
\alpha_{pt}(\omega) \;=\; \frac{D_p \cdot (1+g_\omega)^{y_t-2022} \cdot \varphi_t(\omega)}{Y},
\label{eq:alpha-scenario}
\end{equation}
where $D_p$ is the base-year (2022) annual truck flow on path $p$ from FAF5, $g_\omega$ is the annual compound freight growth rate, and $\varphi_t(\omega)$ is the share of freight served by electric trucks. Based on FAF5 projections (approximately 29\% cumulative growth from 2022 to 2030), we draw $g_\omega$ from $\mathcal N(0.033,\,0.01^2)$, clipped to $[0.005,0.08]$. The electric share follows a logistic adoption curve. Let $q_\omega$ denote a sampled adoption-trajectory parameter and let $a_\omega=1.5q_\omega$ denote the corresponding logistic asymptote; then
\begin{equation}
\varphi_t(\omega) = \frac{a_\omega}{1+\exp\bigl(-k(y_t-t_{\mathrm{mid}}(\omega))\bigr)},
\qquad
t_{\mathrm{mid}}(\omega)=2026+\frac{1}{k}\log\!\left(\frac{a_\omega}{0.01}-1\right),
\label{eq:penetration-curve}
\end{equation}
where $k=0.35$, so $\varphi_{2026}(\omega)=0.01$ by construction. The adoption parameter is sampled as $q_\omega\sim\mathcal N(0.30,0.10^2)$ and clipped to $[0.05,0.70]$ before constructing the curve. This mean is consistent with industry projections: the International Council on Clean Transportation (ICCT) estimates heavy-duty zero-emission vehicle (ZEV) sales shares of 39--52\% by 2030--2032 \citep{icct2023ira}. Fleet penetration lags sales share because of vehicle lifetimes.
The standard deviation is set proportionally as $\beta_{pt}(\omega) = \alpha_{pt}(\omega)/3$, reflecting moderate within-period demand variability relative to the mean. The pair $(g_\omega,q_\omega)$ is drawn by Latin Hypercube Sampling (LHS) \citep{mckay1979comparison}: each marginal is stratified into $|\Omega|$ equal-probability intervals, one sample per stratum is mapped through the corresponding normal quantile before clipping, and each scenario receives probability $p_\omega=1/|\Omega|$.

Long-run uncertainty is carried by two common drivers $(g_\omega,q_\omega)$. Within a period, potential path demand shares a multiplicative factor, so the generator represents uncertainty about the speed of electrification rather than independent changes in corridor priorities. The coefficient of variation $\beta_{pt}/\alpha_{pt}=1/3$ is a calibration assumption, not a statistical estimate; it influences both stock sizing and fleet margins. With $\beta_{pt}=\alpha_{pt}/3$, the demand-variance input of Section~\ref{sec:capacity} exceeds the Poisson count variance on corridors with more than nine trucks per day; in SE-5 the exceptions are thin early-phase hubs carrying 2.3\% of battery demand over the horizon (0.5\% in 2040), where the input is still at least 56\% of the Poisson level. Section~\ref{sec:planning-insights} reports a stress test with spatially heterogeneous adoption.

\subsubsection{Path activation and phased deployment}
The scenario-dependent active path sets $\mathcal P_t(\omega)$ of~\eqref{eq:activation-scenario} are determined by comparing each path's electric truck demand in scenario $\omega$ against the activation threshold $d_{\min} = 5$ trucks per day. With penetration growing from about 1\% in 2026 to an average of about 33\% in 2040, the number of active paths increases over the horizon, and its \emph{timing varies across scenarios}. On the full network ($|\mathcal P|=249$), the union $\mathcal P_t$ of scenario-wise active sets grows from 6 paths in 2026 to 215 in 2040, and individual scenarios activate 49\% to 100\% of the union in any given year.

\paragraph{Diesel-substitution premium.}
The premium $c^{\mathrm{D}}$ prices demand on active corridors that is not served electrically. We set $c^{\mathrm{D}}=\$250$ per truck-day, between the operating-cost advantage of electric over diesel operation, about \$160--220 per truck-day over typical long-haul daily distances, and its sum with the value of avoided CO$_2$ emissions, \$250--340; Online Supplement~S1 gives the derivation and its sources, and S4 re-solves SE-5 at \$125 and \$500, which bracket this range.

\subsubsection{Test instances}
\label{sec:test-instances}
We construct five nested test instances of increasing size by selecting subsets of the network: SE-1 covers Georgia (18 hubs, 5 paths), SE-2 adds Florida (42, 28), SE-3 adds South and North Carolina (75, 86), SE-4 adds Tennessee (91, 148), and SE-5 is the full seven-state network (111, 249).
All instances use $|\mathcal T|=15$ yearly periods and $|\Omega|=50$ scenarios generated via LHS. The geographic restriction $I_{h'}$ limits BSS-to-CCS assignments to hubs within 3.6 hours driving distance, yielding an average of $|I_{h'}| \approx 14$ candidate swapping hubs per charging hub. Each hub's base-year grid power budget is an assumption that increases with the transmission voltage of its node, taken from the Homeland Infrastructure Foundation-Level Data (HIFLD) transmission lines, with magnitudes following distribution-level studies \citep{borlaug2021heavy,then2025planning}. Budgets grow at 3.8\% per year, the rate implied by the U.S.\ Department of Energy's National Transmission Needs Study \citep{doe2023needs}, so the grid constrains early deployment most strongly; Online Supplement~S1 gives the tiers and their derivation. Cost parameters are derived from published cost data in Online Supplement~S1 (Table~S1). The per-truck cost $M^v_t$ combines annualized vehicle capital with driver wages and benefits, so each hour of replenishment-induced dwell is priced at the full cost of a driver--truck unit; the battery holding cost $M^b_t$ is calibrated to battery-as-a-service lease rates and declines at 8\% per year.

\paragraph{Charging tiers and battery logistics.}
A truck charging on site holds a driver and a vehicle, so it uses 350\,kW-class equipment, $T^{c,\text{fast}}=15$ minutes per 89\,kWh pack \citep{wattev2025mcs}; returned swap batteries recharge in $T^c=120$ minutes per pack \citep{qi2023scaling}. Battery shuttles are dedicated medium-duty trucks carrying up to $\bar Q=19$ packs at $\eta_t=\$0.4$K per driving-day, without a driver (Online Supplement~S4 adds one and tests the third-party rate of \citet{qi2023scaling}); co-located pairs incur no shuttle cost and $T^t_{hh}=7.2$ minutes of handling. The reference pack's 120-mile range corresponds to a favorable duty cycle, about 0.74\,kWh per mile, against 1.4--2.1\,kWh per mile in heavily loaded Class~8 operation; Section~\ref{sec:planning-insights} summarizes how the plan responds to heavier energy use.

\subsubsection{Scenario-fan stability}
\label{sec:saa-stability-main}
To choose the number of scenarios, we test in-sample and out-of-sample stability on SE-3 with scenario fans of $N \in \{5, 10, 30, 50, 80, 100\}$, five independent LHS replications each, all solved by LcBD+ to a certified gap of at most 0.47\%. Each first-stage plan is re-evaluated on an independent set of $1{,}000$ scenarios, with capacity shortfalls, which the training model excludes by construction, priced at the diesel-substitution premium $c^{\mathrm{D}}$. From $N=5$ to $N=50$, the frequency of evaluation scenarios with a capacity exceedance falls from 16.7\% to 2.3\%, the mean difference between out-of-sample cost and in-sample objective falls from 26.7\% to 3.5\%, almost all of it the remediation charge, and the in-sample variability falls from 14.5\% to 7.6\%; the standard error of each 1{,}000-scenario cost estimate is at most 2.0\%. Beyond $N=50$ the in-sample variability rises again, because the replications with the most extreme sampled loads install markedly more capacity, so we retain $N=50$. The exceedance frequencies describe empirical protection under the selected generator rather than guaranteed violation probabilities, and the cost differences are policy-evaluation discrepancies rather than optimality-gap estimates; Online Supplement~S3 reports all fan sizes and replication ranges. The 50-scenario set used in the rest of the study is a separate draw from the same generator with a milder upper tail, which explains why its SE-3 objective (Table~\ref{tab:algo-comparison}) lies below the in-sample values of the five replications.

\subsection{Algorithmic Performance}
\label{sec:algo-comparison}

We compare three methods: the extensive form solved directly by Gurobi (Extensive), the Lov\'asz-cut Benders decomposition of Algorithm~\ref{alg:LBDA} (LcBD), and its accelerated variant of Section~\ref{sec:acceleration} (LcBD+). All runs use $|\Omega|=50$ scenarios, $\gamma=1$, a 3-hour wall-clock limit, a $0.01\%$ target gap and a 12-core Apple M3 Pro workstation with 36\,GB of memory, with Gurobi 13.0; the master problems and the extensive form use 4 threads, and the conic subproblems of LcBD use Gurobi's default thread setting. The extensive-form MIPGap and the Benders outer tolerance are $10^{-4}$, and the extensive form runs with a 28\,GB soft memory limit; LcBD solves each master to a fixed tolerance of $10^{-4}$, and LcBD+ uses the adaptive tolerances of Section~\ref{sec:impl-refinements}. Both decomposition methods add a cut only when an epigraph is violated by more than $\varepsilon^{\mathrm{cut}}=0.1$ (\$K), so $2\varepsilon^{\mathrm{cut}}$ is below $1.4\times10^{-5}$ of the objective on every instance; battery-cut coefficients below $10^{-6}$ are set to zero, and the master-bound safeguard uses a relative tolerance of $10^{-5}$. Masters are solved from scratch because, in preliminary runs, a master re-solved from the previous incumbent occasionally reported a bound above its optimum.

Extensive solves the same model in an exact strengthened form, so that it benefits from the same exact reformulations as the decomposition: assignment, batch and dispatch variables are defined on feasible links only, links without demand are removed, the pooled-variance constraint is written as a second-order cone, which is exact for binary $z$ under the link condition of Section~\ref{sec:analytic-subproblem}, and the fleet cost is replaced by its exact affine form (Proposition~\ref{prop:fleet-affine}). On SE-5 the model still has 1.26 million rows, 1.22 million columns and 0.57 million quadratic constraints. Online Supplement~S6 reports the formulation without these reformulations.

Table~\ref{tab:algo-comparison} summarizes the results. Extensive reaches the target on SE-1 and SE-2, in 69 seconds and 4.7 minutes. On SE-3 its bound reaches the optimum, but branch and bound stops at the 28\,GB memory limit after 94 minutes with an incumbent $0.035\%$ above it. On SE-4 its lower bound comes within $0.05\%$ of the best known value after 31 minutes, but its first incumbent within 1\% of that value appears only after 2.6 hours, and the run ends at a gap of $0.92\%$; on SE-5 the root relaxation is not solved within the limit. On SE-3 and SE-4 the difficulty therefore lies in finding good solutions rather than in the strength of the relaxation. LcBD reaches the target on all five instances, in 62 seconds to 34.3 minutes. Its time is dominated by the conic battery subproblems, which take 85\% of the SE-5 run. LcBD+ removes this bottleneck: it reaches the target in 10 seconds to 4.1 minutes, five to eight times faster than LcBD, and the two variants end within $0.004\%$ of each other. On SE-5, the gap of LcBD+ is 2.5\% after the second master iteration, 0.37\% after the third and 0.06\% after the fourth. The decomposition solutions improve on the extensive-form incumbents by $0.035\%$ on SE-3 and $0.9\%$ on SE-4.

\begin{table}[!htbp]
\centering
\caption{Computational performance of the extensive form (Extensive), the Lov\'asz-cut Benders decomposition (LcBD, Algorithm~\ref{alg:LBDA}), and its accelerated variant (LcBD+, Section~\ref{sec:acceleration}); $|\Omega|=50$ scenarios, 4 threads, 3-hour limit, $0.01\%$ target gap. ``t$\,\to x$\%'' is the wall-clock time to reach an optimality gap of $x$\%; ``---'' indicates that the level was not reached.}
\label{tab:algo-comparison}
\small
\begin{tabular}{ll rrr rr}
\toprule
Instance $(|\mathcal H|,|\mathcal P|)$ & Method & t$\,\to$1\% & t$\,\to$0.1\% & t$\,\to$0.01\% & Final gap & Best UB \\
\midrule
SE-1 $(18,5)$    & Extensive & 69s & 69s & 69s & $<$0.01\% & \textbf{14{,}334.4} \\
                 & LcBD      & 62s & 62s & 62s & $<$0.01\% & \textbf{14{,}334.4} \\
                 & LcBD+     & \textbf{10s} & \textbf{10s} & \textbf{10s} & $<$0.01\% & \textbf{14{,}334.4} \\
\midrule
SE-2 $(42,28)$   & Extensive & 4.7m & 4.7m & 4.7m & $<$0.01\% & \textbf{32{,}648.5} \\
                 & LcBD      & 3.2m & 3.2m & 3.2m & $<$0.01\% & \textbf{32{,}648.5} \\
                 & LcBD+     & \textbf{33s} & \textbf{34s} & \textbf{34s} & $<$0.01\% & \textbf{32{,}648.5} \\
\midrule
SE-3 $(75,86)$   & Extensive & 52.9m & 52.9m & --- & 0.035\%$^{\mathrm{a}}$ & 57{,}793.3 \\
                 & LcBD      & 8.2m & 8.2m & 8.2m & $<$0.01\% & \textbf{57{,}773.1} \\
                 & LcBD+     & \textbf{92s} & \textbf{92s} & \textbf{92s} & $<$0.01\% & \textbf{57{,}773.1} \\
\midrule
SE-4 $(91,148)$  & Extensive & 2.57h & --- & --- & 0.92\% & 85{,}037.3 \\
                 & LcBD      & 12.5m & 16.4m & 16.5m & $<$0.01\% & \textbf{84{,}275.4} \\
                 & LcBD+     & \textbf{2.3m} & \textbf{2.6m} & \textbf{2.7m} & $<$0.01\% & 84{,}278.4 \\
\midrule
SE-5 $(111,249)$ & Extensive & --- & --- & --- & ---$^{\mathrm{b}}$ & --- \\
                 & LcBD      & 21.9m & 27.9m & 34.3m & $<$0.01\% & \textbf{119{,}916.9} \\
                 & LcBD+     & \textbf{3.4m} & \textbf{3.7m} & \textbf{4.1m} & $<$0.01\% & 119{,}919.3 \\
\bottomrule
\end{tabular}

\par\smallskip
\begin{minipage}{\textwidth}\footnotesize
$^{\mathrm{a}}$Stopped at the memory limit after 94 minutes. $^{\mathrm{b}}$Root relaxation not solved within 3 hours.
\end{minipage}
\end{table}

\paragraph{Contribution of each component.} In Figure~\ref{fig:ablation}, each change slows the part of the run it targets. The Gurobi conic oracle and the removal of memoization lengthen only the battery evaluations. A fixed master tolerance lengthens only the master solves, from 1.1 to 5.5 minutes on SE-5, and has no measurable effect on SE-3. Three permutations per block double the time on both instances, mainly because each permutation requires its own chain of subproblem evaluations, while saving at most one iteration. The block layout matters most. With the aggregated cuts~\eqref{eq:Lovasz-global}, which are dominated by the block cuts (Proposition~\ref{prop:cut-valid}), the method reaches 0.1\% after 2.4 and 6.0 minutes but not 0.01\% within three hours on either instance, after 63 and 37 iterations instead of 4 and 6, and almost all of its time is spent in master solves.

\begin{figure}[!htbp]
\centering
\includegraphics[width=\textwidth]{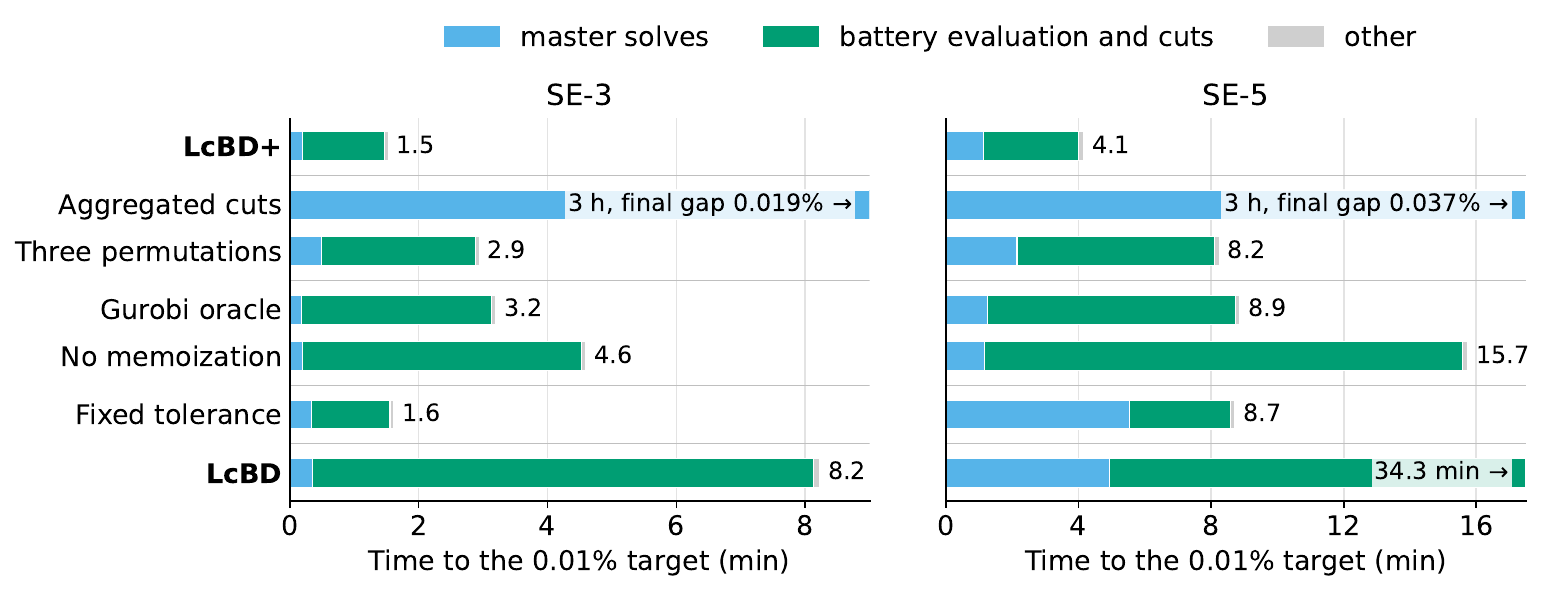}
\caption{Leave-one-out comparison on SE-3 and SE-5 ($|\Omega|=50$, 3-hour limit): wall-clock time to the $0.01\%$ target. Each variant changes one component of LcBD+: aggregated instead of block cuts, three permutations per block instead of one, the Gurobi conic oracle instead of the analytic one, no memoization, or a fixed instead of an adaptive master tolerance; LcBD changes the last three together. Truncated bars are labelled with their value.}
\label{fig:ablation}
\end{figure}

\paragraph{Oracle accuracy.} On SE-5 the convexity condition of Section~\ref{sec:analytic-subproblem} holds on all 530{,}724 positive-demand link--period--scenario combinations (minimum 0.246 against $1/6\approx0.167$), and over the 1.18 million oracle solves of the LcBD+ run the certified suboptimality bound is at most $\bar\epsilon=1.3\times10^{-6}$ (\$K). A battery cut can therefore overstate the exact recourse by at most $|\mathcal H||\mathcal T|\bar\epsilon\approx0.002$ (\$K) per scenario, which changes no reported gap. Of the 3.2 million chain differences, 222 were negative, all smaller than $1.4\times10^{-12}$ in magnitude, and were set to zero; memoization serves 81.7\% of the 6.49 million evaluation requests. For LcBD, the 64{,}353 battery-cut rows of the SE-5 run overstate the recourse at the best plan by at most $2.0\times10^{-4}$ (\$K) per row.

\subsection{The Value of Stochastic and Integrated Planning}
\label{sec:joint-vs-sequential}

Our central modeling claim is that two forms of integration matter: planning against the demand \emph{distribution} rather than its mean, and optimizing location, battery inventory, and fleet \emph{jointly} rather than in stages. We compare the integrated stochastic plan, the solution of the recourse problem (RP), with three baselines on SE-5: the mean-value plan (EV) removes distributional information, while two controlled sequential variants remove downstream-cost feedback from the siting objective.

\begin{itemize}
\item \textbf{Mean-value plan (EV)}: the joint model solved on the mean scenario (expected demand moments, mean-threshold activation), retaining structural integration while removing distributional information. This is the classical expected-value problem of \citet{birge1997introduction}.
\item \textbf{Sequential--informed}: a staged location$\,\to\,$inventory$\,\to\,$fleet pipeline. The location-and-capacity stage uses the \emph{same} first-stage feasible set as RP, including all 50 scenario-wise charging-throughput constraints, but its objective replaces exact downstream recourse with expected diesel substitution and a battery-holding surrogate, $\delta_t M_t^b (T^c + 2T^t_{hh'})\,\mathbb{E}_\omega[\mu_{ht}(\omega)]$ per assignment, the demand-weighted-distance proxy that classical location models use to price inventory \citep{shen2003joint,farahani2015location}. It therefore retains uncertainty-aware capacity feasibility while omitting variance pooling and the dwell-time--fleet feedback from the location objective.
\item \textbf{Sequential--naive}: the same scenario-feasible pipeline with no battery surrogate. This controlled variant removes all downstream-cost awareness from the location objective while holding uncertainty-aware capacity feasibility fixed.
\end{itemize}

Each plan is evaluated on the same 50 scenarios with the joint model's battery and fleet recourse evaluators. This evaluation is in-sample, since RP is optimized on the same set, so it isolates design mechanisms rather than testing generalization. Because the sequential baselines keep RP's scenario-wise throughput constraints, their differences from RP isolate the value of feeding downstream recourse costs back into location decisions. EV sizes capacity on the mean scenario and therefore fails the capacity requirement of the training scenarios themselves, so we report how often its capacity is breached rather than a remedy-dependent cost; its contrast with RP combines capacity protection and distributional recourse effects. Sequential--naive is a deliberately limited reference, and the comparison with Sequential--informed is more informative.

\begin{table}[!htbp]
\centering
\caption{Value of integrated stochastic planning on SE-5 ($|\Omega|=50$). Costs in \$K (15-year net present value; lower is better); all optimization runs reached a certified gap below 0.01\%.}
\label{tab:value-integration}
\small
\setlength{\tabcolsep}{4pt}
\begin{tabular}{l rrrrr r rr}
\toprule
& \multicolumn{5}{c}{Cost components} & & \multicolumn{2}{c}{Capacity outcomes} \\
\cmidrule(lr){2-6} \cmidrule(lr){8-9}
Plan & Infra. & Capacity & Battery & Fleet & Diesel & \makecell{Component sum,\\shortfall excluded} & \makecell{Installed\\(MW-yr)} & \makecell{Scenarios\\affected} \\
\midrule
Stochastic (RP) & 53{,}130.5 & 49{,}598.8 & 15{,}132.8 & 1{,}508.6 & 548.6 & 119{,}919.3 & 1{,}471.4 & 0\% \\
Mean value (EV) & 52{,}465.6 & 31{,}239.0 & 11{,}104.8 & 2{,}004.5 & 11{,}082.8 & 107{,}896.6$^\dagger$ & 991.8 & 62\% \\
Sequential--informed       & 39{,}242.5 & 52{,}526.9 & 36{,}079.9 & 76{,}043.9 & 368.7 & 204{,}261.9 & 1{,}454.8 & 0\% \\
Sequential--naive          & 33{,}512.1 & 54{,}994.4 & 143{,}848.1 & 18{,}058.2 & 323.9 & 250{,}736.7 & 1{,}456.1 & 0\% \\
\bottomrule
\end{tabular}

\par\smallskip
\begin{minipage}{\textwidth}\footnotesize
$^\dagger$Not a feasible stochastic objective value: EV breaches its installed charging capacity in 62\% of the evaluation scenarios, while capacity restoration is excluded from the component sum. Because the scenario-wise capacity requirement admits no recourse, EV is infeasible for the stochastic program and the classical value of the stochastic solution (VSS) is unbounded. Two tests adapted from \citet{maggioni2012analyzing}, solved by LcBD+ in the same model, measure EV's quality instead: keeping EV's stations and capacities as lower bounds gives the expected input value $\mathrm{EIV}=121{,}036.5$, i.e., a loss of upgrading the deterministic solution $\mathrm{LUDS}=\mathrm{EIV}-\mathrm{RP}=1{,}117.2$ (0.9\% of RP); allowing only the stations and BSS-to-CCS assignments that EV uses gives 145{,}421.3, i.e., a loss using the skeleton solution $\mathrm{LUSS}=25{,}502.1$ (21.3\%).\par\smallskip

Perfect-information benchmark: averaging the 50 scenario-wise solutions of the same hard-capacity model gives the wait-and-see value $\widehat{\mathrm{WS}} = 97{,}279.9$ and the estimated expected value of perfect information $\widehat{\mathrm{EVPI}} = \mathrm{RP} - \widehat{\mathrm{WS}} = 22{,}639.4$ (18.9\% of RP).
\end{minipage}
\end{table}

Table~\ref{tab:value-integration} separates economic performance from feasibility. EV's 2040 siting pattern is close to RP's (45 of its 46 swapping hubs are shared with RP's 48), but it installs 33\% less charging capacity. Its component sum is 10\% below RP only because capacity restoration is excluded: the plan breaches installed capacity in 62\% of evaluation scenarios and also incurs about 20 times RP's diesel-substitution cost. The tests in the table note, solved within the same model, show where the deficiency lies. Keeping EV's stations and capacities and allowing additions yields a plan within 0.9\% of RP: about 98\% of EV's station-years recur in RP, and the upgrade mainly adds 507\,MW-yr of capacity, about half of EV's, plus a few stations. Restricting the plan to the stations and assignments used by EV, in contrast, costs 21.3\% more than RP, almost entirely through diesel substitution in 2038--2040, when up to 3.3\% of expected demand is no longer electrified. The binding locations are grid-constrained hubs at which EV charges on site because mean loads fit the local connection. In 2039, EV's configuration requires up to 12.0\,MW across the training scenarios at a hub with an 8.1\,MW connection, and 19.1\,MW at the hub with the longest assignment discussed in Section~\ref{sec:planning-insights}, whose connection is 16.2\,MW. RP converts both into swap-only hubs replenished at grid-rich hubs 0.2 and 2.8 hours away. The skeleton cannot, so the corridors through these hubs remain on diesel in every scenario, because electrification is a first-stage decision. Mean-value planning thus selects nearly the right network, but it sizes capacity for the mean and misses where local grid access becomes binding; additional capacity cannot repair the latter where the connection itself is too small.

The sequential comparison isolates a different mechanism because both staged plans satisfy every scenario-wise throughput constraint and therefore have no capacity shortfall. Their component sums nevertheless exceed RP by 70\% (informed) and 109\% (naive). Without downstream-cost feedback, the naive variant selects inexpensive locations whose replenishment assignments raise expected battery cost to $9.5\times$ RP's (143{,}848.1 versus 15{,}132.8); fleet cost is also $12.0\times$ RP's. The demand-weighted battery surrogate repairs much of the inventory failure (battery cost falls to $2.4\times$ RP's) but shifts the loss to truck dwell: fleet cost rises to $50.4\times$ RP's (76{,}043.9 versus 1{,}508.6). Thus, even with capacity feasibility held fixed, a one-directional battery proxy does not capture the two-way coupling among location, pooling, and dwell.

\subsection{Planning Insights for Long-Haul Truck Energy Infrastructure}
\label{sec:planning-insights}

The SE-5 solution reveals how freight demand, charging access, and grid availability jointly shape the infrastructure network. We describe four features of the 50-scenario solution and then their robustness.

\paragraph{A hybrid network of hub roles.}
Figure~\ref{fig:deployment-map-m5} compares the 2040 service geography with the transmission network. Among 111 candidate hubs, the solution deploys 20 co-located BSS--CCS hubs, 25 charging-only hubs, and 28 swap-only hubs; the remaining 38 have no station. By 2040, 194 of 249 paths meet the expected-active threshold, while installed charging interfaces range from 0.6 to 32.0 MW. The coexistence of all three station roles indicates that freight-service geography and charging/grid geography do not fully coincide. Co-location is attractive where service demand and grid access align, whereas specialized hubs preserve coverage where only one function is economical.
The deployment findings are conditional on the fixed-demand convention of Section~\ref{sec:network}. Over 2026--2040, corridors with $w_{pt}=0$ carry at most 0.02\% of expected potential demand in RP and in both sequential plans of Section~\ref{sec:joint-vs-sequential}, and corridors with $0<w_{pt}<1$ at most 0.001\%; the diesel-served share accounts for at most 0.03\% of the battery demand at deployed BSSs and less than 0.06\% of total cost, so the convention has a negligible effect on the reported plans and comparisons.

\begin{figure}[!htbp]
\centering
\includegraphics[width=0.9\textwidth]{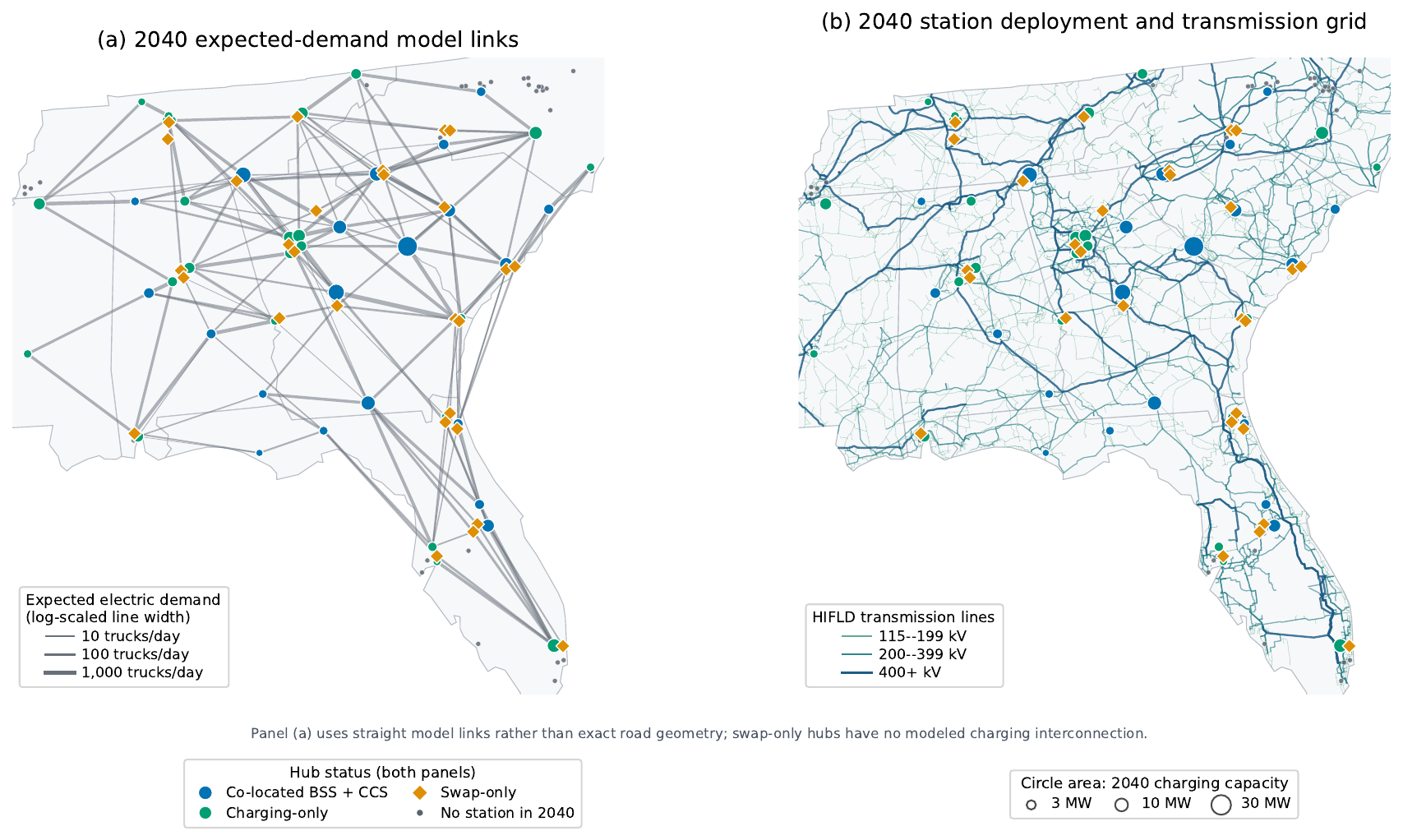}
\caption{The 2040 hybrid station network. (a) Paths with at least five expected electric trucks per day, aggregated onto straight model links that represent connectivity rather than road geometry. (b) Station deployment on HIFLD transmission lines. Circle area denotes installed charging capacity. A swap-only hub has no modeled charging interconnection, which does not imply that it lacks a real-world grid connection. Map lines delineate study areas and do not necessarily depict accepted national boundaries.}
\label{fig:deployment-map-m5}
\end{figure}

\paragraph{The network expands first in footprint and then in capacity.}
Figure~\ref{fig:network-scaling-m5} separates the evolution of station roles from the charging-capacity commitment. In 2026, nine hubs have stations (four charging-only, four swap-only, and one co-located), with 1.8 MW of installed charging capacity. The station count grows most rapidly in the early years and increases only from 66 to 73 after 2035, whereas installed capacity more than triples over the same period, from 100.1 to 334.9 MW. The planning emphasis within the optimal horizon-committed plan therefore shifts from geographic coverage toward capacity intensification as electric-truck demand matures. Because investments are not scenario-adaptive, this is a property of the committed trajectory rather than a rule for revising investments as adoption is observed.

\begin{figure}[!htbp]
\centering
\includegraphics[width=0.85\textwidth]{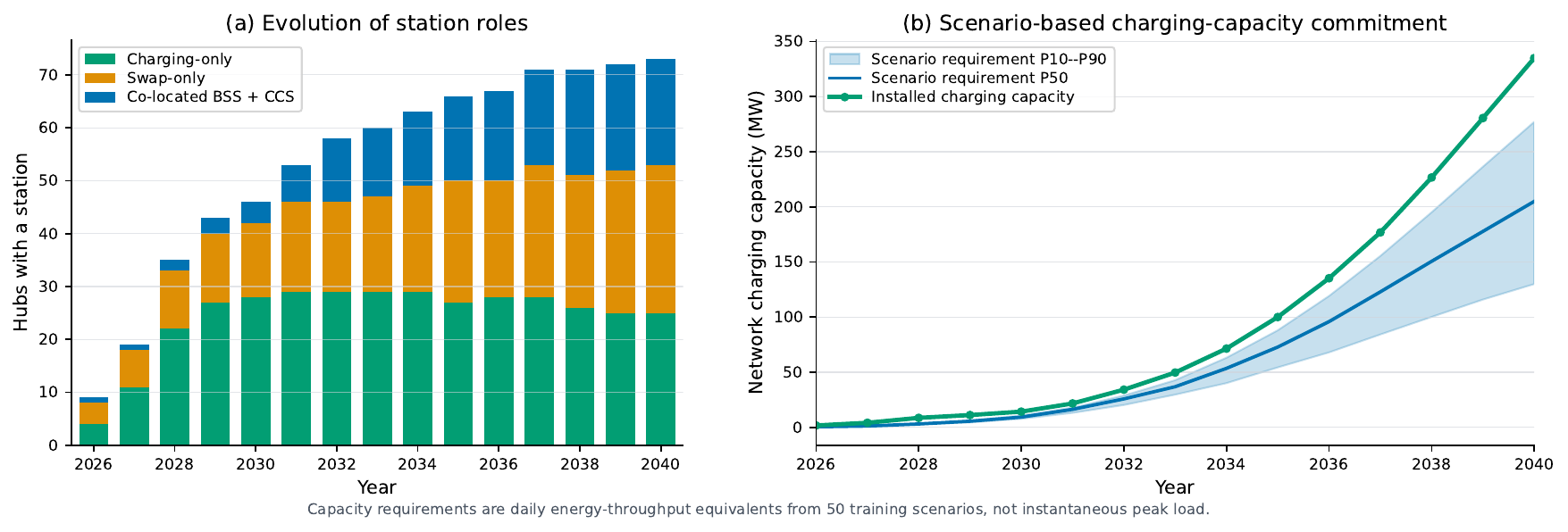}
\caption{Network footprint and charging-capacity evolution. (a) Hubs by station role. (b) Installed charging capacity against the median and P10--P90 range of the charging requirement over the 50 training scenarios; MW values are daily energy-throughput equivalents, not instantaneous peak loads. Early excess capacity partly reflects the 350 kW base capacity of each CCS.}
\label{fig:network-scaling-m5}
\end{figure}

\paragraph{Battery replenishment remains local but spatially flexible.}
Figure~\ref{fig:replenishment-network-m5} examines the 48 BSS-to-CCS assignments in 2040. Twenty assignments are co-located and 28 connect separate hubs. The 48 swapping hubs draw on 32 distinct charging hubs, so 13 of the 45 charging hubs serve on-site truck charging only and replenish no swapping stock. Replenishment pools stay small, with a mean of 1.5 and at most 3 swapping hubs per charging hub, so cross-hub variance pooling covers 68.3\% of expected battery demand but operates over a handful of streams rather than over a large aggregate. Weighted by expected battery demand, 64.4\% of replenishment is co-located and 89.2\% lies within one hour of one-way driving; the longest assignment, 2.8 hours, stays below the 3.6-hour limit. Spatial separation is therefore available without becoming the dominant operating pattern. It can nevertheless matter at selected locations: at the hub of the longest assignment, the maximum 2040 charging requirement across the training scenarios is 23.0 MW against a 16.9 MW grid connection, so remote replenishment preserves swapping service where local charging is constrained.

\begin{figure}[!htbp]
\centering
\includegraphics[width=0.85\textwidth]{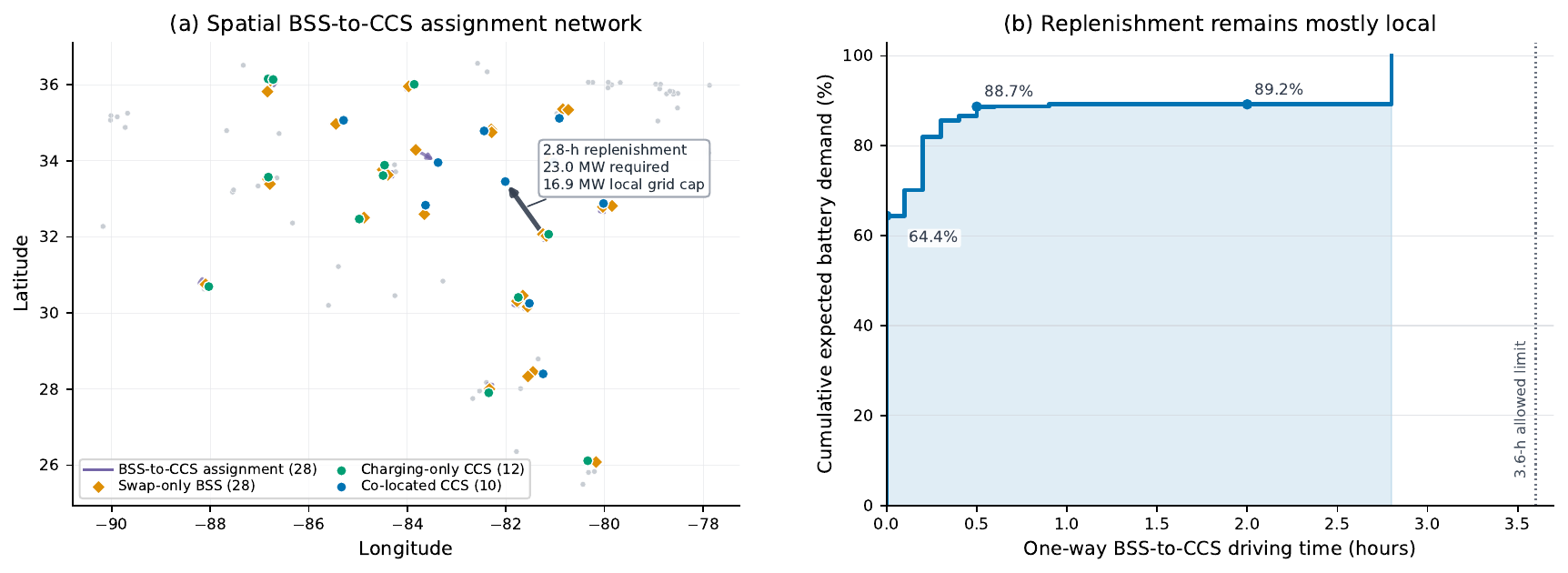}
\caption{Battery replenishment assignments in 2040. (a) Arrows connect each non-co-located BSS to its assigned CCS, with width log-scaled by expected battery demand. Markers show only the hubs of the 28 remote assignments: 28 swapping hubs and the 22 charging hubs that receive them (12 charging-only, 10 also hosting a BSS); all other hubs are grey. (b) Demand-weighted cumulative distribution of the one-way BSS-to-CCS driving time $T^{t}_{hh'}$ of the optimized assignments, excluding the charging time $T^c$, with the 3.6-hour feasibility limit. Map lines delineate study areas and do not necessarily depict accepted national boundaries.}
\label{fig:replenishment-network-m5}
\end{figure}

\paragraph{When separating swapping from charging pays off.}
To isolate the value of separation, we re-solve SE-5 with each swapping hub restricted to a charging station at the same hub ($I_{h'}=\{h'\}$), keeping all other settings; the run is certified optimal within 0.01\%. Total cost rises by 14.5\%, from 119{,}919.3 to 137{,}254.8. Two effects drive the increase. First, 27 of the 28 swap-only hubs of the main plan must install their own charging, so the 2040 network needs 72 charging hubs instead of 45 and infrastructure cost rises by 13{,}694.1. Second, swapping is withdrawn at eight hubs, including the grid-constrained hub served remotely in the main plan, and 2.2\% of 2040 demand reverts to diesel, adding 14{,}214.4. Separation, in turn, costs battery inventory and shuttle transport: the co-location-only plan saves 8{,}965.2 in battery cost and 2{,}746.7 in charging capacity. Separation therefore pays off where swapping demand is high relative to local grid access, or where a nearby grid-rich hub can absorb charging that would otherwise be duplicated, and its value falls as shuttle transport becomes more expensive (Online Supplement~S4).
\paragraph{Robustness.} Online Supplement~S4 re-solves SE-5 under heavier energy use, alternative shuttle costs and diesel premiums, spatially heterogeneous adoption, a tighter service level and irreversible investment; the deployment sequence and the predominance of local replenishment persist in every case. Heavier energy use, 1.2 and 1.5\,kWh per mile, raises total cost by 49\% and 93\% and shifts the mix toward co-location, and the swap load at the hub of the longest assignment then outgrows every grid connection within reach, so the reach of centralized charging is bounded by the largest single connection within the assignment radius. Redistributing adoption across states raises installed 2040 capacity by 6\% to 16\%, and the main plan breaches capacity in 14\% to 30\% of these scenarios, so its capacity protection is specific to the training generator. A 99\% service level raises total cost by 1.4\%, and irreversible investment, which rules out the four charging and eleven swapping station closures of the reported plan, by 1.7\%, both with an almost unchanged footprint.

\section{Conclusion}

This paper studies an SLCC network for long-haul truck electrification, with a multi-period, two-stage stochastic model that coordinates station deployment, BSS-to-CCS assignments, and charging capacity with battery inventory, replenishment, and incremental fleet requirements under uncertain demand.

Building on the submodularity-based framework of \citet{qi2023scaling}, we give a scalar proof that the battery value function is submodular and embed the resulting hub--period block Lov\'asz cuts, together with exact fleet cuts, in a Benders decomposition accelerated by a convex oracle with certified error bounds and memoization. The method reaches gaps below $0.01\%$ on all five instances and solves the largest in about four minutes. The numerical study shows that mean-value planning provides a useful siting baseline but underestimates capacity needs and misses the assignments required at grid-constrained hubs, that sequential siting reduces infrastructure costs at the expense of higher battery and fleet costs, that the optimized network keeps swapping near freight demand while using remote charging selectively at grid-suitable hubs, and that its committed deployment establishes geographic coverage before expanding charging capacity.

The findings have several limits. Infrastructure is committed over the horizon rather than adapted as adoption becomes observable, and the capacity protection is specific to the training scenario generator. The fixed-demand sizing convention has a negligible cost effect in the reported plans, but an endogenous served-flow formulation would require renewed value-function and cut analysis. The pooled stock meets the service target under day-to-day rate variability except at thin remote hubs, but not at remote hubs if variability also accrues within the day.

Further extensions include multistage investment, joint sizing of the two buffers of remote links, integer batch quantities, corridor-level adoption heterogeneity, endogenous route choice, multiple charging suppliers per swapping hub, and shortfall-recourse or chance-constrained capacity models; extensions that change served battery demand or charging durations require the validity of the present cuts to be re-established.

\medskip
\noindent\textbf{Data availability statement.} The freight-flow data are from the public Freight Analysis Framework \citep{bts2022faf5}. The code, generated scenario sets, and result files supporting this study are available from the corresponding author upon reasonable request.

\medskip
\noindent\textbf{Declaration of competing interest.} The authors report that there are no competing interests to declare.

\medskip
\noindent\textbf{Funding.} This research did not receive any specific grant from funding agencies in the public, commercial, or not-for-profit sectors.

\medskip
\noindent\textbf{Declaration of generative AI and AI-assisted technologies in the manuscript preparation process.} During the preparation of this work the authors used generative AI tools (Anthropic's Claude models) to help with language improvement and consistency checks. After using these tools, the authors reviewed and edited the content as needed and take full responsibility for the content of the published article.

\medskip
\noindent\textbf{CRediT authorship contribution statement.} \textbf{Yujia Xu:} Conceptualization, Methodology, Investigation, Writing -- original draft. \textbf{Walid Klibi:} Conceptualization, Methodology, Supervision, Writing -- review \& editing. \textbf{Benoit Montreuil:} Conceptualization, Methodology, Supervision, Writing -- review \& editing. All authors approved the final version and agree to be accountable for all aspects of the work.

\appendix
\counterwithin{figure}{section}
\counterwithin{table}{section}
\counterwithin{equation}{section}
\section{Proofs}
\label{app:proofs}
\subsection{Proof of Lemma~\ref{lem:tau-min}}
\label{app:proof-lemma}

\begin{proof}[Proof of Lemma~\ref{lem:tau-min}]
Write $m_S(\tau):=\sum_{h\in S}\chi_h(\tau)+\psi(\tau)$, so that $F(S)=\inf_{\tau>0}m_S(\tau)$ and, for $h\notin S$, $m_{S\cup\{h\}}(\tau)=m_S(\tau)+\chi_h(\tau)$. Fix $S\subseteq T\subseteq I$ and $h\in I\setminus T$; we show the equivalent inequality
\[
F(S\cup\{h\})+F(T)\;\ge\;F(S)+F(T\cup\{h\}).
\]
Let $\varepsilon>0$ and pick $\tau_a,\tau_b>0$ with
$m_{S\cup\{h\}}(\tau_a)\le F(S\cup\{h\})+\varepsilon$ and
$m_{T}(\tau_b)\le F(T)+\varepsilon$.

\emph{Case 1: $\tau_a\ge\tau_b$.} Evaluate the two right-hand-side terms at $\tau_a$ and $\tau_b$ respectively:
\[
F(S)\;\le\;m_S(\tau_a)\;=\;m_{S\cup\{h\}}(\tau_a)-\chi_h(\tau_a),
\qquad
F(T\cup\{h\})\;\le\;m_{T}(\tau_b)+\chi_h(\tau_b).
\]
Adding these and using $\chi_h(\tau_b)\le\chi_h(\tau_a)$ (monotonicity, $\tau_b\le\tau_a$),
\[
F(S)+F(T\cup\{h\})
\;\le\;
m_{S\cup\{h\}}(\tau_a)+m_{T}(\tau_b)
\;\le\;
F(S\cup\{h\})+F(T)+2\varepsilon.
\]

\emph{Case 2: $\tau_a<\tau_b$.} Evaluate the right-hand side at the swapped points, using $m_S = m_T-\sum_{h'\in T\setminus S}\chi_{h'}$ and $m_{T\cup\{h\}}=m_{S\cup\{h\}}+\sum_{h'\in T\setminus S}\chi_{h'}$:
\[
F(S)\;\le\;m_{S}(\tau_b)\;=\;m_{T}(\tau_b)-\!\!\sum_{h'\in T\setminus S}\!\!\chi_{h'}(\tau_b),
\qquad
F(T\cup\{h\})\;\le\;m_{S\cup\{h\}}(\tau_a)+\!\!\sum_{h'\in T\setminus S}\!\!\chi_{h'}(\tau_a).
\]
Adding these and using $\chi_{h'}(\tau_a)\le\chi_{h'}(\tau_b)$ for every $h'\in T\setminus S$ (monotonicity, $\tau_a<\tau_b$) gives the same bound. Letting $\varepsilon\downarrow0$ completes the proof.
\end{proof}

\subsection{Normalization and Monotonicity of the Battery Value Function}
\label{app:proof-monotone}
Clearly, $f_{h't}^\omega(\emptyset)=0$. To verify monotonicity, take $S\subseteq T$ and any batch vector $Q_T\in\mathcal Q(T)$. Its restriction $Q_T|_S$ lies in $\mathcal Q(S)$. Each added hub $h\in T\setminus S$ has nonnegative itemwise contribution $a_h+c_h(Q_h)$: the term $-1+Q_h$ is nonnegative because $Q_h\ge1$, and all remaining expected-inventory, transportation, and travel-dispersion terms are nonnegative. Removing those hubs also cannot increase the pooled standard-deviation term. Hence the objective at $Q_T$ for $T$ is at least the objective at $Q_T|_S$ for $S$, which is at least $f_{h't}^\omega(S)$; minimizing over $Q_T$ gives $f_{h't}^\omega(T)\ge f_{h't}^\omega(S)$. If $\mu_{ht}(\omega)=\sigma_{ht}(\omega)=0$, then $\overline Q_{hh't}(\omega)=1$, so $Q_h=1$ and the added contribution is exactly zero. Consequently, every exact chain marginal is nonnegative.

\subsection{Proof of Proposition~\ref{prop:cut-valid}}
\label{app:proof-cutvalid}
\begin{proof}[Proof of Proposition~\ref{prop:cut-valid}]
Fix $\omega$ and a master-feasible binary $(x^s,z)$, and set $\theta^B_{\omega h't}=f_{h't}^\omega\bigl(S_{h't}(z)\bigr)$ for every block and $\theta^B_\omega=\sum_{t,h'}\theta^B_{\omega h't}$, which by~\eqref{eq:recourse-decomp} is the battery recourse of $\omega$; each $\theta^B_{\omega h't}\ge0$ by Appendix~\ref{app:proof-monotone}. Proposition~\ref{prop:lovasz} with $S=S_{h't}(z)$ gives $L_{h't,\pi}^\omega(z_{h't})\le\theta^B_{\omega h't}$ for every block and permutation, so every block cut~\eqref{eq:battery-cut} holds, and summing over the blocks gives every aggregated cut~\eqref{eq:Lovasz-global}. For dominance, summing the block cuts satisfied by any master values for the permutations $\pi_{h't}$ and applying the linking inequality gives $\theta^B_\omega\ge L_k^\omega(z)$. The coefficients are chain marginals of $f_{h't}^\omega$, which depends only on fixed scenario data and its argument set, so the validity of a cut does not depend on the iteration at which it was generated.
\end{proof}

\subsection{Proof of Proposition~\ref{prop:fleet-affine}}
\label{app:fleet-verification}
\begin{proof}[Proof of Proposition~\ref{prop:fleet-affine}]
Fix $\omega$, $t$, $p\in\mathcal P_t(\omega)$, and a binary $x^s$, and write $\lambda_{pt}:=\alpha_{pt}(\omega)+\Phi^{-1}(1-\varepsilon^C)\beta_{pt}(\omega)\ge0$. By~\eqref{eq:dwell-excess} and~\eqref{eq:fleet-closed},
\[
G^V_{t,p}(x^s;\omega)=\delta_tM^v_t\,n^v_{t,p}(\omega)=\frac{\delta_tM^v_t}{O_c}\,\lambda_{pt}\sum_{h\in\mathcal H_p}\max\bigl\{0,\,T^sx^s_{ht}+T^{c,\mathrm{fast}}(1-x^s_{ht})R_{ph}(\omega)-T^r_{hp}\bigr\}.
\]
For binary $x^s_{ht}$, each maximum equals
\[
x^s_{ht}\bigl(\max\{T^r_{hp},T^s\}-T^r_{hp}\bigr)+(1-x^s_{ht})\bigl(\max\{T^r_{hp},T^{c,\mathrm{fast}}R_{ph}(\omega)\}-T^r_{hp}\bigr),
\]
which is affine in $x^s_{ht}$. Summing over $p\in\mathcal P_t(\omega)$ and $t\in\mathcal T$, and collecting the coefficient of each $x^s_{ht}$ over the paths $p\in P(h)\cap\mathcal P_t(\omega)$, gives $G^V(x^s;\omega)=G^V(\mathbf 0;\omega)+\sum_{h,t}d^\omega_{ht}x^s_{ht}$ with $d^\omega_{ht}$ as in~\eqref{eq:fleet-gradient}. Finally, at a binary $x^s$ the right-hand side of~\eqref{eq:fleet-cut} equals $G^V(\mathbf 0;\omega)+\sum_{h,t}d^\omega_{ht}\bar x^s_{ht}+\sum_{h,t}d^\omega_{ht}(x^s_{ht}-\bar x^s_{ht})=G^V(x^s;\omega)$.
\end{proof}

\section{Validation of the Battery-Stock Approximation}
\label{app:mc-validation}
\label{app:pool-sim}
Online Supplement~S2 validates the batching-variance expression~\eqref{eq:VarQX-scenario} for a single demand stream by simulation. Within the batch sizes the model admits, the expression is exact on low-flow corridors, tracks the simulation on mid-flow corridors, and overestimates the variance on high-flow corridors, where the batch capacity binds first; beyond the small-batch bound it would understate the variance by up to 23\%, so $\gamma=1$ is the largest round value that keeps every corridor class on the conservative side.
To assess the composed system of Section~\ref{sec:capacity}, we simulate every charging hub that replenishes at least one swapping hub in the 2040 SE-5 plan (32 pools) in five scenarios, at the 10th to 90th percentiles of 2040 pool demand, which gives 227 hub--scenario cases with positive demand (100 co-located, 127 remote). Each pool holds $\lceil\Psi_{t,h'}\rceil$ packs at integer batch sizes, divided between the swapping shelves and the charging hub as in the three lines of~\eqref{eq:Battery2}; a second allocation keeps, with the same pool total, one additional truck request on each swapping shelf, because a request of two or three packs can overshoot the reorder point. Each scenario-active relay path carries one truck stream that reaches every pool hub on the path at its driving-time offset, and a control run gives each path--hub pair its own stream. In the base arrival model, which follows the interpretation of~\eqref{eq:theta-normal} as day-to-day variability, the daily rate is random and trucks arrive as a Poisson process within the day; a stress test lets the variance-rate approximation hold within the day as well. The cycle service level is the share of replenishment cycles at a swapping hub that end without a backordered request, the per-cycle, per-location measure targeted by~\eqref{eq:Battery-service}; the pack fill rate is the share of packs supplied immediately. Online Supplement~S5 gives the full protocol.

\begin{table}[!htbp]
\centering
\caption{Simulated service of the 2040 SE-5 battery pools (32 pools, five scenarios, 227 hub--scenario cases; target cycle service 95\%). Service levels and shares of battery demand are weighted by battery demand; ``meeting 95\%'' counts hub--scenario cases.}
\label{tab:pool-service-main}
\footnotesize
\setlength{\tabcolsep}{4pt}
\begin{tabular}{l rr rr}
\toprule
Arrival model & \multicolumn{2}{c}{Day-to-day variability} & \multicolumn{2}{c}{Stress test (renewal)} \\
\cmidrule(lr){2-3}\cmidrule(lr){4-5}
Shelf allocation & Derived & With overshoot & Derived & With overshoot \\
\midrule
Cycle service, all hubs & 99.4\% & 99.7\% & 92.2\% & 93.1\% \\
\quad co-located hubs & 100.0\% & 100.0\% & 97.0\% & 96.9\% \\
\quad remote hubs & 98.2\% & 99.2\% & 83.4\% & 86.4\% \\
Pack fill rate & 99.8\% & 99.9\% & 94.8\% & 95.2\% \\
\midrule
Co-located hubs meeting 95\% (100 cases) & 100\% & 100\% & 91\% & 90\% \\
Remote hubs meeting 95\% (127 cases) & 60\% & 81\% & 0\% & 7\% \\
\quad above 150 packs per day (73 cases) & 90\% & 100\% & 0\% & 0\% \\
\quad 50 to 150 packs per day (28 cases) & 36\% & 89\% & 0\% & 14\% \\
\quad below 50 packs per day (26 cases) & 0\% & 19\% & 0\% & 19\% \\
Battery demand at hubs below 95\% & 3.4\% & 0.7\% & 39.2\% & 39.6\% \\
Remote requests dispatched at once by the CCS & 98.5\% & 97.8\% & 89.5\% & 88.3\% \\
\midrule
Common minus independent streams, cycle service (mean / worst hub) & \multicolumn{2}{c}{$<$0.1 / $-$1.6 points} & \multicolumn{2}{c}{$<$0.1 / $-$0.9 points} \\
\bottomrule
\end{tabular}
\end{table}

Table~\ref{tab:pool-service-main} reports the results. Shared-path dependence is present but small. In the SE-5 plan, 15.6\% of the pool--period--scenario combinations with at least two assigned hubs and positive demand contain a shared path. By the covariance expression of Online Supplement~S7, the covariances raise the pooled demand standard deviation by 0.4\% on demand-weighted average and by at most 3.0\%, within the 3.7\% bound derived there, and common streams change cycle service by less than 0.1 percentage points on average and by at most 1.6 points at any hub.

Under day-to-day variability, the pooled stock meets the target in aggregate, with a demand-weighted cycle service of 99.4\% and a pack fill rate of 99.8\%, and at every co-located hub. The margin reflects the variance-rate approximation, because day-to-day rate variability materializes only partly within a few hours. Shortfalls occur at thin remote hubs, where a request of two or three packs can overshoot the reorder point and, on paths with fewer than nine trucks per day, the Poisson count variance exceeds $\beta^2_{pt}$; these hubs carry 3.4\% of battery demand, and 0.7\% once one truck request is kept on each shelf.

The stress test separates the two buffers of~\eqref{eq:Battery2}. Co-located hubs, whose service depends on the pooled buffer alone, reach 97\%, but remote hubs reach only 83\% to 86\% at every volume level, because a cycle is stockout-free only if both the pooled buffer and the local return-leg buffer suffice, and each is sized at the 95\% quantile. The inequality of Section~\ref{sec:capacity} therefore does not carry over to the per-hub service event when variability accrues within the day. Sizing the two buffers of remote links jointly, for example each at $\Phi^{-1}(1-\varepsilon^C/2)$, would address this case, but it changes~\eqref{eq:Battery2} and every reported plan, and is left for future work.

\bibliography{mybibliography}

\clearpage
\counterwithout{figure}{section}
\counterwithout{table}{section}
\counterwithout{equation}{section}
\setcounter{section}{0}
\setcounter{figure}{0}
\setcounter{table}{0}
\setcounter{equation}{0}
\renewcommand{\thesection}{S\arabic{section}}
\renewcommand{\thetable}{S\arabic{table}}
\renewcommand{\thefigure}{S\arabic{figure}}
\renewcommand{\theequation}{S\arabic{equation}}
\renewcommand{\theHsection}{S\arabic{section}}
\renewcommand{\theHtable}{S\arabic{table}}
\renewcommand{\theHfigure}{S\arabic{figure}}
\renewcommand{\theHequation}{S\arabic{equation}}
\setlength{\parindent}{0pt}
\setlength{\parskip}{0.4em}
\pdfbookmark[0]{Online Supplement}{online-supplement}

\begin{center}
{\large Online Supplement to}\\[0.5em]
{\Large\bfseries Designing battery swapping and charging networks\\[0.2em]
for electric freight transport under demand uncertainty}\\[1em]
Yujia Xu, Walid Klibi, and Benoit Montreuil
\end{center}
\vspace{0.5em}

\noindent This supplement contains the full parameter table of the computational study (Section~S1), the derivation of the batching variance and its Monte Carlo validation for a single demand stream (Section~S2), the per-replication results of the scenario-fan stability test (Section~S3), the robustness experiments (Section~S4), the service simulation of the deployed battery pools (Section~S5), the extensive form without the exact reformulations (Section~S6), and the derivations behind the shared-path covariance and the fleet-sizing constraints (Section~S7). Numbers with the prefix S refer to this supplement; all other section, equation, remark, table, and figure numbers refer to the main paper.

\section{Model parameters}
\label{app:parameters}

Table~\ref{Tab:experiment-parameters} lists the parameters of the computational experiments. Section~6.1.4 of the paper summarizes their sources, and the derivations follow the table.

\begin{table}[!htbp]
\centering
\footnotesize
\renewcommand{\arraystretch}{1.00}
\setlength{\tabcolsep}{4pt}
\caption{Summary of model parameters for computational experiments. Derivations follow the table.}
\label{Tab:experiment-parameters}
\begin{tabularx}{\textwidth}{@{}>{\raggedright\arraybackslash}X r l >{\raggedright\arraybackslash}p{5.2cm}@{}}
\toprule
Parameter & Value & Unit & Source \\
\midrule
\multicolumn{4}{@{}l}{\textbf{Electric Trucks}} \\
\quad Purchase cost (without battery) & 216,000 & \$ & \cite{hunter2021spatial,ICCT2022} \\
\quad Salvage period & 7 & years & \cite{danielis2018probabilistic} \\
\quad Annualized capital cost & 30.9 & \$K/yr & Derived \\
\quad Driver wages + benefits & 38 & \$/hr & \cite{atri2025operational} \\
\quad Driver cost per operated truck & 205.2 & \$K/yr & Derived ($18\,$h $\times$ 300 days) \\
\quad Total cost per provisioned truck ($M^v_t$) & 236.1 & \$K/yr & Derived \\
\quad Battery range per pack & 120 & miles & Assumed; favorable duty cycle (see text) \\
\quad Max battery packs per truck & 6 & packs & \cite{MANBattery} \\
\midrule
\multicolumn{4}{@{}l}{\textbf{Battery Packs}} \\
\quad Capacity per pack & 89 & kWh & \cite{ABBMega, MANBattery} \\
\quad Pack lease rate (NIO BaaS, 100\,kWh) & 1,128 & RMB/mo & \cite{cnevpost2025baas} \\
\quad Annual lease per pack ($M^b_t$) & 1.70 & \$K/yr & Derived (89\,kWh share) \\
\quad Annual cost decline & 8\% & /yr & \cite{bnef2025battery} (2025 survey) \\
\quad Swap time ($T^s$) & 5 & min & \cite{catl2025qiji} \\
\quad Replenishment charge time ($T^c$) & 120 & min/pack & \cite{qi2023scaling}; 89\,kWh at $\sim$45\,kW \\
\quad On-site fast-charge time ($T^{c,\text{fast}}$) & 15 & min/pack & 350\,kW-class high-power CCS \citep{wattev2025mcs} \\
\midrule
\multicolumn{4}{@{}l}{\textbf{Charging Station (CCS)}} \\
\quad Installation cost (DC Fast, 5 chargers) & 515 & \$K & \cite{icf2022charging} \\
\quad Annualized fixed cost ($A^c$) & 73.6 & \$K/yr & Derived (7yr) \\
\quad Annual maintenance ($M^c_t$) & 25.8 & \$K/yr & 5\% of capital \\
\quad Initialization cost ($U^c_t$) & 51.6 & \$K & 10\% of capital \\
\quad Charging capacity cost ($c^{\mathrm{kW}}$) & 68 & \$/kW-yr & Derived from \cite{nicholas2019estimating} \\
\quad Base-site capacity ($\bar C_0$) & 350 & kW & 5$\times$70\,kW per \cite{icf2022charging} \\
\quad Max average utilization ($\rho$) & 0.8 & & Assumed \\
\midrule
\multicolumn{4}{@{}l}{\textbf{Grid power for charging ($\bar P_{ht}$, base year)}} \\
\quad Node $\ge$500\,kV & 40 & MW & Assumed, after \cite{then2025planning} \\
\quad Node $\ge$230\,kV & 20 & MW & Assumed, after \cite{then2025planning} \\
\quad Node $\ge$161\,kV & 15 & MW & Assumed, after \cite{then2025planning} \\
\quad Node $\ge$115\,kV & 10 & MW & \cite{borlaug2021heavy,then2025planning} \\
\quad Node $\ge$69\,kV & 5 & MW & \cite{borlaug2021heavy} \\
\quad Node $<$69\,kV & 2.5 & MW & \cite{then2025planning} \\
\quad Annual growth (T\&D expansion) & 3.8\% & /yr & \cite{doe2023needs} median \\
\midrule
\multicolumn{4}{@{}l}{\textbf{Swapping Station (BSS)}} \\
\quad Installation cost (heavy-duty, modular) & 250 & \$K & \cite{catl2025qiji} \\
\quad Annualized fixed cost ($A^s$) & 35.7 & \$K/yr & Derived (7yr) \\
\quad Annual maintenance ($M^s_t$) & 12.5 & \$K/yr & 5\% of capital \\
\quad Initialization cost ($U^s_t$) & 25.0 & \$K & 10\% of capital \\
\midrule
\multicolumn{4}{@{}l}{\textbf{Operational}} \\
\quad Working days per year ($Y$) & 300 & days & Industry standard \\
\quad Daily operating fraction ($O_c$) & 0.75 & 18 h/day & Industry standard \\
\quad Max transshipment batch size ($\bar{Q}$) & 19 & packs & 10-t truck / 523\,kg pack \\
\quad Battery transport vehicle cost ($\eta_t$) & 0.4 & \$K/day & Medium-duty shuttle, vehicle only, $\approx$\$0.30/mile \\
\quad Loading/unloading at origin--destination hubs ($T^r_{\text{depot}}$) & 60 & min & Assumed \\
\quad Operational floor at intermediate hubs ($T^r_{\text{mid}}$) & 5 & min & Assumed \\
\midrule
\multicolumn{4}{@{}l}{\textbf{Planning}} \\
\quad Diesel-substitution premium ($c^{\mathrm{D}}$) & 250 & \$/truck-day & Derived (see text) \\
\quad Corridor activation threshold ($d_{\min}$) & 5 & trucks/day & Assumed \\
\quad Service level ($1-\varepsilon^C$) & 95\% & & Standard \\
\quad Discount rate & 5\% & /yr & Standard \\
\quad Max BSS-to-CCS driving time & 3.6 & hrs & Network design \\
\quad Intra-facility handling time, co-located pairs ($T^t_{hh}$) & 7.2 & min & Assumed \\
\quad Planning horizon & 2026--2040 & & 15 yearly periods \\
\bottomrule
\end{tabularx}
\end{table}

\paragraph{Derivation of cost and operating parameters.}
\emph{Charging capacity and stations.} The capacity price follows from the ICCT charging-cost decomposition \citep{nicholas2019estimating}: a 350\,kW-class unit costs about \$166K installed at multi-charger sites (\$140K hardware plus \$26K asymptotic per-charger installation), i.e., \$474/kW, or \$68 per kW-year over the seven-year depreciation period. The residual site-level cost in the \$515K station of \citet{icf2022charging} (medium-voltage interconnection, transformers, and site work) is carried by the fixed cost $A^c$, with the included capacity credited through $\bar C_0$.
\emph{Grid budgets.} Each hub's base-year power budget increases with the transmission voltage of its node (Table~\ref{Tab:experiment-parameters}). The magnitudes follow distribution-level studies: new loads above about 5\,MW typically require feeder upgrades and loads above about 12\,MW a dedicated feeder \citep{borlaug2021heavy}, and in a heavy-vehicle charging case an existing network supplies about 2.5\,MW, a dedicated feeder about 10\,MW, and a dedicated substation about three times as much \citep{then2025planning}. Budgets grow at 3.8\% per year, the compounded rate implied by the U.S.\ Department of Energy's National Transmission Needs Study's median finding of 57\% transmission-infrastructure growth by 2035 relative to the early-2020s system \citep{doe2023needs}.

\emph{Trucks and batteries.} The per-truck cost $M^v_t$ combines annualized vehicle capital (\$30.9K) with driver wages and benefits (\$38/hr over $18\,\text{h}\times300$ days, slightly below the \$38.90--40.03/hr reported for 2023--2024 by the American Transportation Research Institute \citep{atri2025operational}). The battery holding cost $M^b_t$ is calibrated to the market rate for battery-as-a-service (BaaS) leasing rather than to depreciation: NIO's BaaS program rents its 100\,kWh swappable pack for RMB~1,128 per month \citep{cnevpost2025baas}, or about \$1.9K per pack-year, which scales to \$1.70K for the 89\,kWh pack used here; the lease rate is assumed to decline at 8\% per year, the pack-price trajectory in BloombergNEF's 2025 survey \citep{bnef2025battery}.

\emph{Reference pack and battery logistics.} The reference pack is a standardized modeling unit rather than a single commercial product: the module size and six-module layout follow MAN-class architectures (534\,kWh installed, comparable to the 513\,kWh full configuration of CATL's QIJI system), while the swap time follows QIJI's chassis-level exchange. Returned swap batteries recharge at about 45\,kW average power. At CATL's module-to-bracket energy density of 170\,Wh/kg \citep{cnevpost2022mtb}, an 89\,kWh pack weighs about 523\,kg, so a 10-ton shuttle carries $\bar Q=10{,}000/523\approx19$ packs per load. The shuttle cost $\eta_t=\$0.4$K per driving-day corresponds to about \$0.30 per mile at 55\,mph; it covers the vehicle operating cost of an in-house electric shuttle but not a driver. Co-located pairs exchange packs by intra-facility handling covered by station O\&M. Under the just-in-time policy, at most three packs are exchanged at any stop across all instances, so the six-pack capacity never binds, and even at heavier energy intensities (about 60 miles per pack) the 534\,kWh vehicle covers every leg in the network, whose longest is 302 miles.

\emph{Diesel-substitution premium.} Under the reference pack's energy use of 0.74\,kWh per mile, diesel fuel at \$0.48 per mile \citep{atri2025operational} against electricity at the 2024 U.S.\ commercial average of 12.8 cents per kWh \citep{eia2025electricity}, with 95\% charging efficiency, gives an energy-cost differential of about \$0.38 per mile, and battery-electric trucks save a further \$0.02--0.06 per mile in maintenance \citep{theicct2023tco}. Over typical long-haul daily distances of 400--500 miles, the operating differential is therefore about \$160--220 per truck-day. Electric operation also avoids about 0.5--0.6\,t of CO$_2$ per truck-day net of grid emissions (10.18\,kg of CO$_2$ per gallon of diesel at 6.6--7 miles per gallon, against 0.35\,t per MWh on the U.S.\ grid; \citealp{epa2022ghg,theicct2023tco,epa2025egrid}), which at \$185 per tonne \citep{rennert2022comprehensive} adds about \$90--120. The value $c^{\mathrm{D}}=\$250$ lies between the operating differential alone (\$160--220) and its sum with the avoided emissions (\$250--340); Section~S4 re-solves SE-5 at \$125 and \$500, which bracket this range.
\emph{Path activation.} On the full network ($|\mathcal P|=249$), the union of the scenario-wise active path sets contains 6 paths in 2026, 62 in 2030 and 215 in 2040.

\section{Batching variance: derivation and simulation protocol}
\label{app:mc-protocol}

\paragraph{Derivation of the batching variance.}
This paragraph derives approximation~(7) of the paper. Let $N_{ht}(u;\omega)$ be the cumulative number of swap requests at $h$ up to time $u$, a stationary counting process calibrated to the one-day demand moments. Propagating the variance linearly to a general interval is a variance-rate approximation, rather than an exact identity for an arbitrary renewal process. The demand over the most recent effective charging interval is represented by
\begin{equation}
Z_{hh't}(\omega) := N_{ht}(u;\omega) - N_{ht}\bigl(u-\Delta_{hh'};\omega\bigr),
\quad
\mathbb E[Z_{hh't}(\omega)] = \mu_{ht}(\omega)\,\Delta_{hh'},
\quad
\mathrm{Var}\bigl(Z_{hh't}(\omega)\bigr) \approx \sigma_{ht}^2(\omega)\,\Delta_{hh'}.
\label{eq:Z-def}
\end{equation}
Let $X_{hh't}(\omega)$ be the number of full batches shipped from $h$ to $h'$ after time $u-\Delta_{hh'}$, so that the \emph{batched deficit} $Q_{hh't}(\omega)X_{hh't}(\omega)$ counts the packs removed from the stock at $h$ and tied up in full batches in transit or charging. For integer batch sizes, it differs from the window count only through the incomplete-batch remainders $\mathrm{rem}_{hh't}:=\operatorname{mod}_{Q_{hh't}(\omega)}(N_{ht}(\cdot;\omega))$ at the two endpoints:
\begin{equation}
Q_{hh't}(\omega)X_{hh't}(\omega)
\;=\;
Z_{hh't}(\omega) + \mathrm{rem}_{hh't}(u-\Delta_{hh'}) - \mathrm{rem}_{hh't}(u).
\label{eq:Y-def}
\end{equation}
In steady state each endpoint remainder is approximately uniform on $\{0,\ldots,Q_{hh't}(\omega)-1\}$, with variance $(Q^2_{hh't}(\omega)-1)/12$. Treating the two remainders as independent of each other and of $Z_{hh't}(\omega)$, as in \citet{qi2023scaling}, and allowing $Q$ to be continuous gives the smooth planning approximation
\[
  \mathrm{Var}\bigl(Q_{hh't}(\omega)X_{hh't}(\omega)\bigr)
    \;\approx\; \mathrm{Var}\bigl(Z_{hh't}(\omega)\bigr) + \frac{Q_{hh't}^2(\omega)-1}{6}
    \;\approx\; \Delta_{hh'}\,\sigma_{ht}^2(\omega) + \frac{Q_{hh't}^2(\omega)-1}{6},
\]
which is~(7) of the paper; its first term reflects arrival variability over the interval and its second reflects batching. If $z_{hh't}=1$ and $\gamma\mu_{ht}(\omega)\Delta_{hh'}<1$, the small-batch bound~(6) of the paper fixes $Q_{hh't}(\omega)=1$. The remainder is then identically zero, so the batching-variance term vanishes and the remainder approximation is not needed; the variance-rate approximation of the window demand still applies.

\paragraph{Simulation protocol.}
This section describes the simulation protocol behind the calibration of the small-batch parameter $\gamma=1$ in (6); the results follow it. The protocol follows the validation experiment of \citet[Sec.~4.2]{qi2023scaling}, with the parameters changed to our regime. For each cell, $0.5$ million swap arrivals are generated from a renewal process whose inter-arrival times follow a normal distribution truncated at zero with count-process dispersion $\sigma^2/\mu = 1.175$ (Qi et al.'s value, slightly over-dispersed relative to Poisson; the truncation lowers the effective window mean, by about a quarter at the nominal values used here, so all formulas are evaluated at the empirical process moments, which are the values shown in Figure~\ref{fig:mc-validation}). A sliding window of length $\Delta$ is rolled over the arrival epochs in steady state, recording the window count $Z(t)$ and the batched deficit $QX(t) = Q\left(\lfloor N(t)/Q\rfloor - \lfloor N(t-\Delta)/Q\rfloor\right)$, where $N(t)$ is the cumulative arrival counter; reorders fire whenever $N(t)$ crosses a multiple of $Q$. The simulated variance of $QX$ is compared against the small-batch segment $\Delta\sigma^2 + (Q^2-1)/6$ of the piecewise variance in \citet[Lemma~2]{qi2023scaling}, which we use in (7), and against the full piecewise $\phi(Q)$ of \citet{qi2023scaling}, whose two segments cross at $\nu = 3\Delta\mu - \sqrt{3(\Delta\mu)^2 - 6\Delta\sigma^2 + 1}$. Relative to Qi et al.'s single setting ($\mu\Delta = 20$, $\Delta = 2.5$\,h), we set $\Delta = 2.2$\,h (the $T^c = 120$\,min replenishment pace plus $\sim$10\,min transport) and sweep corridor flows $\mathbb{E}[Z] \in \{2, 5, 10, 30, 60\}$ with batch ratios $Q/\mathbb{E}[Z]$ from $0.3$ to $2$, bracketing the candidate $\gamma$ values. Day-to-day rate variability (the $\mathrm{CV}=1/3$ law for $\beta$) is deliberately \emph{not} injected into the simulated arrival process: it enters the model through the separate demand-variance term $\Delta\sigma^2$, whereas this experiment isolates the batching remainder.

\begin{figure}[!htbp]
\centering
\includegraphics[width=\textwidth]{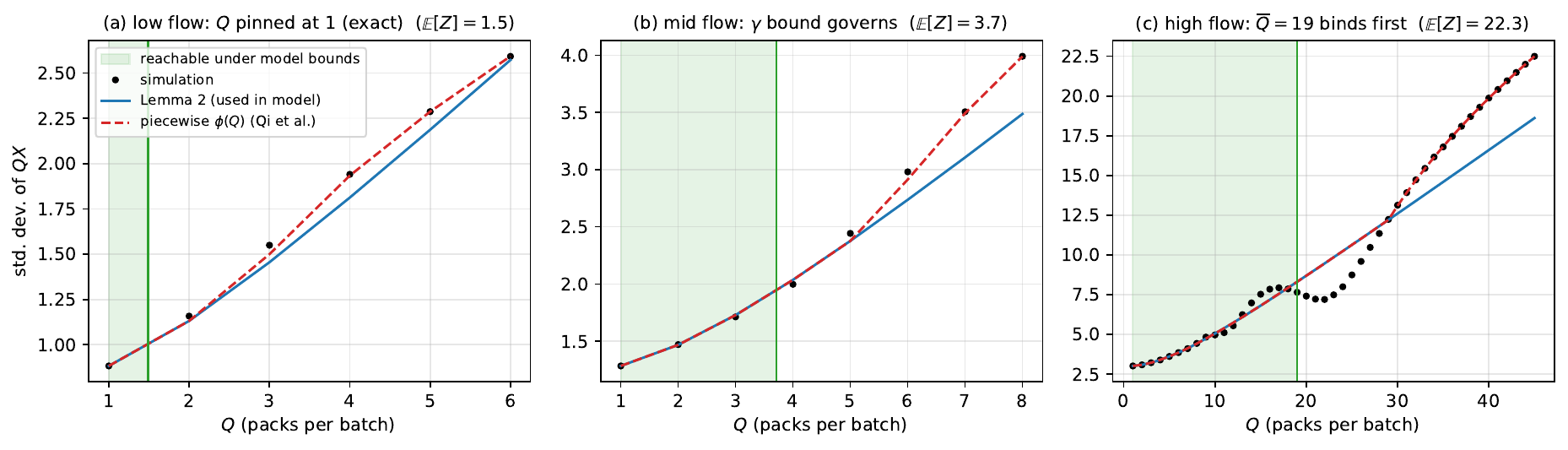}
\caption{Simulated versus approximated standard deviation of the batched deficit on three corridor classes. The shaded band is the batch-size region reachable under the model's constraints, $Q \le \min\{\overline{Q}, \max\{1, \gamma\,\mathbb{E}[Z]\}\}$.}
\label{fig:mc-validation}
\end{figure}

\paragraph{Results.} Figure~\ref{fig:mc-validation} shows that a different mechanism confines each corridor class to the region where~(7) is accurate or conservative. On low-flow corridors (panel a), essentially only unit batches are reachable, for which the expression is exact. On mid-flow corridors (panel b), the small-batch bound binds: within the reachable band the expression tracks the simulation closely, whereas beyond $Q=\gamma\,\mathbb E[Z]$ it would understate the variance by up to 23\%. On high-flow corridors (panel c), the batch capacity $\overline Q=19$ binds first, and the remaining deviations overestimate the variance, which inflates the safety stock. The choice $\gamma=1$ is thus the largest round value that keeps every corridor class on the conservative side. At $Q/\mathbb E[Z]=1.5$ the understatement on mid-flow corridors ranges from 5\% to 23\%, and on high-flow corridors the overestimate reaches 71\%. The choice $\gamma=1$ also lies below the crossing point $\nu/\mathbb E[Z]\approx1.4$--1.5 of the two segments of the piecewise variance of \citet{qi2023scaling}.
\section{Scenario-fan stability of the SAA approximation}
\label{app:saa-stability}

Table~\ref{tab:saa-stability} reports the in-sample and out-of-sample stability test underlying the choice of $|\Omega|=50$ scenarios (Section~6.1.5). For each fan size $N \in \{5, 10, 30, 50, 80, 100\}$ and each of five independent LHS replications, the model is solved on SE-3 and the resulting first-stage plan is re-evaluated, with all first-stage decisions fixed, on an independent reference set of $1{,}000$ LHS scenarios. The out-of-sample evaluation applies the same cost components used in training, augmented with a remediation rule that the training model excludes by construction: when the realized charging load in an evaluation scenario exceeds the installed capacity (constraint~(13) is enforced as a hard constraint in training), the excess energy is assumed to revert to diesel operation and is priced at the same substitution premium $c^{\mathrm{D}}$ used in the objective, converted at $E^{\text{pack}}$ and two packs per truck-stop, which approximates the average requirement of 2.2 packs across path--hub stops with positive replenishment. The capacity-exceedance column of Table~\ref{tab:saa-stability} reports, for each replication, the frequency of evaluation scenarios in which this rule is invoked; the discussion after the table relates them to the $1/(N+1)$ level of a fixed scalar load.

\begin{table}[!htbp]
\centering
\caption{In-sample and out-of-sample stability of the SAA approximation (SE-3, $\gamma=1$, $c^{\mathrm{D}}=\$250$; five independent LHS replications per fan size; evaluation on an independent reference set of $1{,}000$ scenarios). In-sample stability: $(\max-\min)/\text{mean}$ of the five in-sample objective values. Out-of-sample stability: mean over replications of $|\text{out-of-sample} - \text{in-sample}|/\text{out-of-sample}$. Capacity exceedance: fraction of evaluation scenarios in which the realized charging load exceeds installed capacity at some hub--period. Costs in \$K (15-year NPV).}
\label{tab:saa-stability}
\small
\setlength{\tabcolsep}{4pt}
\begin{tabular}{cc rrr cc}
\toprule
Scenarios & Run & \makecell{In-sample\\objective value} & \makecell{Out-of-sample\\expected cost} & \makecell{Capacity\\exceedance (\%)} & \makecell{In-sample stability\\{[}(Max$-$Min)/Avg{]} (\%)} & \makecell{Out-of-sample stability\\{[}mean rel.\ diff.{]} (\%)} \\
\midrule
\multirow{5}{*}{5}   & 1 & 54{,}190 & 82{,}721 & 17.9 & \multirow{5}{*}{14.46} & \multirow{5}{*}{26.69} \\
                     & 2 & 56{,}495 & 69{,}109 & 13.6 & & \\
                     & 3 & 57{,}709 & 66{,}045 & 6.1 & & \\
                     & 4 & 51{,}258 & 123{,}480 & 36.9 & & \\
                     & 5 & 59{,}327 & 65{,}624 & 9.0 & & \\
\midrule
\multirow{5}{*}{10}  & 1 & 57{,}089 & 66{,}800 & 6.5 & \multirow{5}{*}{14.11} & \multirow{5}{*}{15.74} \\
                     & 2 & 57{,}439 & 65{,}093 & 8.5 & & \\
                     & 3 & 55{,}067 & 75{,}026 & 13.8 & & \\
                     & 4 & 55{,}329 & 73{,}233 & 19.5 & & \\
                     & 5 & 63{,}195 & 64{,}065 & 3.7 & & \\
\midrule
\multirow{5}{*}{30}  & 1 & 58{,}072 & 63{,}839 & 3.5 & \multirow{5}{*}{6.43} & \multirow{5}{*}{7.96} \\
                     & 2 & 57{,}729 & 64{,}086 & 4.4 & & \\
                     & 3 & 58{,}386 & 63{,}431 & 3.4 & & \\
                     & 4 & 61{,}504 & 63{,}526 & 2.5 & & \\
                     & 5 & 57{,}851 & 64{,}067 & 5.3 & & \\
\midrule
\multirow{5}{*}{50}  & 1 & 60{,}789 & 63{,}803 & 2.5 & \multirow{5}{*}{7.64} & \multirow{5}{*}{3.49} \\
                     & 2 & 62{,}331 & 63{,}816 & 2.7 & & \\
                     & 3 & 58{,}829 & 63{,}878 & 4.3 & & \\
                     & 4 & 63{,}550 & 64{,}258 & 1.0 & & \\
                     & 5 & 63{,}346 & 64{,}231 & 0.9 & & \\
\midrule
\multirow{5}{*}{80}  & 1 & 61{,}666 & 63{,}516 & 1.9 & \multirow{5}{*}{11.61} & \multirow{5}{*}{1.38} \\
                     & 2 & 61{,}934 & 63{,}659 & 2.1 & & \\
                     & 3 & 63{,}940 & 64{,}576 & 1.3 & & \\
                     & 4 & 67{,}521 & 67{,}582 & 0.6 & & \\
                     & 5 & 69{,}193 & 69{,}052 & 0.0 & & \\
\midrule
\multirow{5}{*}{100} & 1 & 65{,}723 & 65{,}919 & 0.7 & \multirow{5}{*}{14.68} & \multirow{5}{*}{1.01} \\
                     & 2 & 62{,}844 & 63{,}826 & 0.8 & & \\
                     & 3 & 62{,}127 & 63{,}623 & 1.6 & & \\
                     & 4 & 64{,}314 & 64{,}820 & 0.8 & & \\
                     & 5 & 71{,}717 & 71{,}644 & 0.3 & & \\
\bottomrule
\end{tabular}
\end{table}

\emph{Settings and replication ranges.} Each replication is solved by LcBD+ to a 0.5\% tolerance within a one-hour limit, and all 30 runs close to a certified gap of at most 0.47\%. At $N=50$, the capacity-exceedance frequency ranges from 0.9\% to 4.3\% across the five replications and the out-of-sample difference from 1.1\% to 7.9\%; without the remediation charge the difference is at most 0.14\%. Five replications per fan size support a sensitivity assessment rather than a precise statistical certification.
\emph{Protection level.} Because capacity is scenario-feasible, $N$ affects both the expected-cost approximation and the range of training loads that installed capacity must cover. A fixed scalar load sized to the maximum of $N$ independent continuous observations has an expected fresh-sample exceedance frequency of $1/(N+1)$. That order-statistic identity is not a formal guarantee for the present design, which is optimized using dependent LHS samples and contains integer decisions; the convex scenario-program results of \citet{campi2008exact} are not invoked to certify it, and the protection level is therefore assessed through the independent evaluation set. Dependence between hub loads can cause exceedances to occur together. Sample maxima also drive the in-sample variability at large $N$, but bounded demand distributions do not imply that such maxima fail to concentrate as $N$ grows.

The out-of-sample differences of Section~6.1.5 are empirical policy-evaluation discrepancies, not optimality-gap estimates: the lower-bound property of the mean SAA optimum \citep{mak1999monte,kleywegt2002sample} requires the same expected-cost problem in training and evaluation, whereas training enforces the throughput constraint in every training scenario and evaluation adds a remediation charge. The independent $1{,}000$-scenario evaluation of a fixed plan is an unbiased estimate of that plan's cost under the evaluation rule. Each in-sample value is a certified upper bound within 0.47\% of its SAA optimum, so using it in place of the optimum changes each discrepancy by at most 0.47 percentage points. In-sample instability persists for large fans because capacity is sized against the most demanding training scenarios: one $N=100$ replication draws an unusually heavy tail and prices a plan roughly 13\% more expensive in-sample, while its out-of-sample difference remains at 0.1\%; the plan is more expensive, not less accurate. Increasing $N$ beyond 50 therefore mainly tightens the protection level: the five $N=80$ plans disperse more on the common reference set (out-of-sample range 63.5--69.1) than the five $N=50$ plans (63.8--64.3), whose deployed 2040 configurations essentially coincide. Finally, hub loads are nondecreasing in both scenario parameters (differentiating (39) shows that $\varphi_t$ is nondecreasing in the adoption parameter for every $t$) and both parameters have truncated support, so a fully support-robust capacity variant obtains by enforcing (13) at the single corner scenario of the support; we retain the sample-based construction, whose protection level is explicit and tunable through $N$.

\section{Robustness to energy intensity, shuttle cost, diesel premium, spatial adoption, service level, and irreversibility}
\label{app:energy-robustness}

The reference pack corresponds to about 0.74\,kWh per mile. To test the planning insights against heavier duty cycles, SE-5 is re-solved at 1.2 and 1.5\,kWh per mile, which reduce the range of the 89\,kWh pack to 74 and 59 miles. All other parameters, the 50 training scenarios, and the solution protocol are those of the main study, so only the per-stop pack requirement $R_{ph}$ changes; the mean requirement over stops with positive replenishment rises from 2.2 to 3.3 and 3.9 packs, and the six-pack capacity is reached but never exceeded, so every leg remains feasible. Both instances are solved to a gap below $0.01\%$. Table~\ref{tab:energy-robustness} compares the plans.

\begin{table}[!htbp]
\centering
\caption{SE-5 plans under alternative energy intensities ($|\Omega|=50$; costs in \$K, 15-year NPV; all runs end with a gap below $0.01\%$).}
\label{tab:energy-robustness}
\small
\begin{tabular}{l rrr}
\toprule
Energy intensity (kWh/mile) & 0.74 (main) & 1.2 & 1.5 \\
Range per pack (miles) & 120 & 74 & 59 \\
\midrule
Total cost & 119{,}919.3 & 178{,}187.3 & 230{,}846.3 \\
\quad Infrastructure & 53{,}130.5 & 53{,}492.2 & 56{,}307.3 \\
\quad Charging capacity & 49{,}598.8 & 75{,}434.1 & 90{,}707.9 \\
\quad Battery & 15{,}132.8 & 27{,}274.7 & 37{,}222.4 \\
\quad Fleet & 1{,}508.6 & 1{,}347.2 & 3{,}908.8 \\
\quad Diesel substitution & 548.6 & 20{,}639.0 & 42{,}700.0 \\
\midrule
Station hubs in 2026 / 2035 / 2040 & 9 / 66 / 73 & 9 / 66 / 73 & 9 / 67 / 75 \\
Installed capacity in 2035 / 2040 (MW) & 100.1 / 334.9 & 146.3 / 473.2 & 175.2 / 560.8 \\
2040 co-located / charging-only / swap-only & 20 / 25 / 28 & 24 / 20 / 29 & 34 / 14 / 27 \\
2040 electrified share of active demand & 100.0\% & 95.4\% & 94.7\% \\
Replenishment within one hour & 89.2\% & 93.5\% & 94.4\% \\
Longest assignment (hours) & 2.8 & 2.6 & 2.6 \\
Overlap with main plan, 2040 BSS / CCS sites & --- & 0.84 / 0.82 & 0.73 / 0.63 \\
\bottomrule
\end{tabular}
\end{table}

Three features of the main solution are unchanged. The deployment sequence is essentially identical: the station footprint follows the same trajectory, and installed capacity grows by a factor of about 3.2 to 3.3 between 2035 and 2040 in every case, so coverage still precedes capacity intensification. Replenishment remains predominantly local and becomes slightly more so, with the share completed within one hour rising from 89.2\% to 94.4\%. All three station roles persist.

Two features change. First, the station mix shifts toward co-location. Because on-site charging dwell scales with the number of packs a truck takes, heavier energy use raises the fleet cost of charging trucks on site, and the model responds by adding swapping stations at hubs that previously offered charging only; co-located hubs rise from 20 to 34 and charging-only hubs fall from 25 to 14. These new swapping stations are mostly replenished on site, since the shuttle and in-transit inventory costs of remote replenishment grow with battery throughput while co-located replenishment avoids both; the number of remote assignments stays between 27 and 29.

Second, part of the 2040 demand reverts to diesel, and this effect is local rather than system-wide. Installed capacity uses only 17\% to 21\% of the network's aggregate grid budget, and at most two charging hubs reach their individual budgets. All corridors that are not fully electrified pass through a single high-volume hub on a 115\,kV node with a 16.9\,MW budget in 2040. In the main plan, this hub's maximum swap load across the training scenarios, 23.0\,MW, exceeds the local budget but fits within a remote 230\,kV hub 2.8 hours away, which is the remote assignment discussed in Section~6.4. At 1.2 and 1.5\,kWh per mile, the load grows to 33.9 and 42.4\,MW and exceeds the headroom of every single charging hub within the assignment radius, the largest of which is 33.7\,MW. Because each swapping hub is replenished by exactly one charging hub, the load cannot be divided, so the model charges trucks on site up to the local budget and leaves the remaining demand on diesel. The reach of centralized charging is therefore bounded by the largest grid connection within the assignment radius rather than by aggregate grid capacity.

Table~\ref{tab:shuttle-pi-robustness} re-solves SE-5 under alternative shuttle costs and diesel premiums, again with all other parameters, scenarios, and the solution protocol of the main study.

\begin{table}[!htbp]
\centering
\caption{SE-5 plans under alternative shuttle costs $\eta$ (\$K per driving-day) and diesel premiums $c^{\mathrm{D}}$ (\$ per truck-day); the main plan uses $\eta=0.4$ and $c^{\mathrm{D}}=250$ ($|\Omega|=50$; costs in \$K, 15-year NPV; all runs end with a gap below $0.01\%$).}
\label{tab:shuttle-pi-robustness}
\footnotesize
\setlength{\tabcolsep}{4pt}
\begin{tabular}{l rrrrr}
\toprule
 & & \multicolumn{2}{c}{Shuttle cost $\eta$} & \multicolumn{2}{c}{Diesel premium $c^{\mathrm{D}}$} \\
\cmidrule(lr){3-4}\cmidrule(lr){5-6}
 & Main & 1.31 & 2.40 & 125 & 500 \\
\midrule
Total cost & 119{,}919.3 & 130{,}952.3 & 135{,}289.7 & 119{,}595.5 & 120{,}274.2 \\
\quad Infrastructure & 53{,}130.5 & 61{,}248.2 & 64{,}642.8 & 52{,}897.0 & 53{,}494.7 \\
\quad Charging capacity & 49{,}598.8 & 47{,}852.3 & 46{,}977.5 & 49{,}531.5 & 49{,}637.3 \\
\quad Battery & 15{,}132.8 & 15{,}074.7 & 6{,}781.0 & 15{,}097.7 & 15{,}188.9 \\
\quad Fleet & 1{,}508.6 & 1{,}896.9 & 2{,}382.5 & 1{,}582.3 & 1{,}485.7 \\
\quad Diesel substitution & 548.6 & 4{,}880.3 & 14{,}505.9 & 487.1 & 467.7 \\
\midrule
Station hubs in 2026 / 2035 / 2040 & 9 / 66 / 73 & 9 / 66 / 73 & 9 / 65 / 73 & 9 / 65 / 72 & 11 / 66 / 73 \\
Installed capacity in 2040 (MW) & 334.9 & 334.9 & 328.7 & 334.6 & 334.9 \\
2040 co-located / charging-only / swap-only & 20 / 25 / 28 & 33 / 25 / 15 & 39 / 28 / 6 & 20 / 25 / 27 & 20 / 25 / 28 \\
Remote assignments in 2040 & 28 & 15 & 6 & 27 & 28 \\
Co-located share of battery demand & 64.4\% & 83.6\% & 99.6\% & 64.4\% & 64.4\% \\
Longest assignment (hours) & 2.8 & 2.8 & 2.3 & 2.8 & 2.8 \\
2040 electrified share of active demand & 100.0\% & 100.0\% & 97.8\% & 100.0\% & 100.0\% \\
Overlap with main plan, 2040 BSS / CCS & --- & 0.88 / 0.69 & 0.86 / 0.62 & 0.98 / 1.00 & 1.00 / 1.00 \\
\bottomrule
\end{tabular}
\end{table}

\emph{Shuttle cost.} The reference shuttle cost $\eta=\$0.4$K per driving-day covers the vehicle's operating cost only. Adding a dedicated driver at \$38/hr raises it to \$1.31K (about \$0.99 per mile), and the third-party rate of \citet{qi2023scaling} corresponds to \$2.40K. Total cost rises by 9.2\% and 12.8\%. The footprint and capacity trajectory are unchanged, but replenishment shifts toward co-location: remote assignments fall from 28 to 15 and 6, co-located hubs rise from 20 to 33 and 39, and swap-only hubs fall from 28 to 15 and 6. With a dedicated driver, the 2.8-hour remote assignment of Section~6.4 is retained; at the third-party rate it is abandoned, the hub charges trucks on site up to its local grid budget, and 2.2\% of 2040 demand reverts to diesel, the same mechanism as under heavier energy use.

\emph{Diesel premium.} Halving or doubling $c^{\mathrm{D}}$ changes total cost by $-0.3\%$ and $+0.3\%$ and leaves the plan essentially unchanged: the 2040 charging sites coincide with the main plan, the swapping sites differ by at most one hub, and diesel substitution stays below 0.5\% of total cost. Within this range the premium affects the timing of early deployment rather than the design.

\emph{Spatially heterogeneous adoption.} The main scenarios vary the speed of electrification through two common drivers, so adoption is spatially uniform within a scenario. To test whether the findings depend on this, we keep each scenario's growth rate $g_\omega$ and adoption parameter $q_\omega$ and redistribute adoption across the seven states. State $s$ receives the parameter $q_\omega m_{s\omega}$, clipped to $[0.03,0.90]$, with the mean-preserving multiplier $m_{s\omega}=e^{\epsilon_{s\omega}}/\sum_{s'}\pi_{s'}e^{\epsilon_{s'\omega}}$, where $\epsilon_{s\omega}\sim\mathcal N(0,\sigma^2)$ is truncated at $\pm2\sigma$ and drawn by an independent Latin hypercube, and $\pi_s$ is the state's share of base-year freight (each origin--destination flow split equally between its two end states). Each state follows the logistic curve of Section~6.1.2, an origin--destination pair uses the mean of its two end states' curves, and each scenario--year is rescaled so that network-wide electric demand equals that of the main scenario set; only the spatial allocation changes. For $\sigma=0.25$ and $0.5$, the state multipliers range from 0.53 to 1.77 and from 0.26 to 2.80. Table~\ref{tab:spatial-adoption} reports the re-solved plans; its last row evaluates the main plan, with all first-stage decisions fixed, on the heterogeneous scenarios.

\begin{table}[!htbp]
\centering
\caption{SE-5 plans under spatially heterogeneous adoption ($|\Omega|=50$; network-wide electric demand of every scenario--year equal to the main scenario set; costs in \$K, 15-year NPV; all runs end with a gap below $0.01\%$).}
\label{tab:spatial-adoption}
\small
\begin{tabular}{l rrr}
\toprule
State dispersion $\sigma$ & 0 (main) & 0.25 & 0.50 \\
\midrule
Total cost & 119{,}919.3 & 122{,}709.5 & 128{,}205.4 \\
\quad Infrastructure & 53{,}130.5 & 52{,}969.1 & 52{,}727.6 \\
\quad Charging capacity & 49{,}598.8 & 52{,}046.0 & 55{,}975.2 \\
\quad Battery & 15{,}132.8 & 15{,}314.3 & 17{,}505.0 \\
\quad Fleet & 1{,}508.6 & 1{,}444.8 & 1{,}378.7 \\
\quad Diesel substitution & 548.6 & 935.3 & 618.9 \\
\midrule
Station hubs in 2026 / 2035 / 2040 & 9 / 66 / 73 & 9 / 66 / 73 & 9 / 66 / 73 \\
Installed capacity in 2035 / 2040 (MW) & 100.1 / 334.9 & 102.3 / 356.3 & 106.1 / 388.4 \\
2040 co-located / charging-only / swap-only & 20 / 25 / 28 & 20 / 24 / 29 & 20 / 23 / 30 \\
Remote assignments in 2040 & 28 & 29 & 30 \\
Co-located share of battery demand & 64.4\% & 64.4\% & 62.4\% \\
Replenishment within one hour & 89.2\% & 89.2\% & 89.4\% \\
Longest assignment (hours) & 2.8 & 2.8 & 3.1 \\
Overlap with main plan, 2040 BSS / CCS & --- & 0.98 / 0.98 & 0.96 / 0.91 \\
\midrule
Main plan: scenarios with a capacity breach & 0\% & 14\% & 30\% \\
\bottomrule
\end{tabular}
\end{table}

The planning insights are unchanged: the station footprint follows the same trajectory, all three station roles persist, and replenishment remains predominantly local. Spatial dispersion raises the capacity requirement, because each hub must cover its own largest scenario load although the network total is unchanged: installed 2040 capacity rises by 6\% and 16\%, and total cost by 2.3\% and 6.9\%. Dispersion also strengthens the separation of swapping from charging: swap-only hubs and remote assignments increase from 28 to 30, and the longest assignment from 2.8 to 3.1 hours, as charging moves away from hubs whose local loads become more variable. The main plan, sized for spatially uniform adoption, breaches installed capacity in 14\% and 30\% of the heterogeneous scenarios. Its capacity protection therefore holds with respect to the scenario generator used in training, and spatially differentiated adoption should be represented in the training scenarios when it is expected.

\emph{Service level and irreversible investment.} Re-solving SE-5 at $1-\varepsilon^C=0.99$ (gap below $0.01\%$) gives a total cost of 121{,}641.9 (+1.4\%), with battery cost +11.4\% and fleet cost +7.4\%; the 2040 footprint keeps 73 station hubs, the BSS and CCS sets overlap by 98\%, and one charging-only hub becomes swap-only. Repricing the reported plan at 0.99 separates the two roles of the quantile: battery cost, which depends only on the cycle-service target, rises by 10.0\%, and fleet cost, which depends only on the offered-load margin, by 14.7\%; the repriced plan is within 0.02\% of the re-optimized one. The reported plan closes four charging and eleven swapping stations before 2040 and reduces capacity at eight hub--periods, which the annualized costs of leased or contracted equipment permit. Imposing $x^c_{ht}\ge x^c_{h,t-1}$, $x^s_{ht}\ge x^s_{h,t-1}$ and $C_{ht}\ge C_{h,t-1}$ (gap below $0.01\%$) raises total cost to 121{,}917.3 (+1.7\%), mainly through charging capacity (+1{,}405) and infrastructure (+358). The station footprint remains 9/66/73 in 2026/2035/2040; the 2040 co-located/charging-only/swap-only mix becomes 23/21/29 instead of 20/25/28, installed 2040 capacity 350.0 instead of 334.9\,MW, and the co-located share of battery demand 69.4\% instead of 64.4\%. The coverage-then-capacity sequence and the predominance of local replenishment are therefore unchanged under irreversibility.

\section{Service validation of the deployed battery pools}
\label{app:pool-service}

Section~\ref{app:mc-protocol} validates the batching remainder of a single demand stream, and Appendix~B of the paper summarizes the simulation detailed here. This section simulates the composed system of Section~3.2.1 for the deployed network: all 32 charging hubs that replenish at least one swapping hub in the 2040 SE-5 plan, in the five scenarios at the 10th, 30th, 50th, 70th and 90th percentiles of 2040 pool demand, which gives 227 hub--scenario cases with positive demand. In this plan, 20 charging hubs serve a single swapping hub, 8 serve two, and 4 serve three.

\emph{Stock and policy.} Each pool holds $\lceil\Psi_{t,h'}\rceil$ packs, with~(8) evaluated at the integer batch sizes $Q_h=\allowbreak\max\{1,\allowbreak\mathrm{round}(Q^*_h)\}$, where $Q^*$ solves the battery subproblem. Following the decomposition of~(8), each remote swapping hub keeps on its shelf the expected return-leg pipeline, the unpooled buffer and the batching allowance, $\lceil\mu_hT^t_{hh'}+\Phi^{-1}(1-\varepsilon^C)\sqrt{T^t_{hh'}}\,\sigma_h+Q_h-1\rceil$, but at least one truck's largest request $\max_pR_{ph}$, and the charging hub holds the remainder, i.e., the outbound and charging pipeline and the pooled buffer. A co-located swapping hub draws directly on the charging stock of its site. A swap takes the truck's $R_{ph}$ packs from the shelf and leaves the same number of depleted packs; whenever $Q_h$ depleted packs have accumulated, a batch travels to the charging hub, each pack charges for $T^c$, and the charging hub returns $Q_h$ charged packs as soon as it holds them, first come, first served. Requests that find an empty shelf are backordered. A second allocation keeps, with the same pool total, one additional truck request ($\max_pR_{ph}-1$ packs) on each swapping shelf, because a request of two or three packs can overshoot the reorder point.

\emph{Arrivals.} Each scenario-active relay path that requires packs at a pool hub carries one truck stream, which reaches every pool hub on the path at its cumulative driving-time offset (common stream); a control run gives each path--hub pair its own stream (independent streams). The base arrival model follows the interpretation of~(1) as day-to-day variability: the daily rate is gamma distributed with mean $\alpha_{pt}$ and variance $\max\{\beta^2_{pt}-\alpha_{pt},0\}$, and trucks arrive as a Poisson process within the day, so the daily count has variance $\max\{\beta_{pt}^2,\alpha_{pt}\}$. A stress test instead uses a renewal process with gamma inter-arrival times whose squared coefficient of variation is $\beta^2_{pt}/\alpha_{pt}$, so that the variance-rate approximation $\beta^2_{pt}\Delta$ also holds within the day. Each case is simulated for 400 working days after a 30-day warm-up.

\emph{Measures.} The cycle service level is the share of replenishment cycles at a swapping hub, between successive batch receipts, that end without a backordered request, which is the per-cycle, per-location measure targeted by~(4); the pack fill rate is the share of packs supplied immediately from the shelf.

Table~B.1 of the paper reports the results. Shared-path dependence has a negligible effect: on the pools with a shared path, common streams change cycle service by less than 0.1 percentage points on average relative to independent streams, and by at most 1.6 points at any hub, in line with the covariance bound of Section~\ref{app:supp-pool-cov}. Under day-to-day variability, the pooled stock meets the target at every co-located hub and, with the overshoot allowance, at every remote hub with more than 150 packs per day; the target is met at 89\% of the remote hubs with 50 to 150 packs per day and at 19\% of those below 50. The margin at high-volume hubs reflects the variance-rate approximation: over windows of two to five hours, the simulated demand variance is 14\% to 50\% of $\Delta\sigma^2$, and the window demand stays below the normal 95\% quantile in 98.7\% to 100\% of windows (five representative hubs in the median scenario), because day-to-day rate variability materializes only partly within a few hours. At thin hubs this margin vanishes: a single truck takes two or three packs, which can overshoot the reorder point by up to one request, and the Poisson count variance exceeds $\beta^2$ on paths with fewer than nine trucks per day; at a hub with 19 packs per day, the simulated window variance is 149\% of $\Delta\sigma^2$ and the normal quantile covers 90.8\% of windows.

The stress test separates the two buffers of~(8). For co-located hubs, whose service depends on the pooled buffer alone, cycle service remains close to the target (97\%): the variance of window demand is 82\% to 90\% of $\Delta\sigma^2$, and the window demand falls below the normal 95\% quantile in 94.0\% to 95.6\% of windows. Remote hubs reach 83\% to 86\%, because a cycle is stockout-free only if both the pooled buffer at the charging hub and the local return-leg buffer suffice, and each is sized at the 95\% quantile; the charging hub dispatches 88\% to 90\% of remote requests at once. The inequality in Section~3.2.1, which shows that splitting a buffer does not reduce the total stock relative to a single round-trip buffer, therefore does not carry over to the per-hub service event of a two-location policy when variability accrues within the day. Sizing the two buffers of remote links jointly, for example each at $\Phi^{-1}(1-\varepsilon^C/2)$, would address this case but changes~(8) and every reported plan, and is left for future work.

\section{Extensive form without the exact reformulations}
\label{app:ext-original}

Table~\ref{tab:ext-original} reports the extensive form as first written. Assignment, batch and dispatch variables are defined on all hub pairs, with off-link variables fixed at zero; zero-demand links are retained, the pooled-variance row is kept in its nonconvex quadratic form, and fleet variables are defined per scenario. Gurobi solves this formulation as a nonconvex mixed-integer quadratic program, with 3.2 million rows and 19.2 million columns on SE-5. The runs use the 3-hour limit and $0.01\%$ target of Section~6.2, with 12 threads on SE-1 to SE-3, which favors this formulation relative to the 4-thread runs of Section~6.2, and, for memory reasons, 4 threads on SE-4 and SE-5. The formulation reaches 1\% on the three smaller instances but 0.1\% on none. On SE-4 and SE-5, branch and bound does not leave the root, the incumbents are about twice the optimum, and the gap after three hours exceeds 50\%. With the same 4 threads, the strengthened formulation of Section~6.2 ends SE-4 at a gap of $0.92\%$. On SE-5 the strengthened formulation was run with a 33\,GB soft memory limit and approximate minimum-degree ordering in the barrier method, because with the default ordering the root relaxation exceeded the memory limit before the first barrier iteration; the barrier method stalled with numerical difficulties after 152 iterations.

\begin{table}[!htbp]
\centering
\caption{Extensive form without the exact reformulations ($|\Omega|=50$, 3-hour limit; 12 threads on SE-1 to SE-3 and 4 threads on SE-4 and SE-5; costs in \$K, 15-year NPV).}
\label{tab:ext-original}
\small
\begin{tabular}{l rrrr}
\toprule
Instance & t$\,\to$1\% & Final gap & Best UB & Lower bound \\
\midrule
SE-1 & 76s & 0.42\% & 14{,}334.4 & 14{,}274.6 \\
SE-2 & 10.7m & 0.70\% & 32{,}648.5 & 32{,}420.9 \\
SE-3 & 2.14h & 0.47\% & 57{,}776.8 & 57{,}505.1 \\
SE-4 & --- & 50.1\% & 167{,}736.6 & 83{,}771.5 \\
SE-5 & --- & 52.7\% & 234{,}064.9 & 110{,}789.2 \\
\bottomrule
\end{tabular}
\end{table}

\section{Derivations for the capacity model}
\label{app:supp-derivations}
This section gives the derivations behind two relationships of Section~3 of the paper: the covariance between hub demand streams that the pooled battery buffer omits, and the linear constraints and subproblem that represent the fleet requirement.

\subsection{Covariance between hub demand streams}
\label{app:supp-pool-cov}
A path that visits two hubs $h$ and $k$ assigned to the same charging hub $h'$ makes their demands dependent. Let $\tau^p_{hk}$ denote the signed travel time from $h$ to $k$ along $p$, negative if $k$ precedes $h$, so that the trucks counted at $k$ during its window passed $h$ a time $\tau^p_{hk}$ earlier. Arrivals on different paths remain conditionally independent, so the covariance of the pack demands of $h$ and $k$ over their replenishment windows is a sum of shared-path terms. Given the marginal variances of~(7), the Cauchy--Schwarz inequality bounds each term by $\beta^2_{pt}(\omega)R_{ph}(\omega)R_{pk}(\omega)\sqrt{\Delta_{hh'}\Delta_{kh'}}$, whatever the temporal dependence along the path. Under the variance-rate approximation with independent increments and one-day variance $\beta^2_{pt}(\omega)$, the covariance is
\begin{equation}
\begin{aligned}
\mathrm{Cov}_{hk,h't}(\omega)&=\sum_{p\in P(h)\cap P(k)\cap\mathcal P_t(\omega)}\beta^2_{pt}(\omega)R_{ph}(\omega)R_{pk}(\omega)\,\ell^{p,h'}_{hk},\\
\ell^{p,h'}_{hk}&=\Bigl[\min\{0,-\tau^p_{hk}\}-\max\{-\Delta_{hh'},\,-\Delta_{kh'}-\tau^p_{hk}\}\Bigr]_+,
\end{aligned}
\label{eq:pool-cov}
\end{equation}
for $h,k\in I_{h'}$, where $\ell^{p,h'}_{hk}$ is the overlap, in days, of the two windows once the window of hub $k$ is shifted by the travel time. The pooled term of~(8) omits these nonnegative covariances and therefore understates the pooled variance whenever two hubs of a pool share a path. The independence approximation is therefore not conservative by construction; Appendix~B of the paper quantifies its effect for the SE-5 plan.

\subsection{Fleet-sizing constraints and fleet subproblem}
\label{app:supp-fleet}
For the scenario-active paths $p\in\mathcal P_t(\omega)$, let $o_{pht}(\omega)$ denote the incremental busy-truck requirement of path $p$ at hub $h$ and $n^v_{t,p}(\omega)$ the incremental fleet provisioned for path $p$. The fleet requirement~(11) is represented by the linear constraints
\begin{subequations}
\label{eq:Truck-all2}
\begin{alignat}{2}
&\tilde r_{pht}(\omega)\ge0,
&\quad&\forall p\in\mathcal P_t(\omega),\ \forall h\in\mathcal H_p,\ \forall t\in\mathcal T,
\label{eq:truck-1}\\[0.2em]
&\tilde r_{pht}(\omega)\ge
T^sx_{ht}^s+T^{c,\mathrm{fast}}(1-x_{ht}^s)R_{ph}(\omega)-T_{hp}^r,
&&\forall p\in\mathcal P_t(\omega),\ \forall h\in\mathcal H_p,\ \forall t\in\mathcal T,
\label{eq:truck-1b}\\[0.2em]
&o_{pht}(\omega)\ge
\Bigl(\alpha_{pt}(\omega)+\Phi^{-1}(1-\varepsilon^C)\beta_{pt}(\omega)\Bigr)
\tilde r_{pht}(\omega),
&&\forall p\in\mathcal P_t(\omega),\ \forall h\in\mathcal H_p,\ \forall t\in\mathcal T,
\label{eq:truck-2}\\[0.2em]
&n_{t,p}^v(\omega)\ge
\frac{\sum_{h\in\mathcal H_p}o_{pht}(\omega)}{O_c},
&&\forall p\in\mathcal P_t(\omega),\ \forall t\in\mathcal T.
\label{eq:truck-3}
\end{alignat}
\end{subequations}
Constraints~\eqref{eq:truck-1}--\eqref{eq:truck-1b} are the epigraph form of~(10);~\eqref{eq:truck-2} converts the safety-adjusted offered load (trucks per day) times the incremental dwell (days) into busy trucks; and~\eqref{eq:truck-3} converts busy trucks into provisioned trucks, priced at the annual cost $M^v_t$ of a driver--truck unit. Because the fleet cost is minimized and these variables are bounded only from below, the constraints bind at optimality and reproduce~(11).

Accordingly, the fleet recourse of path $p$ in Section~4 of the paper is the linear program
\begin{subequations}
\label{eq:GtpV}
\begin{alignat}{2}
G_{t,p}^{\mathrm V}(x^s;\omega)
:=\;&
\min_{n_{t,p}^v(\omega),\,\tilde r_{p\cdot t}(\omega),\,o_{p\cdot t}(\omega)}
\ \delta_tM_t^v n_{t,p}^v(\omega)
&\label{eq:GtpV-obj}\\[1mm]
\text{s.t.}\quad
&\eqref{eq:truck-1}\text{--}\eqref{eq:truck-3},
&\notag\\
&n_{t,p}^v(\omega)\ge0,
\quad \tilde r_{pht}(\omega),o_{pht}(\omega)\ge0,
&\quad \forall h\in\mathcal H_p.
\label{eq:GtpV-nonneg}
\end{alignat}
\end{subequations}
where $\tilde r_{p\cdot t}(\omega)$ and $o_{p\cdot t}(\omega)$ collect the incremental-dwell and busy-truck variables of path $p$ at its hubs.

\end{document}